\documentclass{article}
\usepackage{amsmath}
\usepackage{graphicx}
\usepackage{enumerate}
\usepackage{enumitem}
\usepackage{natbib}
\usepackage{url} 

\usepackage[utf8]{inputenc}
\usepackage{amsfonts,epsfig}

\usepackage{amsthm}
\usepackage{microtype}
\usepackage{mathtools}
\usepackage{tabularx}
\usepackage{pifont} 
\usepackage{xcolor}
\definecolor{darkblue}{rgb}{0,0.08,0.45}
\definecolor{darkred}{RGB}{139,0,0}
\definecolor{Darkblue}{RGB}{0,0,139}
\definecolor{forestgreen}{RGB}{34,139,34}
\definecolor{darkgreen}{RGB}{0,100,0}
\usepackage{tikz}
\usepackage{bbm}
\usepackage{algpseudocode}
\usepackage{algorithm}
\usepackage{stackengine}
\usetikzlibrary{shapes,arrows,chains}
\usetikzlibrary[calc]
\usetikzlibrary{bayesnet}
\tikzset{>=latex}

\usepackage{hyperref}
\hypersetup{
	colorlinks=true,
	linkcolor=black,
	citecolor=black,
	filecolor=black,
	urlcolor=black,
}
\usepackage{xspace}
\usepackage{mathtools}
\usepackage{authblk}
\usepackage{makecell}

\DeclarePairedDelimiterX{\kldivx}[2]{(}{)}{%
  #1\;\delimsize\|\;#2%
}
\newcommand{\kldiv}{D\kldivx}

\DeclarePairedDelimiterX{\kldivxbig}[2]{\Big(}{\Big)}{%
  #1\;\Big\|\;#2%
}
\newcommand{\kldivbig}{D\kldivxbig}

\theoremstyle{plain}
\newtheorem{theorem}{Theorem}

\theoremstyle{plain}
\newtheorem{itheorem}{Informal Theorem}

\theoremstyle{plain}
\newtheorem{lemma}{Lemma}

\theoremstyle{definition}

\renewcommand{\P}{\mathbb{P}}

\newcommand{\wv}{\mathbf{w}}

\newcommand{\Am}{\mathbf{A}}

\newcommand{\RR}{\mathbb{R}}

\newcommand{\uv}{\mathbf{u}}
\newcommand{\vv}{\mathbf{v}}

\newcommand{\covm}{\boldsymbol{\Sigma}}

\renewcommand{\b}{\mathrm{base}}

\newcommand{\yv}{Y}
\newcommand{\yvc}{y}
\newcommand{\yrep}{Y^{\mathrm{rep}}}
\newcommand{\yrepc}{y^{\mathrm{rep}}}
\newcommand{\thetav}{\Theta}
\newcommand{\thetavc}{\theta}
\newcommand{\lambdav}{\Lambda}
\newcommand{\lambdavc}{\lambda}
\newcommand{\p}{f}
\newcommand{\palt}{\p_{\mathrm{alt}}}
\newcommand{\paux}{\widetilde{\p}}
\newcommand{\Xm}{\mathbb{X}}
\newcommand{\xv}{x}
\newcommand{\imat}{\mathbb{I}}
\newcommand{\MI}{I}
\newcommand{\absparam}{\Phi}
\newcommand{\absparamc}{\phi}
\newcommand{\E}{\mathbb{E}}
\newcommand{\subtheta}{\absparam}
\newcommand{\subthetac}{\absparamc}
\newcommand{\yval}{\yv^{\mathrm{val}}}

\newcommand{\imi}{IMI\xspace}
\newcommand{\fmi}{FMI\xspace}
\newcommand{\psd}{PSD\xspace}
\newcommand{\psdm}{\mathrm{PSD}}
\newcommand{\imim}{\MI^{id}}
\newcommand{\fmim}{\MI^{fa}}

\newcommand{\ddelfa}{\Delta^{fa}}
\newcommand{\daddfa}{\Sigma^{fa}}
\newcommand{\ddelid}{\Delta^{id}}
\newcommand{\daddid}{\Sigma^{id}}
\newcommand{\pvp}{\mathrm{PVP}}
\newcommand{\expratio}{R\left(\p_{\b}, \p\right)}
\newcommand{\ind}{\mathrel{\perp\!\!\!\!\!\:\perp}}
\newcommand{\dspace}{\RR^n}
\newcommand{\tspace}{\RR^d}

\definecolor{sigcolor}{RGB}{204, 143, 0}
\definecolor{delcolor}{RGB}{216, 27, 27}

\theoremstyle{plain}

\theoremstyle{plain}
\newtheorem{definition}{Definition}

\theoremstyle{plain}

\title{Identifiability and Falsifiability: Two Challenges for Bayesian Model Expansion}
\author[1]{Collin Cademartori}
\affil[1]{Department of Statistical Sciences, Wake Forest University}

\begin{document}

\maketitle

\begin{abstract}
  We study the identifiability of parameters and falsifiability of predictions under the process of model expansion in a Bayesian setting. Identifiability is represented by the closeness of the posterior to the prior distribution and falsifiability by the power of posterior predictive tests against alternatives. To study these two concepts formally, we develop information-theoretic proxies, which we term the identifiability and falsifiability mutual information. We argue that these are useful indicators, with lower values indicating a risk of poor parameter inference and underpowered model checks, respectively. Our main result establishes that a sufficiently complex expansion of a base statistical model forces a trade-off between these two mutual information quantities -- at least one of the two must decrease relative to the base model. We illustrate our result in three worked examples and extract implications for model expansion in practice. In particular, we show as an implication of our result that the negative impacts of model expansion can be limited by offsetting complexity in the likelihood with sufficiently constraining prior distributions.
\end{abstract}

\section{Introduction}\label{sec:intro}

Model expansion - the process of passing from a simpler base model to a larger, more flexible model - is a common component of statistical workflow \citep{BayesWorkflow1, BayesWorkflow2,BoxsLoop}. Such expansions are often motivated by a desire to make model assumptions more plausible or the quantification of uncertainty more realistic \citep{GreenlandId,DraperExpansion}. It is well-known that the complex models which arise from model expansion can pose challenges for inference and interpretation \citep{PCPriors}. In particular, larger models can easily yield uncertain inferences for individual parameters and power-deficient tests for goodness-of-fit. We refer to these two difficulties as the identifiability and falsifiability challenges respectively.

Yet, it is also evident that these challenges are not inevitable consequences of model expansion. \citet{GustafsonId1} demonstrate that model expansion can improve parameter identification in some cases, for example. Likewise, the negative effects of expansion on the power of fitness tests may be offset if we are able to derive a better test quantity, e.g. one which is pivotal \citep{PivotalChecks, GelmanExamples}.
Thus, while there is an apparent tension between model expansion, identifiability, and falsifiability, the dynamics of this tension are often unclear. This paper aims to clarify these dynamics by answering two key questions:

\begin{enumerate}
\item[] \textbf{(Question 1) To what extent are these two challenges avoidable?} Specifically, is there a limit to the complexity of the expanded model beyond which we can no longer avoid diminishing identifiability and falsifiability?
\item[] \textbf{(Question 2)  How can we best address these two challenges in practice?} Specifically, are there generally applicable strategies which can be implemented to limit the negative effects of model expansion on identifiability and falsifiability? 
\end{enumerate}

We answer these questions by first formalizing identifiability and falsifiability in terms of the information-theoretic concept of mutual information between observed data and model parameters. Using these quantities, our main result answers Question 1 in the negative. In particular, Theorem \ref{thm:tradeoff} establishes that, when an expanded model is sufficiently more complex than a base model (in an appropriately defined sense), at least one of identifiability and falsifiability must decrease from base model to expanded model. In the process of developing this result, we provide a partial answer to Question 2 by isolating conditions on the prior distribution which can limit this trade-off.

\subsection{Model Expansion}

 We start by defining the types of model expansions to which our results apply. We write $\p_{\b}(\yvc,\thetavc)$ for the density function of some base model with prior $\p_{\b}(\thetavc)$ defined over parameters $\thetav\in\RR^d$ and likelihood $\p_{\b}(\yvc\mid\thetavc)$ defined for data $\yv\in\RR^n$. Expansions of this base model are defined as follows.

\begin{definition}[Model Expansion]\label{def:model_expansion}
  Let $\p_{\b}(\yvc,\thetavc)$ be a base model density as above. For the same data $\yv$, let $\p$ be the density function of an additional model with likelihood $\p(\yvc\mid\thetavc,\lambdavc)$ and prior $\p(\thetavc,\lambdavc)$ defined over parameters $\thetav\in\RR^d$ and $\lambdav\in\RR^k$. Then $\p$ is an expansion of $\p_{\b}$ if
\begin{equation}
  \label{eq:expansion_def}
  \p_{\b}(\yvc,\thetavc) = \p(\yvc,\thetavc\mid\lambdavc = \lambdavc_0) \text{ for a fixed } \lambdavc_0\in[-\infty,\infty]^k.
\end{equation}
If $[\lambdavc_0]_j$, the $j^{\mathrm{th}}$ component of $\lambdavc_0$, is $\pm\infty$ for some $1\leq j\leq k$, then $\p(\yvc,\thetavc\mid\lambdavc = \lambdavc_0)$ in \eqref{eq:expansion_def} is understood as the density of the distributional limit of $\p(\yvc,\thetavc\mid\lambdavc)$ as $\lambdavc\to\lambdavc_0$ (when the limit and density exist).
\end{definition}
This framework includes many common examples of model expansion:

\begin{itemize}
\item Let $\p_{\b}(\yvc,\thetavc)$ be a generalized linear model with response vector $\yv$ and parameters $\thetav$. Adding a new predictor and coefficient $\lambdav$ with independent prior is then an expansion since $\p_{\b}(\yvc,\thetavc)=p(\yvc,\thetavc\mid\lambdavc=0)$.
\item Let $\p_{\b}(\yvc,\thetavc)$ be an exchangeable Poisson model over data $[\yv]_i$ with rate $\thetav$. Consider extending this to a negative binomial model with overdispersion parameter $\lambdav$ (with independent prior):
  \[
    \p(y\mid\theta,\lambda) = \binom{y + \lambda - 1}{y}\left( \frac{\theta}{\theta+\lambda} \right)^y\left( \frac{\lambda}{\theta+\lambda} \right)^\lambda.
  \]
  For all $(\yvc,\thetavc)$, we have $\p_{\b}(\yvc,\thetavc)=\lim_{\lambdavc\to\infty}\p(\yvc,\thetavc\mid\lambdavc)$, and thus $\p(\cdot\mid\lambdavc)$ converges in distribution to $\p_{\b}$ as $\lambdavc\to \infty$ by Scheffe's Theorem, so this is again a model expansion by our definition.
\end{itemize}

\subsection{Identifiability and Falsifiability}\label{subsec:id_fa_informal}

We now describe the statistical concepts of identifiability and falsifiability informally, before providing formal information-theoretic definitions in Section \ref{sec:id_and_fa}.

\subsubsection{Identifiability}

Identifiability refers to our ability to use observed data to gain information about unobserved parameters. In frequentist inference, identification is usually defined as a binary property of a parametric family $\{\p(\yvc\mid \thetavc)\}_{\thetavc\in\tspace}$, where the family is identified when $\thetavc_1 \neq \thetavc_2$ implies $\p(\cdot\mid \thetavc_1) \neq \p(\cdot \mid \thetavc_2)$. When this property fails, the model is nonidentified, in which case the maximum likelihood estimator typically fails to be well-defined.

For Bayesian models, estimation is still possible with nonidentified likelihood families, as the posterior distribution is well-defined whenever the prior is a proper probability distribution. However, nonidentification can still undermine the usefulness of posterior inference. For instance, the overparametrized location model $\yvc \sim \mathrm{normal}\left(\thetav_1 + \thetav_2, 1\right)$ is nonidentified in the frequentist sense. In the Bayesian setting, if $\thetav_1$ and $\thetav_2$ are assigned i.i.d. normal priors, then $\p(\thetavc_2-\thetavc_1 \mid \yvc) = \p(\thetavc_2-\thetavc_1)$, i.e. the marginal posterior reduces to the prior for $\thetav_2-\thetav_1$. In other words, we learn nothing about this difference.

On the other hand, for parameters $\absparam$ (possibly equal to the full parameter vector $\thetav$), it is possible to learn \textit{nearly} nothing about $\absparam$ and have $\p(\absparamc\mid \yvc)\approx \p(\absparamc)$ even if the model is identified in the frequentist sense. We will refer to such a parameter $\absparam$ as being weakly identified in the model $\p(\thetavc, \yvc)$ if $\p(\absparamc\mid \yvc)$ is sufficient close to $\p(\absparamc)$. We formally quantify the weakness of identification using constructs from information theory in Section \ref{sec:id_and_fa}.

\subsubsection{Falsifiability}

Falsifiability refers to our ability to detect deficiencies in model fitness. In the context of goodness-of-fit testing, falsifiability is closely related to power. In Bayesian model checking, for a test statistic $T$, it is common to assess model fitness by comparing the observed value of $T(\yv)$ to values $T(\yrep)$ which might be observed in a replicated dataset $\yrep$ (i.e. an independent dataset drawn from the same distribution as $\yv$). Such comparisons can be made quantitative by computing the posterior predictive $p$-value:
\begin{equation}
  \label{eq:pppv_def}
  p_T(\yvc) = \P\left(|T(\yrep)| \geq |T(\yvc)| \mid \yv = \yvc\right),
\end{equation}
where $\yrep$ is sampled from the posterior predictive distribution, which is given as $\p(\yrepc\mid \yvc) = \E\left[ \p(\yrepc\mid \thetav) \mid \yv = \yvc \right]$. We can then construct a test of the model by comparing $p_T(\yv)$ to some significance threshold $\alpha$ and rejecting the model if $p_T(\yv) < \alpha$.

The power of such a test depends on the proximity of our proposed model to the true model, the test statistic $T$, and the rejection threshold $\alpha$. We set aside the question of proximity between the true and proposed model, since we can never directly control this in practice. We also take the test statistic $T$ to be given (though our later analysis will not depend on a choice of statistic). For now, we focus on the choice of threshold $\alpha$.

Under the null hypothesis that the model is correct (i.e. that $\yv\sim \p(\yvc)$), \citet{PostP} showed that $p_T(\yv)$ is typically more concentrated around $1/2$ than a uniform random variable. As a result, the level of this test is typically below the rejection threshold, i.e. $\P\left(p_T(\yv) < \alpha\right) < \alpha$. We can convert $p_T(\yv)$ to a uniformly distributed, or calibrated, $p$-value by plugging it into its cumulative distribution function:
\begin{equation}
  \label{eq:post_pval_cal}
  p^{\mathrm{cal}}_T(\yvc) = \P_{\yv\sim\p(\yvc)}\left(p_T(\yv) \leq p_T(\yvc)\right).
\end{equation}
Compared to a test using $p_T^{\mathrm{cal}}$ and the same threshold $\alpha$, a test using $p_T$ will usually have lower power against alternative models. Theoretically, the added power of the calibrated $p$-value is ``free'' insofar as it requires no modification of the model or statistic. In practice, computing $p_T^{\mathrm{cal}}$ is orders of magnitude more expensive than fitting the model once, and is thus computationally intractable in all but the simplest models. 

This power deficit of the uncalibrated test can be explained by a kind of overfitting -- a result of using the data $\yv$ to both construct both the posterior $\p(\thetavc\mid\yv)$ and the test statistic $T(\yv)$. In particular, if we could observe a separate validation dataset $\yv^{\mathrm{val}}\sim \p(\yvc\mid\thetavc)$ independent of $\yv$, then the posterior predictive test using $T(\yv^{\mathrm{val}})$ would be calibrated (as is easily seen using the probability integral transform).

We can quantify the gap in power between $p_T$ and $p_T^{\mathrm{cal}}$ in terms of the behavior of conditional $p$-values. For any fixed $\thetavc$, define the conditional $p$-value
\begin{equation}
  \label{eq:cond_pval_intro}
  p_T(\yvc\mid \thetavc) = \P\left( |T(\yrep)| \geq |T(\yvc)|\mid \thetav = \thetavc \right),
\end{equation}
where $\yrep\sim \p(\yvc\mid\thetavc)$. It is easy to see that $p_T(\yvc) = \E\left[ p_T(\yvc\mid\thetav)\mid \yv = \yvc \right]$. It also follows from the proof of the nonuniformity of $p_T(\yv)$ in \cite{PostP} that the degree of nonuniformity of $p_T(\yv)$ grows with the posterior variance of conditional $p$-values (PVP):
\begin{equation}
  \label{eq:pval_var_intro}
  \pvp_T(\yvc) = \mathrm{Var} \left[ p_T(\yvc\mid\thetav) \mid \yv = \yvc\right].
\end{equation}
We interpret this variance as quantifying posterior uncertainty about the fitness of the unknown data-generating distribution $\p(\yvc\mid\thetav)$ (as measured by the statistic $T$).

In summary, falsifiability is related to the power of posterior predictive tests. The power of these tests relative to their calibrated counterparts is controlled by the variance \eqref{eq:pval_var_intro}.
In Section \ref{sec:id_and_fa}, we measure falsifiability with a quantity that can be viewed as generalizing the variance \eqref{eq:pval_var_intro} and which is independent of any particular test statistic $T$.

\subsection{Related Work}\label{subsec:related}
 Recently, statistical workflow has enjoyed increased attention as a discrete topic. This literature has developed a consistent framework and practical advice for each step of statistical analysis, including model expansion (see, e.g. \citet{BayesWorkflow1,BayesWorkflow2,BayesWorkflow3}).
 We seek to complement this work by studying model expansion as a distinct regime.

\citet{GreenlandId} defines a notion of model expansion which is similar to our Definition \ref{def:model_expansion} and studies nonidentification in the expanded model. Whereas Greenland considers strictly nonidentified cases, this work studies model expansion as a process which tends to weaken identifiability continuously. Our results reinforce Greenland’s conclusion that, in the presence of weak identification, “any analysis should thus be viewed as a part of a sensitivity analysis which depends on external plausibility considerations.”

Gustafson studies the asymptotics of posterior distributions for strictly nonidentified likelihood models, showing that the posterior may be substantially more informative than the prior in the infinite-data limit \citep{GustafsonId1,GustafsonId2}. Specifically, \citet{GustafsonId2} shows this can occur when the prior encodes dependence between identified and non-identified parameters. Our main result echoes this conclusion, showing that prior dependence between parameters can improve identification also in the preasymptotic regime.
Other methods for detecting and dealing with identification problems have been studied in, e.g. \citet{BayesLearning,DataCloning}.

 As we argued above, problems of falsifiability are directly connected to debates over the conservativity and power of the posterior predictive $p$-value. Various forms of this problem have been described, and possible solutions have been proposed in \citet{BayarriBerger99,BayarriBerger00,RobinsEtAl00,PostPPharma,SampledPValues,ChiSquaredGOF,PivotalSampled,PostProcessPValues}. This work complements these arguments by relating the degree of conservativity in posterior predictive checks to model complexity.

 Our approach to the problems of identifiability and falsifiability follows many previous successes in using information-theoretic tools to study the properties of statistical models.
\begin{enumerate}
 \item We quantify uncertainty and information with the (differential) entropy and mutual information respectively. \cite{MaxEntropyPriors} used the representation of uncertainty as entropy to argue for the use of maximum entropy priors. Likewise, \cite{LindleyMeasure} pioneered the application of mutual information to the problem of designing experiments for optimal information gain.
 \item The mutual information between data and model parameters has been extensively studied as an optimization target for problems in Bayesian inference. Reference priors are defined asymptotically by maximizing this mutual information under successive sampling of the data generating process \citep{RefPrior,RefPriorProduct,RefPriorPartial,ITAsymp,RefPriorAsymp}. This maximality property justifies viewing reference priors as containing minimal prior information about the parameter of interest. On the other hand, optimizations of the mutual information with respect to the likelihood function have proven useful in Bayesian experimental design and Bayesian optimization \citep{BayesOpt,BayesOED1,BayesOED2}.
 \item \citet{PCPriors} propose a method of prior specification that penalizes deviation from a base model which closely reflects the notion of model expansion used in this work. These penalized complexity priors impose a joint structure on the parameters by constructing a density which decays with an information-theoretic measure of model complexity. Similarly, \citet{R2D2M2} introduce a joint prior on multilevel regression parameters which is explicitly designed to scale in a controllable and interpretable manner as the number of regressors increases. More generally, \citet{BayesWorkflow1} emphasize the need to ``to think in terms of the joint prior over all the parameters in a model''. Our conclusions about prior specification in Section \ref{sec:thm1_formal} mirror and reinforce these ideas.
 \item The Rashomon effect, defined by \cite{TwoCultures}, occurs when many models achieve similar overall loss but provide very different predictions. We demonstrate that falsifiability is related to the multiplicity of plausible predictive distributions. Our measure of falsifiability also rests on a similar KL divergence to the Rashomon capacity, a metric for quantifying the Rashomon effect \citep{RashomonCapacity}.
 \end{enumerate}

 \subsection{Outline of Paper}
 The remainder of this paper is organized as follows. In Section \ref{sec:reg_ex}, we present a simple example to build intuition for the effect of model expansion on identifiability and falsifiability using familiar statistical quantities. Section \ref{sec:id_and_fa} begins by presenting our information-theoretic proxies for the concepts of identifiability and falsifiability, and concludes with the statement of our main result, which establishes a trade-off between these quantities under model expansion. We examine the implications of this result in Section \ref{sec:id_fa_expand} with three worked examples, and Section \ref{sec:conclusion} presents concluding remarks.

 \subsection{Notation}

 Random variables and vectors are denoted by upper-case letters, with Roman letters (e.g. $\yv$) for data and Greek letters (e.g. $\thetav$) for parameters. Corresponding densities are denoted with lower-case arguments (e.g. $f(\yvc,\thetavc)$ for the joint density of $(\yv,\thetav)$). Some expressions will require plugging random variables into their own densities. In this case, we write, e.g., $\p(\yv)$ for $\yv$ plugged into its density $\p(\yvc)$. 
 The full parameter vector of a model is written either $\thetav$ (for a base model) or $(\thetav,\lambdav)$ (for an expanded model). At times, we will wish to refer to an unspecified subset of these parameter vectors, which we denote by $\absparam$.
 Matrices are written in blackboard face (e.g. $\mathbb{X}$), with $\mathbb{I}_d$ denoting the $d\times d$ identity matrix. The 2-norm of vector $v$ is $\|v\|$; the $i^{\mathrm{th}}$ component is $[v]_i$; and orthogonality is denoted by $\perp$. Independence of random variables is denoted by $\ind$.
 
\section{A Toy Regression Example}\label{sec:reg_ex}
We now illustrate the connection between identifiability, falsifiability, and model expansion in a simple regression example. 
Consider a linear regression base model $\p_{\b}$ with three observations $\yv\in\RR^3$, two predictors $[x_1\;x_2]=\Xm_{\b}\in\RR^{3\times 2}$, and known noise variance equal to $1$. Using unit normal priors, $\p_{\b}$ becomes
\begin{equation}
  \label{eq:reg_ex_base}
  \yv\mid\thetav \sim \mathrm{normal}\left(\Xm_{\b}\thetav, \imat_3\right),\quad \thetav\sim\mathrm{normal}\left(0, \imat_2\right).
\end{equation}
We construct $x_1$ and $x_2$ to be linearly independent with unit norm, and with interaction $\xv_{\mathrm{int}} = ([\xv_1]_1[\xv_2]_1, [\xv_1]_2[\xv_2]_2)^T$ linearly independent of $\xv_1$ and $\xv_2$. 

We form an expansion $\p(\yvc, \thetavc, \lambdavc)$ of \eqref{eq:reg_ex_base} by adding a third predictor $\xv_3$ with coefficient $\lambdav$. Let $\Xm = [\Xm_{\b}\; \xv_3]$ be the expanded predictor matrix, so that the expansion $\p$ becomes
\begin{equation}
  \label{eq:reg_ex_exp}
  \yv\mid\thetav,\lambdav \sim\mathrm{normal}\left( \Xm
  \left[\begin{smallmatrix}
    \thetav\\\lambdav
  \end{smallmatrix}\right],
   \imat_3\right),\quad \left[\begin{smallmatrix}
    \thetav\\\lambdav
  \end{smallmatrix}\right]\sim\mathrm{normal}\left(0, \imat_3\right).
\end{equation}

The coefficients $\thetav$ are shared by the base and expanded models. We quantify the identification of $\thetav$ using the posterior standard deviations $\sqrt{\mathrm{Var}\left([\thetav]_i\mid \yv=\yvc\right)}$ for $i = 1,2$. Identification is then compared between models using the worst-case ratio of these:
\begin{equation}
  \label{eq:max_std_ratio}
  \mathrm{SR} \stackrel{\mathrm{def}}{=} \min_{i \in \{1,2\}}\frac{\sqrt{\mathrm{Var}_{\p_{\b}}\left([\thetav]_i\mid \yv=\yvc\right)}}{\sqrt{\mathrm{Var}_{\p}\left([\thetav]_i\mid \yv=\yvc\right)}}.
\end{equation}
Smaller values of \eqref{eq:max_std_ratio} indicate worse identification of $\thetav$ in the expanded model.

We study falsifiability by fixing an alternative model $\p_{\mathrm{alt}}(\yvc)$ with nonzero interaction:
\begin{equation}
  \label{eq:reg_ex_alt}
  \yv \sim \mathrm{normal}\left(\xv_1 + \xv_2 + 2\xv_{\mathrm{int}}, \imat_3\right).
\end{equation}
To test against this model, we choose test statistic $T(\yvc) = \xv_{\mathrm{int}}^T\yvc$ and consider the tests that reject when $p_T(\yv) < 0.1$ and $p_T^{\mathrm{cal}}(\yv) < 0.1$ respectively, where $p_T$ is the posterior predictive $p$-value \eqref{eq:pppv_def} and $p_T^{\mathrm{cal}}$ is the calibrated $p$-value \eqref{eq:post_pval_cal}. We then define the power under our two tests as
\begin{equation}
  \label{eq:reg_ex_power_def}
  \mathrm{Pow}^{\mathrm{post}}\left(T,\alpha\right) = \P_{\p_{\mathrm{alt}}}\Big(p_T(\yv) < \alpha\Big),\quad \mathrm{Pow}^{\mathrm{cal}}\left(T,\alpha\right) = \P_{\p_{\mathrm{alt}}}\Big(p^{\mathrm{cal}}_T(\yv) < \alpha\Big).
\end{equation}
In Section \ref{subsec:id_fa_informal}, we noted that a major threat to falsifiability is the power deficit of the uncalibrated test relative to the calibrated test. We quantify falsifiability in this example using the relative power of the uncalibrated test,  $\mathrm{Pow}^{\mathrm{post}}\left(T,\alpha\right) / \mathrm{Pow}^{\mathrm{cal}}\left(T,\alpha\right)$. We quantify  the change in falsifiability from base to expanded model with the ratio:
\begin{equation}
  \label{eq:reg_ex_power}
  \mathrm{PR} = \frac{\mathrm{Pow}^{\mathrm{post}}_{\p}\left(T,\alpha\right) / \mathrm{Pow}^{\mathrm{cal}}_{\p}\left(T,\alpha\right)}{\mathrm{Pow}^{\mathrm{post}}_{\p_{\b}}\left(T,\alpha\right) / \mathrm{Pow}^{\mathrm{cal}}_{\p_{\b}}\left(T,\alpha\right)}.
\end{equation}
Smaller values of this ratio correspond to expanded models which have worse falsifiability (using test statistic $\xv_{\mathrm{int}}^T\yvc$) relative to the base model.

To study the effect of different choices of $\xv_3$ on SR and PR, we sample $\xv_3$ uniformly from the sphere $S = \{\xv\mid \|\xv\| = 1\}$. We then compare SR and PR to the quantity:

\begin{equation}
  \label{eq:reg_ex_proj}
  \pi(\xv_3) = \frac{\xv_3^T \left[\imat_2 - \Xm_{\b} \left(\Xm_{\b}^T\Xm_{\b}\right)^{-1}\Xm_{\b}^T\right]\xv_3}{\|\xv_3\|^2}.
\end{equation}
The numerator of \eqref{eq:reg_ex_proj} is the squared norm of the projection of $\xv_3$ onto the orthogonal complement of the column space of $\Xm_{\b}$. Consequently, we have $\pi\in [0,1]$, with $\pi(\xv_3) = 1$ if $\xv_3\perp \xv_1,\xv_2$ (in which case $\xv_3$ is collinear with $\xv_{\mathrm{int}}$).
Figure \ref{fig:reg_ex} plots the values of $\mathrm{SR}$ and $\mathrm{PR}$ against $\pi$ when $\xv_3\sim\mathrm{uniform}(S)$. For almost all $\xv_3$, we have $\mathrm{SR} < 1$ and $\mathrm{PR} < 1$, indicating worse identifiability and falsifiability relative to the base model. We also observe a trade-off: SR tends to increase with $\pi$ whereas PR tends to decrease. 

\begin{figure}[t]
  \centering
  \includegraphics[scale=0.6]{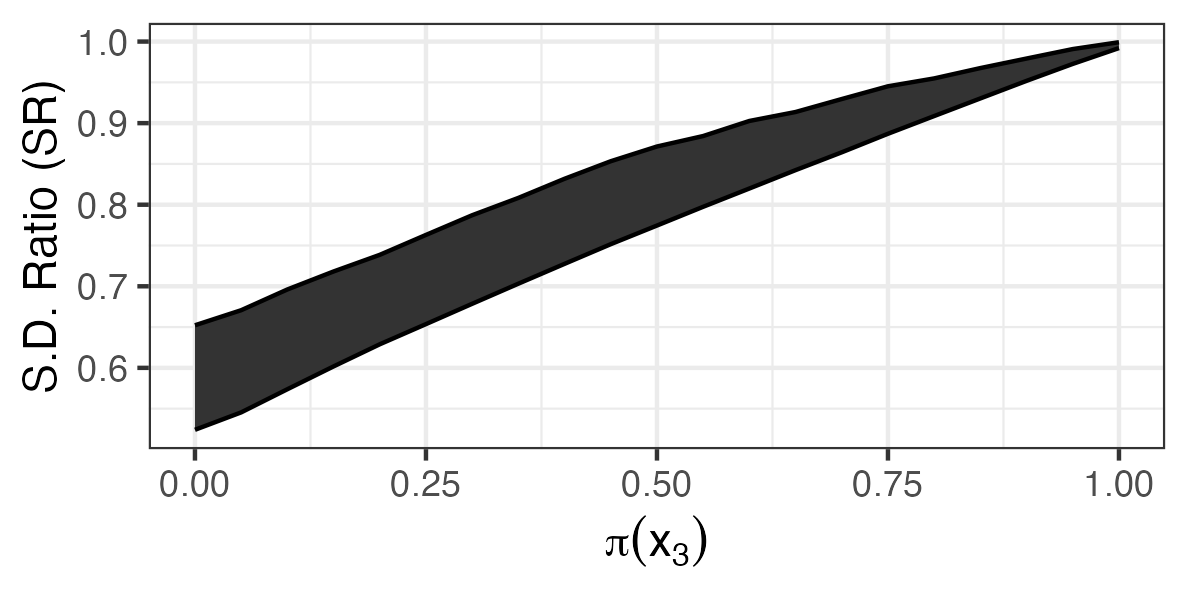}
  \includegraphics[scale=0.6]{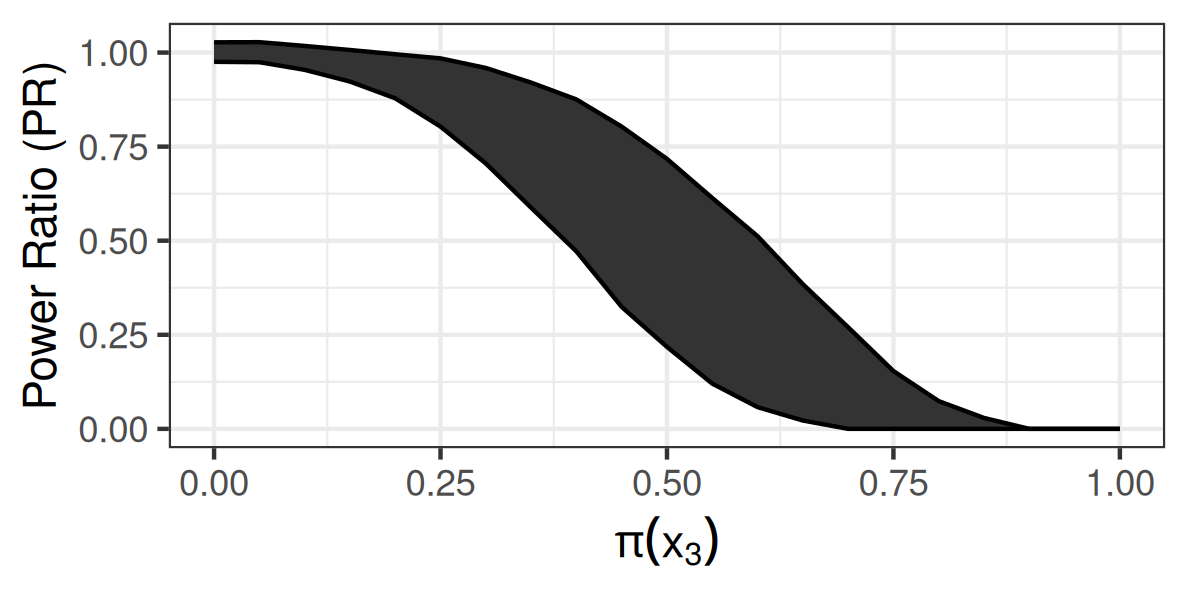}
  \caption{80\% quantile bands for the standard deviation ratio SR (left) and the power deficit ratio PR (right) against $\pi$. Identification increases with $\pi$, nearly matching the base model at $\pi =1$. Falsifiability falls with $\pi$, nearly matching the base model at $\pi=1$.}
  \label{fig:reg_ex}
\end{figure}

This trade-off can be explained as follows. We noted in Section \ref{subsec:id_fa_informal} that the power (relative to a calibrated test) falls as the posterior variance of the conditional $p$-values (PVP) increases. For our expanded model, we can write the conditional $p$-values as 
\begin{align}
  p_T(\yvc\mid \thetavc,\lambdavc)
  &= 1 - F\Big(\xv_{\mathrm{int}}^T\yvc\;\Big\vert\; [\thetavc\;\lambdavc]\Xm^T\xv_{\mathrm{int}}, \|\xv_{\mathrm{int}}\|^2\Big),\label{eq:reg_ex_cond_pval}
\end{align}
where $F(\cdot\mid \mu, \sigma^2)$ is the normal cumulative distribution function with mean $\mu$ and variance $\sigma^2$. Recall that $\Xm_{\b}$ was constructed so that $\Xm_{\b}^T\xv_{\mathrm{int}} = 0$. If $\pi=0$, then $\xv_3$ is collinear with $\xv_1$ and $\xv_2$, and the means $[\thetavc\;\lambdavc]\Xm^T\xv_{\mathrm{int}}$ are all identically zero. Thus, when $\pi =0$, the posterior variance of \eqref{eq:reg_ex_cond_pval} is $0$, the best case for falsifiability. However, when $\pi=0$, we also have that $\Xm$ is singular, the worst-case scenario for identifiability.

\section{The Trade-Off Between Identifiability and Falsifiability}\label{sec:id_and_fa}

We now turn to the relationship between identifiability, falsifiability, and model expansion in general. In order to establish our formal results, we need four mathematical quantities:
\begin{enumerate}
  \item For a base model $\p_{\b}$ and expansion $\p$, an overall \textbf{identifiability measure}, for which larger values are associated with better identifiability.
  \item For a base model $\p_{\b}$ and expansion $\p$, an overall \textbf{falsifiability measure}, for which larger values are associated with better falsifiability.
  \item For base and expanded models $\p_{\b}$ and $\p$, an \textbf{expansion measure} that quantifies how much complexity is added by the expansion parameters $\lambdav$. 
  \item For the expanded model $\p$, an \textbf{expansion threshold} which quantifies how much expansion is sufficient to force a trade-off between identifiability and falsifiability.
\end{enumerate}

Once we have established the definitions of these quantities, the statement (and proof) of the main result is straightforward. By contrast, it takes some effort to fully motivate the definitions and connect them to familiar statistical quantities. For this reason, we state Theorem \ref{thm:tradeoff} informally now and then ``fill in'' the necessary definitions, providing a fully formal restatement of the theorem in Section \ref{sec:thm1_formal}.

\begin{itheorem}[Identifiability - Falsifiability Trade-off]\label{thm:tradeoff_informal}
  Consider a base model and some expansion. If we have $\text{\textup{\textbf{expansion threshold}}} < 1$ and
  \begin{equation*}
    \label{eq:tradeoff_condition_informal}
    \text{\textup{\textbf{expansion measure}}} \geq \text{\textup{\textbf{expansion threshold}}},
  \end{equation*}
  then we have either that
  \begin{align*}
      \text{\textup{expanded model}} & \text{\textup{\textbf{ identifiability measure}}}\\
      \leq \text{\textup{base model}}& \text{\textup{\textbf{ identifiability measure}}},
  \end{align*}
  or that
  \begin{align*}
      \text{\textup{expanded model}} & \text{\textup{\textbf{ falsifiability measure}}}\\
      \leq \text{\textup{base model}}& \text{\textup{\textbf{ falsifiability measure}}},
  \end{align*}
\end{itheorem}

Theorem \ref{thm:tradeoff_informal} states that, if we expand a base model enough, either identifiability or falsifiability must decrease. This conclusion is less dire than what we observed in the regression example of Section \ref{sec:reg_ex}, where \textit{both} identifiability and falsifiability decreased in every expansion \eqref{eq:reg_ex_exp} (by our chosen metrics). Despite a leap in generality, Theorem \ref{thm:tradeoff_informal} does retain two key features of that example: (i) an inverse association between expansion and identifiability/falsifiability, and (ii) a trade-off between identifiability and falsifiability among expansions of the base model.

\subsection{Information-Theoretic Background}
We give a brief overview of concepts and measures from information theory which will be needed throughout. The reader may consult the Supplementary Materials for further details and background \citep{SuppMat}.

For a joint model $\p(\yvc, \thetavc)$, the (differential) entropy of $\thetav$ is defined as
\begin{equation}
  \label{eq:entropy_def}
  h_{\p(\thetavc)}(\thetav) = \E_{\thetav\sim\p(\thetavc)}\Big[-\log \p(\thetav)\Big].
\end{equation}
The function $-\log \p(\thetav)$ quantifies how atypical a random sample $\thetav$ is of its distribution $\p$, on average. For instance, among distributions supported on a compact interval $[a,b]$, the entropy is maximized by the uniform distribution over $[a,b]$. Since all sampled values from the uniform distribution are equally typical, we may say that no one sampled value is typical of the distribution as a whole. 
We note that while $h_{\p}(\thetav)$ may appear to be a function of the random variable $\thetav$, \eqref{eq:entropy_def} shows that $h_{\p}$ is an operator on random variables (like expectation and variance) yielding a fixed real number (which only depends on the density $\p(\thetavc)$). This notation is standard in information theory, and we adopt it here. This interpretational convention will also apply the subsequent information-theoretic quantities introduced in this section.

With $(\yv,\thetav)\sim\p\left( \yvc,\thetavc \right)$, the conditional entropy of $\thetav$ given $\yv$ is
\begin{equation}
  \label{eq:cond_entropy_def_app}
  h_{f(\yvc,\thetavc)}\left( \thetav\mid\yv \right) = \E\Big[-\log \p(\thetav\mid\yv)\Big].
\end{equation}
We note that the conditional entropy is also the average entropy of the conditional distributions $\p(\thetavc\mid\yv)$, i.e. $h_{\p(\yvc,\thetavc)}\left( \thetav\mid\yv \right) = \E_{\yv\sim\p(\yvc)}\left[ h_{\p(\thetavc\mid\yv)}(\thetav) \right]$.

The mutual information between $\yv$ and $\thetav$ is defined as
\begin{equation}
  \label{eq:mi_def}
  \MI_{\p(\yvc,\thetavc)}\left( \yv, \thetav \right) = h_{\p(\thetavc)}\left(\thetav\right) - h_{\p(\yvc, \thetavc)}\left(\thetav\mid \yv\right).
\end{equation}
This mutual information may be interpreted as the expected reduction in entropy (and gain in information) when passing from the prior to the posterior distribution. We note that the mutual information is symmetric in its arguments: $\MI_{\p}(\yv,\thetav) = \MI_{\p}(\thetav,\yv)$.
Sometimes, the joint model will extend over a replicated dataset $\yrep$. In this case, we may also define the conditional mutual information between $\thetav$ and $\yrep$ given $\yv$ as
\begin{equation}
  \label{eq:cmi_def}
  \MI_{\p(\yvc,\yrepc,\thetavc)}\left( \yrep, \thetav\mid \yv \right) = h_{\p(\yrepc, \thetavc)}\left(\thetav \mid \yv\right) - h_{\p(\yvc,\yrepc,\thetavc)}\left(\thetav\mid \yv, \yrep\right).
\end{equation}

The mutual information can also be defined in terms of a measure of discrepancy between distributions, the KL divergence, which is defined for densities $\p_1,\p_2$ as
\begin{equation}
  \label{eq:KL_div_def}
  \kldiv{\p_1}{\p_2} = \E_{\thetav \sim \p_1(\thetavc)}\left[\log \frac{\p_1(\thetav)}{\p_2(\thetav)}\right].
\end{equation}
The KL divergence measures discrepancy between $\p_1$ and $\p_2$ insofar as $\kldiv{\p_1}{\p_2} \geq 0$ with equality if and only if $\p_1 = \p_2$ $\p_1$-almost everywhere. While the KL divergence is not a distance (it is asymmetric and does not satisfy the triangle inequality), Pinsker's inequality states that $\sqrt{\kldiv{\p_1}{\p_2}/2}$ upper bounds the total variation distance. Using KL divergence, the mutual information can be written as
\begin{equation}
  \label{eq:MI_KL}
  \MI_{\p(\yvc,\thetavc)}\left(\yv, \thetav\right) = \E_{\yv\sim\p(\yvc)}\Big[\kldiv{\p(\thetavc\mid \yv)}{\p(\thetavc)}\Big].
\end{equation}
When distributions are clear from context, we may drop subscripts from entropies and mutual informations, writing e.g. $h(\thetav)$ and $\MI\left( \yv, \thetav \right)$.

\subsection{Definitions of Quantities for Theorem \ref{thm:tradeoff}}\label{sec:quantity_defs}

Using these information-theoretic ideas, we construct formal definitions for the four quantities used in the informal statement of Theorem \ref{thm:tradeoff_informal} above.

\subsubsection{Identifiability Measure}

For a vector of parameters $\absparam$, we define a measure of identifiability as follows.

\begin{definition}[Identifiability Mutual Information]\label{def:imi}
  For a model $\p$ over data $\yv$, let $\absparam$ denote (a subset of) the model's parameters. We define the identifiability mutual information (\imi) for $\absparam$ as 
  \begin{equation}
    \label{eq:epsilon_id_model}
    \imim_{\p}(\absparam) = \MI_{\p}\left(\yv, \absparam\right).
  \end{equation}
\end{definition}

Smaller values of $\imim_{\p}(\absparam)$ correspond to weaker identification of $\absparam$. In particular, if $\imim_{\p}(\absparam)$ is small, 
then \eqref{eq:MI_KL} implies that the posterior must be close to the prior distribution, i.e. $p(\thetavc\mid \yvc) \approx p(\thetavc)$. More concretely, if $\absparam$ is a scalar parameter with a symmetric, log-concave prior, then Theorem 1 of \citet{MadimanBound} implies that
\begin{equation}
  \label{eq:exp_sd_lb}
  \sqrt{\E \mathrm{Var}\left(\absparam\mid \yv\right)} \geq \left(\frac{6}{\pi e}\right)^{1/2} \exp\left(-\imim_{\p}(\absparam)\right)\mathrm{SD}\left(\absparam\right).
\end{equation}
For instance, if $\imim_p(\absparam) < 0.1$, then this tells us that $\sqrt{\E \mathrm{Var}\left(\absparam\mid \yv\right)}$ must be at least $\approx 3/4$ of the prior standard deviation of $\absparam$.

To compare identifiability between a base model $\p_{\b}$ and expanded model $\p$, we compute the \imi with respect to the shared parameters $\absparam = \thetav$. We will usually simplify the notation as $\imim_{\b} = \imim_{\p_{\b}}(\thetav)$ and $\imim = \imim_{\p}(\thetav)$. For the purpose of such comparisons, we must keep in mind that the \imi measures identifiability \textit{relative} to the prior of $\thetav$. Often, this is natural (e.g. when we are concerned with the contribution of a research finding to existing prior knowledge). However, if we expand a model in such a way that adds prior information about $\thetav$, then it is possible for both the posterior entropy of $\thetav$ and the identifiability mutual information to decrease. In other words, identification relative to the prior may decrease while the posterior becomes more concentrated relative to the base model. This disconnect between absolute and relative notions of identification can be avoided if we exclude from consideration expansions which decrease the prior entropy $h(\thetav)$, which will hold for most interesting cases of model expansion.

\subsubsection{Falsifiability Measure}

We now use information-theoretic constructs to quantify falsifiability. In Section \ref{sec:intro}, we discussed falsifiability in terms of posterior predictive $p$-values and tests, where falsifiability is represented by the power of such tests against alternative models. Here, we give a definition that quantifies falsifiability without reference to a test statistic.

\begin{definition}[Posterior Sampling Divergence]\label{def:psd}
  For data $\yv$ and model $p(\yvc, \absparamc)$, the posterior sampling divergence (PSD) is
  \begin{equation}
    \label{eq:psd}
    \psdm\left( \yvc \right) = \E\Big[ \kldiv{\p(\yrepc\mid\absparam)}{\p(\yrepc\mid\yvc)}\;\Big\vert\; Y = y \Big]. 
  \end{equation}
\end{definition}

Since the density $\p(\yrepc\mid \yvc)$ is equal to $\E \p(\yrepc\mid\absparam)$ for $\absparam\sim \p(\absparamc\mid\yvc)$, the \psd is just the mean discrepancy between a (randomly chosen) sampling distribution $\p(\cdot\mid\absparam)$ and its posterior average. In analogy with the variance, we thus interpret the posterior sampling divergence as measuring the variability of sampling distributions $\p(\cdot\mid \absparam)$ when $\absparam$ is drawn from the posterior. Averaging the PSD over the prior predictive distribution $\p(\yvc)$ yields another mutual information, which we use as our quantification of falsifiability.

\begin{definition}[Falsifiability Mutual Information]\label{def:fmi}
  For a model $\p(\yvc,\absparamc)$, the falsifiability mutual information (\fmi) is defined as
  \begin{equation}
    \label{eq:fmi_def}
    \fmim_{\p}(\absparam) \stackrel{\mathrm{def}}{=}-\E_{\yv\sim\p(\yvc)}\Big[ \psdm(\yv) \Big] = -\MI_{\p}\left(\yrep, \absparam\mid \yv \right),
  \end{equation}
  where $\yrep\mid\yv,\absparam\sim \p(\cdot\mid\absparam)$. 
\end{definition}

The mutual information $\MI_{\p}\left(\yrep, \absparam\mid \yv \right)$ is negated so that falsifiability falls as $\fmim(\absparam)$ falls
 (just as lower values of $\imim(\absparam)$ indicate worse identifiability). When considering base and expanded models $\p_{\b}(\yvc,\thetavc)$ and $\p(\yvc,\thetavc,\lambdavc)$, we take $\absparam = \thetav$ and $\absparam = (\thetav,\lambdav)$ respectively, and we abbreviate notation as $\fmim_{\b} = \fmim_{\p_{\b}}(\thetav)$ and $\fmim = \fmim_{\p}((\thetav,\lambdav))$.

While the mutual information underlying the \imi has been extensively studied in the Bayesian statistics literature, we are aware of only one prior occurrence of the mutual information \eqref{eq:fmi_def}, applied to study a problem of prediction for nonexchangeable data \citep{SoofiInfo}. Furthermore, while the \imi has a simple interpretation as expected information gain, the connections between the \fmi and model assessment are more subtle and various. We give two interpretations here, and a third in Section \ref{sec:alt_interp}.

\textbf{Testability of model predictions.} If $\fmim(\absparam) = 0$, then for almost all $\absparamc$ in the support of $\p(\absparamc\mid\yvc)$, we must have $\p(\yrepc \mid\absparamc) = \p(\yrepc\mid\yvc)$ (almost everywhere). In this case, the model makes a fully specific prediction about the true process $\p(\cdot\mid\absparamc)$ that generated the data, and testing the model reduces to evaluating this single predicted distribution (e.g. using a hypothesis test). On the other extreme, when $\fmim(\absparam) \gg 0$, there will be many pairs of $\absparamc_1$, $\absparamc_2$ with $\p(\absparamc_1\mid\yvc)\approx \p(\absparamc_2\mid\yvc)$, but for which the distributions $\p(\yrepc\mid\absparamc_1)$ and $\p(\yrepc\mid\absparamc_2)$ differ substantially. For a particular testing procedure, it may easily happen in this case that $\p(\yrepc\mid\absparamc_1)$ is rejected and $\p(\yrepc\mid\absparamc_2)$ is not. How the model as a whole should be evaluated in such cases is ambiguous unless we have some benchmark with which to compare our results - i.e. a means of ``calibrating'' the test.

This view bears some resemblance to Karl Popper's concept of the informative content of a scientific theory, whereby a theory which makes more precise predictions is more readily testable and has a higher corresponding informative content \citep{Popper}. By this analogy --- eliding the important differences between a scientific theory and a statistical model --- the \fmi can be thought of as measuring the ``informative content'' of a model.

\textbf{Connection with posterior predictive power.} The \fmi may also be related to the posterior predictive testing framework. Recall that, for a fixed significance threshold, the power of a posterior predictive test tends to decline relative to a calibrated test as the PVP increases, where this was given as $\pvp_T(\yvc)=\mathrm{Var}_{\thetav \sim \p(\thetavc\mid\yvc)}\left(p_T(\yvc\mid\thetav)\right)$. 
Like the \fmi, the PVP may be thought of as a kind of average discrepancy between sampling distributions. To illustrate this, we define for observed data $\yvc$, test statistic $T$, and a distribution $\p(\yrepc)$ the tail probability $S_T(\p) = \P_{\yrep \sim \p(\yrepc)}\left(|T(\yrep)| \geq |T(\yvc)|\right)$. Then we define a ``divergence'' $d^2_S$ between distributions $\p_1$ and $\p_2$ by
\begin{equation}
  \label{eq:stat_div}
  d^2_S\kldivx{\p_1}{\p_2} = \left(S_T(\p_1) - S_T(\p_2)\right)^2.
\end{equation}
Like the KL divergence, we have $d^2_S(\p_1,\p_2)\geq 0$ with equality if $\p_1 = \p_2$. However, we may have $d^2_S(\p_1,\p_2) = 0$ even if $\p_1\neq \p_2$. The PVP and \psd may now be expressed as:
\begin{equation}
  \label{eq:div_compare}
  \begin{split}
    \pvp_T(\yvc) &= \E \Big[ d^2_S\kldivxbig{\p(\yrepc\mid\thetav)}{\E\Big[ \p(\yrepc\mid\thetav)\;\Big\vert\; \yv=\yvc \Big]}\;\Big\vert\; \yv = \yvc \Big],\\
    \psdm(\yvc) &= \E \left[ \kldivbig{\p(\yrepc\mid\thetav)}{\E \left[ \p(\yrepc\mid\thetav)\;\Big\vert\; \yv=\yvc \right]} \;\Big\vert\; \yv=\yvc\right],
  \end{split}
\end{equation}
where all expectations are with respect to $\thetav\sim\p(\thetavc\mid\yvc)$. For any statistic $T$, the \psd cannot directly control the power of the corresponding test, since it is defined independently of any specific test quantity. However, the \psd does control an upper bound on the worst case $\pvp_T$ over all possible test statistics $T$. Specifically, the Bretagnolle–Huber inequality implies the following result.

\begin{lemma}\label{lem:psd_connection}
\begin{equation}
  \label{eq:pinsker_psd}
  \sup_{T:\dspace\to\RR}\pvp_T(\yvc) \leq 1 - \exp\left(-\psdm(\yvc)\right),
\end{equation}
where the supremum is taken over all measurable test statistics $T$.
\end{lemma}
\begin{proof}
  See Section 2 of the Supplementary Material \citep{SuppMat}.
\end{proof}

This bound is trivial when the right-hand side exceeds $1/4$ since Popoviciu's inequality tells us that $\sup_{T:\dspace\to\RR}\pvp_T(\yvc)\leq 1/4$. Nevertheless, in light of both \eqref{eq:pinsker_psd} and \eqref{eq:div_compare}, we associate a lower \fmi ($= -\E\left[ \psdm(\yv) \right]$) with an increased risk of low posterior predictive power, at least when testing with casually chosen test statistics $T$.
While we focus on these given interpretations of the \imi and \fmi, we note that alternative interpretations in terms of certain Bayes factors and out-of-sample fitness measures are also possible. We discuss these in more detail in Section \ref{sec:alt_interp}.

\subsubsection{Expansion Measure}

To define our expansion measure, we first decompose $\fmim$ as follows:
\begin{equation}
  \label{eq:cmi_decomp_excess}
  \fmim = -\MI\left(\yrep, (\thetav,\lambdav)\mid \yv\right) = -\MI\left(\yrep, \thetav\mid \yv\right) - \MI\left(\yrep, \lambdav\mid\yv,\thetav\right).
\end{equation}

The first term can be compared directly to the \fmi in the base model $\fmim_{\b} = -\MI_{\b}\left(\yrep, \thetav\mid\yv\right)$. To interpret the second term, we represent it as a divergence:
\begin{equation}
  \label{eq:excess_cmi_kl}
  \MI\left(\yrep, \lambdav\mid\yv,\thetav\right) = \E_{\yv \sim \p(\yvc)}\Big[\E\Big[ \kldiv{\p(\yrepc\mid \thetav,\lambdav)}{\p(\yrepc\mid \yv,\thetav)}\;\Big\vert\;\yv \Big]\Big].
\end{equation}
This representation shows that $\MI\left(\yrep, \lambdav\mid\yv,\thetav\right)$ measures the variability in the sampling distributions due to variation in the expansion parameters $\lambdav$, fixing $\thetav$. The inner expectation in \eqref{eq:excess_cmi_kl} is essentially the PSD \eqref{eq:psd} conditioned on $\thetav$. We refer to this inner expectation as the excess \psd (henceforth EPSD) due to the expansion parameters $\lambdav$:
\begin{equation}
  \label{eq:excess_psd}
  \mathrm{EPSD}_{\p}(\yvc) = \E\Big[ \kldiv{\p(\yrepc\mid \thetav,\lambdav)}{\p(\yrepc\mid \yvc,\thetav)} \;\Big\vert\; \yv = \yvc\Big].
\end{equation}
 In terms of the EPSD, we can write $\MI\left(\yrep,\lambdav\mid\yv,\thetav\right) = \E\left[ \mathrm{EPSD}(\yv) \right]$. With this, we can now define the expansion ratio of an expanded model.

\begin{definition}[Expansion Ratio]\label{def:exp_rat}
  For an expansion $\p(\yvc,\thetavc,\lambdavc)$ of $\p_{\b}(\yvc,\thetavc)$, we define the expansion ratio as:
\begin{equation}
  \label{eq:excess_ratio}
  \expratio = \frac{\E_{\yv \sim \p(\yvc)}\Big[ \mathrm{EPSD}_p(\yv) \Big]}{\E_{\yv \sim \p_{\b}(\yvc)}\Big[ \psdm_{\b}(\yv) \Big]} = \frac{\MI\left(\yrep, \lambdav\mid\yv,\thetav\right)}{-\fmim_{\b}}.
\end{equation}

\end{definition}

We think of $\expratio$ as measuring excess complexity introduced by the expansion parameters $\lambdav$ (after conditioning on $\thetav$) relative to the complexity of the base model (where complexity is expressed in terms of the sampling divergences PSD and EPSD).

\subsubsection{Expansion Threshold}

While the expansion ratio measures the amount of expansion, we now need to know how large $\expratio$ can get before forcing a trade-off between the \imi and \fmi. This threshold is given by the contraction coefficient -- a nonlinear analog of the squared correlation. If $\thetav = \thetavc$ and $\yv = \yvc$ are scalar and linearly dependent, the squared correlation between $\thetav$ and $\yv$ is the fraction of the variance in $\thetav$ explained by $\yv$. The contraction coefficient replaces this variance with the mutual information $\MI\left(\yv,\thetav\right)$ and quantifies how much of this is ``explained'' by $\yv$. 

To make this notion precise, we consider any alternative joint model $\palt(\yvc',\thetavc)$ satisfying the constraints $\MI_{\palt}\left(\yv',\thetav\right)=\MI_{\p}\left(\yv,\thetav\right)$ and $\E_{\palt}\|\thetav\|^2 = \E_{\p}\|\thetav\|^2$. (Subject to these constraints, we may have both $\palt(\thetavc) \neq \p(\thetavc)$ and $\palt(\yvc'\mid \thetavc) \neq \p(\yvc\mid\thetavc)$.) If we then sample $\yv\sim \p(\yvc\mid\thetavc)$, we can construct the auxiliary model:
\begin{equation}
  \label{eq:aux_model}
  \paux\left(\yvc,\yvc',\thetavc\right) = \p\left(\yvc\mid\thetavc\right)\palt\left(\yvc'\mid\thetavc\right)\palt\left(\thetavc\right).
\end{equation}
Roughly, this construction allows us to characterize the ``strength'' of the sampling distribution $\p(\yvc\mid\thetavc)$ by attaching it to various alternative models $\palt$ as in \eqref{eq:aux_model}. In particular, the data processing inequality from information theory tells us that
\begin{equation}
  \label{eq:dpi}
  \MI_{\paux}\left(\yv',\yv\right) \leq \MI_{\paux}\left(\yv',\thetav\right) \stackrel{(*)}{=} \MI_{\p}\left(\yv,\thetav\right),
\end{equation}
where $(*)$ follows from the constraints on $\palt$. When $\yv \sim \p(\yvc\mid\thetavc)$ accurately predicts $\thetav$, we have $\MI_{\paux}\left(\yv',\yv\right) / \MI_{\p}\left(\yv,\thetav\right) \approx 1$. On the other hand, if $\yv$ is uninformative about $\thetav$, then $\MI_{\paux}\left(\yv',\yv\right) / \MI_{\p}\left(\yv,\thetav\right) \ll 1$. This ratio is almost a useful measure of association, but it depends on a particular alternative model $\palt$. Taking the supremum over all $\palt$ satisfying our constraints removes this dependence and yields the contraction coefficient.

\begin{definition}[Contraction Coefficient]
  For a model $\p(\yvc,\thetavc)$, the contraction coefficient is defined as
  \begin{equation}
    \label{eq:contract_coef}
    \eta_{\p} = \frac{\Gamma\left(\MI_{\p}\left(\yv,\thetav\right),\E_{\p}\|\thetav\|^2\right)}{\MI_{\p}\left(\yv,\thetav\right)},
  \end{equation}
  where the function $\Gamma$ is the $F_I$ curve of \cite{SDPI_Nonlinear} and is given as
  \begin{equation}
    \label{eq:FI_curve}
    \Gamma(t;\gamma) = \sup_{\palt(\yvc',\thetavc)}\left\lbrace \MI_{\paux}\left(\yv, \yv'\right)\;\Bigg\vert\;\MI_{\palt}\left(\yv',\thetav\right) \leq t, \E_{\thetav\sim\palt(\thetavc)}\|\thetav\|^2\leq \gamma  \right\rbrace,
  \end{equation}
  and where $\paux\left(\yvc,\yvc',\thetavc\right)$ is the auxiliary model corresponding to $\palt$, as defined in \eqref{eq:aux_model}.
\end{definition}

It follows from the nonnegativity of mutual information and the data processing inequality \eqref{eq:dpi} that $0\leq \eta_{\p}\leq 1$. In light of our previous observations, we view $\eta_{\p}$ as the fraction of $\MI_{\p}\left(\yv,\thetav\right)$ explained by the sampling distribution $\p(\yvc\mid\thetavc)$.

\textit{Remarks.}
\begin{enumerate}
  \item The second moment constraint $\E_{\thetav\sim\palt(\thetavc)}\|\thetav\|^2 \leq \gamma$ in \eqref{eq:FI_curve} is often necessary for $\eta_{\p}  < 1$. This mirrors the properties of correlation, since, for scalar $\yvc$, $\thetavc$ and fixed conditional variance $\mathrm{Var}\left(\yvc\mid \thetavc\right)$, the correlation tends to $1$ as $\mathrm{Var}(\thetavc)\to\infty$. 
  \item On the other hand, for fixed marginal variance $\mathrm{Var}(\theta)$, the correlation tends to $1$ as $\mathrm{Var}(y\mid\theta)\to 0$. Similarly, the contraction coefficient approaches $1$ as (a sample from) $\p(\yvc\mid\thetavc)$ gives more accurate information about $\thetav$.
  \item For scalar $\yv$ and $\thetav$, it is shown in \cite{SDPI_Nonlinear} that if $\yv = \thetav + W$ with $W\in \RR$ supported on an infinite interval, then $\eta_{\p} < 1$. The authors state that this can be generalized to the case of multidimensional additive noise.
\end{enumerate}

\subsection{Main Result}\label{sec:thm1_formal}

With our four quantities fully defined, we now give a formal statement of Theorem \ref{thm:tradeoff}.

\begin{theorem}[Identifiability - Falsifiability Trade-off]\label{thm:tradeoff}
  Let $\p(\yvc,\thetavc,\lambdavc)$ be an expansion of $\p_{\b}(\yvc,\thetavc)$. Let $\eta_{\p}$ be the contraction coefficient of $\p$ with respect to $\thetav$ (i.e. the contraction coefficient computed with respect to the marginal $\p(\yvc,\thetavc) = \int \p(\yvc,\thetavc,\lambdavc)d\lambdavc$). If we have $\eta_{\p} < 1$ and
  \begin{equation}
    \label{eq:tradeoff_condition}
    \expratio \geq \eta_{\p},
  \end{equation}
  then there is a strict trade-off between identifiability and falsifiability in that
  \begin{equation}
    \label{eq:id_fa_trade}
    \begin{split}
      \imim \geq \imim_{\b} &\implies \fmim \leq \fmim_{\b}, \text{ and}\\
      \fmim \geq \fmim_{\b} &\implies \imim \leq \imim_{\b}.
    \end{split}
  \end{equation}
\end{theorem}

\begin{proof}
  See the Supplementary Material \citep{SuppMat} for a full proof. For now, we briefly sketch the main ideas. First, if we suppose (falsely) that, for all models $\p$,
  \begin{equation}
    \label{eq:supposition}
    \imim_{\p} = \MI_{\p}(\yrep,\thetav\mid\yv),
  \end{equation} 
  then the conclusion \eqref{eq:id_fa_trade} would hold for any model expansion. To see this, first consider the case where $\imim \geq \imim_{\b}$. Then \eqref{eq:cmi_decomp_excess} shows that
  \[
    \fmim = -\MI_{\p}\left( \yrep, (\thetav,\lambdav)\mid\yv \right) \stackrel{\mathrm{(a)}}{\leq} -\imim \leq -\imim_{\b} \stackrel{\mathrm{(b)}}{=} -\MI_{\p_{\b}}(\yrep,\thetav\mid\yv) = \fmim_{\b},
  \]
  where $\mathrm{(a)}$ and $\mathrm{(b)}$ follow by applying \eqref{eq:supposition} to the expanded and base models, respectively. Now consider the case where $\fmim \geq \fmim_{\b}$. We now get
  \[
    \imim \stackrel{\mathrm{(a)}}{=} \MI_{\p}(\yrep,\thetav\mid\yv) \leq \MI_{\p}\left( \yrep, (\thetav,\lambdav)\mid\yv \right) = -\fmim \leq -\fmim_{\b} \stackrel{\mathrm{(b)}}{=} \imim_{\b},
  \]
  where $\mathrm{(a)}$ and $\mathrm{(b)}$ follow exactly as above. 
  In reality, the supposition \eqref{eq:supposition} does not hold. However, we demonstrate (in Lemma 10 of the Supplementary Material \citep{SuppMat}) that, for all models with $\eta_{\p}<1$, we do have the constraint
  \begin{equation}
    \label{eq:main_constraint}
    \imim_{\p}(\thetav) \geq \MI_{\p}(\yrep,\thetav\mid\yv) \geq (1 - \eta_{\p})\imim_{\p}(\thetav).
  \end{equation}
  Because of the gap in the inequality \eqref{eq:main_constraint}, an increase in $\imim$ does not guarantee a decrease in $\fmim$. The sufficient condition \eqref{eq:tradeoff_condition} essentially fills this gap, forcing the trade-off \eqref{eq:id_fa_trade}.
\end{proof}

\textit{Remarks:}
\begin{enumerate}
\item The first implication of \eqref{eq:id_fa_trade} holds even if $\eta_p = 1$.
\item The contraction coefficient $\eta_p$ can be replaced with the ratio $\eta_{\p}^* = \frac{\MI_{\p}\left(\yv,\yv'\right)}{\MI_{\p}\left(\yv,\thetav\right)}$,
  where $\yv,\yv'\stackrel{iid}{\sim} \p(\yvc\mid\thetavc)$. Since $\eta^*_{\p}$ depends on the specific prior $\p(\thetavc)$, it does not enjoy the interpretation that $\eta_{\p}$ has in terms of the sampling distribution $\p(\yvc\mid\thetavc)$. 
  However, we have that $\eta_{\p}^* \leq \eta_{\p}$, and $\eta_{\p}^* = 1$ if and only if $\yv'$ and $\thetav$ are independent given $\yv$, i.e. if additional data cannot improve our knowledge of $\thetav$. Thus, we expect $\eta_{\p}^*< 1$ in all practical problems.
\item Both $\eta_{\p}$ and $\expratio$ depend on the expanded model. However, whereas $\eta_{\p}\leq 1$, $\expratio$ is unbounded above. Thus, as long as $\eta_{\p}<1$, a sufficient condition for the trade-off \eqref{eq:id_fa_trade} is for $\expratio \geq 1$, i.e. for the sampling distribution variability due to $\lambdav$ (the EPSD) to at least match the base variability (the base PSD).
\end{enumerate}

\subsubsection{Implications for Priors}

If we wish to avoid the challenges that model expansion poses for identifiability and falsifiability, Theorem \ref{thm:tradeoff} implies that we should construct expansions with small $\expratio$. Unfortunately, due to the difficulty of computing $\expratio$ in realistic models, this advice is hardly actionable. However, an important special case occurs when
\begin{equation}
  \label{eq:good_cind}
  \yv\ind\lambdav\;\mid\;\thetav.
\end{equation}
This conditional independence implies $\yrep\ind\lambdav\mid\thetav, \yv$, and therefore $\expratio = 0$. An expansion satisfies \eqref{eq:good_cind} if $\lambdav$ only enters the prior distribution and has no direct influence on the likelihood. We will thus refer to expansions satisfying \eqref{eq:good_cind} as prior expansions. While limiting ourselves exclusively to prior expansions would be overly restrictive, we may still  benefit from combining general expansions with prior expansions. Because prior expansions can achieve both $\fmim > \fmim_{\b}$ and $\imim >\imim_{\b}$ (see Section \ref{ex:4} for an example of this), combining them with general expansions may limit the overall downward influence of model expansion on $\imim$ and $\fmim$.

Prior expansions that use $\lambdav$ to impose a soft constraint on the complexity of the model are particularly useful for this purpose. In regression, for example, we could expand from an i.i.d. prior on the coefficients to a sparsifying prior, where $\lambdav$ limits the number of ``large'' coefficients. More generally, penalized complexity priors provide an explicit framework for constructing prior distributions $\p(\thetavc,\lambdavc)$ that limit departure from a simpler baseline model \citep{PCPriors}.

We also emphasize that while prior expansions usually encode nontrivial prior information, condition \eqref{eq:good_cind} does not require priors to be \textit{marginally} informative about any (scalar) parameter. For instance, suppose that the parameter vector $\thetav$ contains some particular parameters $\subtheta$ of substantive interest for which we want our inference to be minimally influenced by our choice of prior. The framework of reference priors defines a precise sense in which a prior $\p(\subthetac)$ can be minimally informative about $\subtheta$ \citep{RefPrior,RefPriorAsymp}. Because the reference prior for $\subtheta$ is defined marginally, we could construct a prior expansion using a reference prior on $\subtheta$ while choosing the conditional prior $\p(\lambdavc\mid \thetavc)$ freely (e.g. in a manner that limits model complexity).

\subsection{Alternative Interpretations of \texorpdfstring{$\imim$}{\imi} and \texorpdfstring{$\fmim$}{\fmi}}\label{sec:alt_interp}

Before presenting examples, we examine some additional interpretations of $\imim$ and $\fmim$.
\subsubsection{Measures of Generalizability}

As we discussed in Section \ref{sec:intro}, the power deficiency of posterior predictive tests can be explained by overfitting. It is thus unsurprising that the \fmi can be associated with a measure of out-of-sample generalizability, the expected log predictive density (ELPD):
\begin{equation}
  \mathrm{ELPD}(\yvc, \p_*) = \E_{\yval \sim \p_*} \Big[\log \p(\yv^{\mathrm{val}}\mid\yv = \yvc)\Big].
\end{equation}
Here, $\p_*(\cdot)$ is the density of the true data generating distribution, and $f(y^{\mathrm{val}}\mid \yvc)$ is the posterior predictive density. The ELPD measures how well we expect a fitted Bayesian model to predict the values in a hypothetical validation dataset $\yv^{\mathrm{val}}$ drawn from the same distribution as the observed $\yv$. The ELPD cannot be evaluated directly, but many model evaluation metrics have been proposed to approximate it under appropriate assumptions, including the WAIC \citep{LooWAIC} and approximate leave-one-out cross validation \citep{LOO}.

If we assume that the model is correctly specified, so that $\p_*(\cdot) = \p(\cdot\mid \absparamc)$ for some fixed $\absparam$, then we can relate the ELPD to the \fmi:
\begin{equation}
  \label{eq:elpd_fmi}
  \fmim(\absparam) = \E \Big[ \mathrm{ELPD}(\yv,\p(\cdot\mid\absparam)) \Big] + h\left(\yv\mid\absparam\right).
\end{equation}
The entropy term $h\left(\yv\mid\absparam\right)$ can be viewed as a measure of irreducible error -- unpredictability in the data generating process that remains after determination of all model parameters. Holding this entropy fixed (or at least nondecreasing), \eqref{eq:elpd_fmi} tells us that lower values of the \fmi are also associated with lower ELPD and hence worse out-of-sample generalizability (even if the model is correctly specified).

\subsubsection{Bayes Factors}

The Bayes factor is another common tool for model evaluation and comparison. For two models, $\p_1$ and $\p_2$ and a particular value of the observed data $\yvc$, the Bayes factor is the ratio of the marginal likelihoods: $\p_1(\yvc) / \p_2(\yvc)$. Despite the similarities in purpose between Bayes factors and posterior predictive $p$-values, the Bayes factor is \textit{not} directly connected to the \fmi. While the Bayes factor is an a priori comparison of models (i.e. the Bayes factor does not condition on the observed data), the \fmi relates to the performance of a posteriori model evaluations (including both posterior predictive $p$-values and ELPD approximations like cross validation).

Bayes factors can, however, be related to the \imi for certain choices of models $\p_1$ and $\p_2$. Consider a base model $\p_{\b}(\yvc,\thetavc)$ and a hypothesis that $\subtheta = \subthetac$, where $\subtheta$ is some subset of the total parameter vector $\thetav$ and $\subthetac$ is a particular value. To test this hypothesis, we construct the Bayes factor:
\begin{equation}
  \label{eq:bayes_factor_base}
  \frac{\p_{\b}(\yvc\mid \subtheta = \subthetac)}{\p_{\b}(\yvc)}.
\end{equation}
If we expand $\p_{\b}$ to a larger model $\p(\yvc,\thetavc,\lambdavc)$, our Bayes factor then becomes
\begin{equation}
  \label{eq:bayes_factor_exp}
  \frac{\p(\yvc\mid \subtheta= \subthetac)}{\p(\yvc)}.
\end{equation}
We note that while the hypothesis ($\subtheta = \subthetac$) is the same, both numerator and denominator have changed, since we now marginalize out $\lambdav$. 
If we write $\thetav = (\subtheta, \overline{\subtheta})$, then we have that $\p(\yvc, \subthetac, \lambdavc)$ is an expansion of $\p_{\b}(\yvc, \subthetac)$ (since integrating out $\overline{\subthetac}$ preserves the property \eqref{eq:expansion_def}). Thus, our results can be applied to the mutual informations $\imim_{\b}(\subtheta)$ and $\imim(\subtheta)$. In particular, using the KL divergence representation \eqref{eq:MI_KL}, we can write
\begin{equation}
  \label{eq:kl_bf}
  \begin{split}
  \imim_{\b}(\subtheta) &= \E\left[ \E \left[ \log\frac{\p_{\b}(\yv\mid \subtheta)}{\p_{\b}(\yv)}\;\Bigg\vert\; \subtheta \right] \right],\\
   \imim(\subtheta) &= \E\left[ \E \left[ \log\frac{\p(\yv\mid \subtheta)}{\p(\yv)}\;\Bigg\vert\; \subtheta \right] \right].
  \end{split}
\end{equation}
If the IMI decreases ($\imim(\subtheta)< \imim_{\b}(\subtheta)$), then \eqref{eq:kl_bf} implies that, on average over possible hypotheses $\subthetac$, the Bayes factor is expected to fall with expansion even if the simpler model is correct (since the inner expectations are taken over the simpler model). 

\section{Example Computations}\label{sec:id_fa_expand}

We now present three examples which illustrate the conclusions of Theorem \ref{thm:tradeoff}. Throughout, we work with the lower bound $\eta^*_{\p}$ in place of $\eta_{\p}$ as the former is much simpler to compute (see remarks after Theorem \ref{thm:tradeoff}). Detailed computations for all information-theoretic quantities can be found in the Supplementary Material \citep{SuppMat}.

\subsection{Example 1: Linear Regression}\label{ex:1}

\subsubsection{Models}
Take as our base model a linear regression with predictors $\Xm_{\b}\in\RR^{n\times k}$ (with standardized, unit-norm columns) and known, unit noise variance. Specifically, we take:
\begin{equation}
  \label{eq:ex3_reg_eq}
  \yv\mid \thetav \sim \mathrm{normal}\left(\Xm_{\b}\thetav, \imat_n\right),\quad \thetav \sim\mathrm{normal}\left(0, \tau^{-1}\imat_k\right).
\end{equation}

We expand this regression by adding a new predictor $\xv_{k+1}$ (again with $\|\xv_{k+1}\| = 1$). Denoting $\Xm = [\Xm_{\b}\;\xv_{k+1}]$ and letting $\lambdav$ be the coefficient for $\xv_{k+1}$, the expansion is
\begin{equation}
  \label{eq:ex3_exp_eq}
  \yv\mid\thetav,\lambdav \sim \mathrm{normal}\left(\Xm
    \left[\begin{smallmatrix}
      \thetav\\\lambdav
    \end{smallmatrix}\right]
    , \imat_{n}\right),\quad \left[\begin{smallmatrix}
      \thetav\\\lambdav
    \end{smallmatrix}\right] \sim\mathrm{normal}\left(0, \tau^{-1}\imat_{k+1}\right).
\end{equation}

\subsubsection{Effect on Identifiability/Falsifiability}
Before computing $\imim$ and $\fmim$, we build intuition using just expectations and variances.
First, for any scalar coefficient $[\thetav]_i$ ($1 \leq i\leq k$), we can decompose $\mathrm{Var}_{p}\left([\thetav]_i\mid\yv=\yvc\right)$ as
\begin{multline}
  \label{eq:exp_post_var}
  \mathrm{Var}_{\p}\Big([\thetav]_i\;\Big\vert\;\yv=\yvc\Big) =\\ \E\Big[ \mathrm{Var}\Big( [\thetav]_i\;\Big\vert\; \yv,\lambdav \Big)\;\Big\vert\; \yv=\yvc \Big] + \mathrm{Var}\Big( \E\Big[ [\thetav]_i \;\Big\vert\; \yv,\lambdav\Big]\;\Big\vert\; \yv=\yvc\Big).
\end{multline}
Because $\p([\thetavc]_i,\lambdavc\mid\yvc)$ is jointly normal, $\mathrm{Var}\left( [\thetav]_i\mid \yv = \yvc,\lambdav=\lambdavc \right)$ is free of $\lambdavc$, and we therefore have
\begin{multline}
  \label{eq:base_var_rel}
  \E\Big[ \mathrm{Var}\Big( [\thetav]_i\;\Big\vert\; \yv,\lambdav \Big) \;\Big\vert\; \yv = \yvc\Big] =\\ \mathrm{Var}\Big( [\thetav]_i\mid \yv = \yvc,\lambdav = \lambdavc_0 \Big) = \mathrm{Var}_{\p_{\b}}\Big( [\thetav]_i\;\Big\vert\; \yv=\yvc\Big).
\end{multline}
Combining \eqref{eq:base_var_rel} with \eqref{eq:exp_post_var}, we conclude that $\mathrm{Var}_{\p}\left([\thetav]_i\mid\yv = \yvc\right) \geq \mathrm{Var}_{\p_{\b}}\left( [\thetav]_i\mid \yv=\yvc\right)$, and we thus expect the expanded model to exhibit worse identifiability and lower $\imim$.

To understand the effect of expansion on falsifiability, consider new observations $\yv^{\mathrm{rep}}_1$ and $\yv^{\mathrm{rep}}_2$ with all base predictors equal. In this case, we have 
\begin{equation}
  \label{eq:base_preds_reg}
  \mathrm{Var}_{\p_{\b}}\Big( \E_{\p_{\b}}\Big[ \yv^{\mathrm{rep}}_1\;\Big\vert\;\thetav \Big] - \E_{\p_{\b}} \Big[ \yv^{\mathrm{rep}}_2\;\Big\vert\; \thetav \Big] \;\Big\vert\; \yv = \yvc\Big)= \mathrm{Var}\left( 0 \;\Big\vert\; \yv=\yvc\right) = 0.
\end{equation}
Now if $x^{\mathrm{rep}}_{i, k+1}$ is the value of the new predictor corresponding to $y^{\mathrm{rep}}_i$ for $i=1,2$, then
\begin{align}
  \mathrm{Var}_{\p}&\Big( \E_{\p}\Big[ \yv^{\mathrm{rep}}_1\;\Big\vert\;\thetav,\lambdav \Big] - \E_{\p} \Big[ \yv^{\mathrm{rep}}_2\;\Big\vert\; \thetav,\lambdav \Big] \;\Big\vert\; \yv
  = \yvc\Big)\nonumber\\ &= \left( x^{\mathrm{rep}}_{1, k+1} - x^{\mathrm{rep}}_{2, k+1} \right)^2\mathrm{Var}_{\p}\Big(  \lambdav\;\Big\vert\;\yv = \yvc\Big) \geq 0.\label{eq:exp_preds_reg}
\end{align}
In other words, the expanded model has more predictive flexibility than the base model, giving distinct predictions for observations which the base model treats equally. Since predictive flexibility is associated with worse falsifiability, we therefore expect $\fmim \leq \fmim_{\b}$.
Explicit computations confirm our predictions for $\imim$ and $\fmim$:
\begin{equation}
  \label{eq:ex3_mi_rel}
  \begin{split}
    \imim &= \imim_{\b} - \frac{1}{2}\log\left[\frac{1 + \tau}{\pi(\xv_{k+1},\tau) + \tau}\right],\\
    \fmim &= \fmim_{\b} - \frac{1}{2}\log\left[\frac{2\pi(\xv_{k+1},\tau/2) + \tau}{\pi(\xv_{k+1},\tau) + \tau}\right],
  \end{split}
\end{equation}
where we define the quantity $\pi$ as
\begin{equation}
  \label{eq:ex3_pi_def}
  \pi(\xv_{k+1}, \tau) = \xv_{k+1}^T\left(\imat - \Xm_{\b}\left(\Xm_{\b}^T\Xm_{\b} + \tau\imat\right)^{-1}\Xm_{\b}^T\right)\xv_{k+1}.
\end{equation}
Since $\|\xv_{k+1}\|^2=1$, \eqref{eq:ex3_pi_def} corresponds to the quantity $\pi$ in \eqref{eq:reg_ex_proj} from our initial regression example as the precision $\tau$ tends to $0$. As long as $\pi(\xv_{k+1},\tau)\in (0,1)$ for $\tau\geq 0$, we have
\begin{equation}
  \label{eq:asymp_reg}
  \log\left[\frac{1 + \tau}{\pi(\xv_{k+1},\tau) + \tau}\right]\ > 0, \;\text{and}\; \log\left[\frac{2\pi(\xv_{k+1},\tau/2) + \tau}{\pi(\xv_{k+1},\tau) + \tau}\right]\to \log 2 > 0 \text{ as } \tau \to 0.
\end{equation}
Thus, for a sufficiently noninformative prior, we have $\imim < \imim_{\b}$ and $\fmim < \fmim_{\b}$, as expected. We can also use \eqref{eq:ex3_mi_rel} to quantify the trade-off between $\imim$ and $\fmim$:
\begin{equation}
  \label{eq:reg_tradeoff}
  \left(\imim - \imim_{\b}\right) - \left(\fmim - \fmim_{\b}\right) = \frac{1}{2}\log\left[\frac{2\pi(\xv_{k+1},\tau/2) + \tau}{1 + \tau}\right] = \delta(\pi(\xv_{k+1}, \tau/2), \tau),
\end{equation}
where $\delta(\pi, \tau) = \log[(2\pi + \tau) / (1 + \tau)] / 2$.
The first panel of Figure \ref{fig:ex1_fig} plots $\delta(\pi, \tau)$ against $\pi$ for $\tau = 1/4$ (i.e. for $\mathrm{SD}([\thetav]_i) = 2$). 
We observe that $\pi(\xv_{k+1},\tau / 2)$ defines a continuous trade-off --- larger values yield relatively better $\imim$, smaller values relatively better $\fmim$.

In the special case of orthonormal $\Xm_{\b}$, the function $\pi$ takes a simpler form:
\begin{equation}
  \label{eq:reg_pi_simple}
  \pi\left(\xv_{k+1}, \tau\right) = 1 - \frac{\|\Xm_{\b}^T\xv_{k+1}\|^2}{1 + \tau}.
\end{equation}

Plugging this into \eqref{eq:ex3_mi_rel}, it is easy to see that $\imim \leq \imim_{\b}$ and $\fmim \leq \fmim_{\b}$ for all $\tau$. In this case, both $\imim - \imim_{\b}$ and $\fmim - \fmim_{\b}$ are functions of just $\pi(\xv_{k+1}, \tau / 2)$. These are plotted in the second panel of Figure \ref{fig:ex1_fig}. Note that these curves mirror the trade-off between SR and PR observed in the introductory regression example in Section \ref{sec:reg_ex}.

\begin{figure}[t]
  \centering
  \includegraphics[scale=0.64]{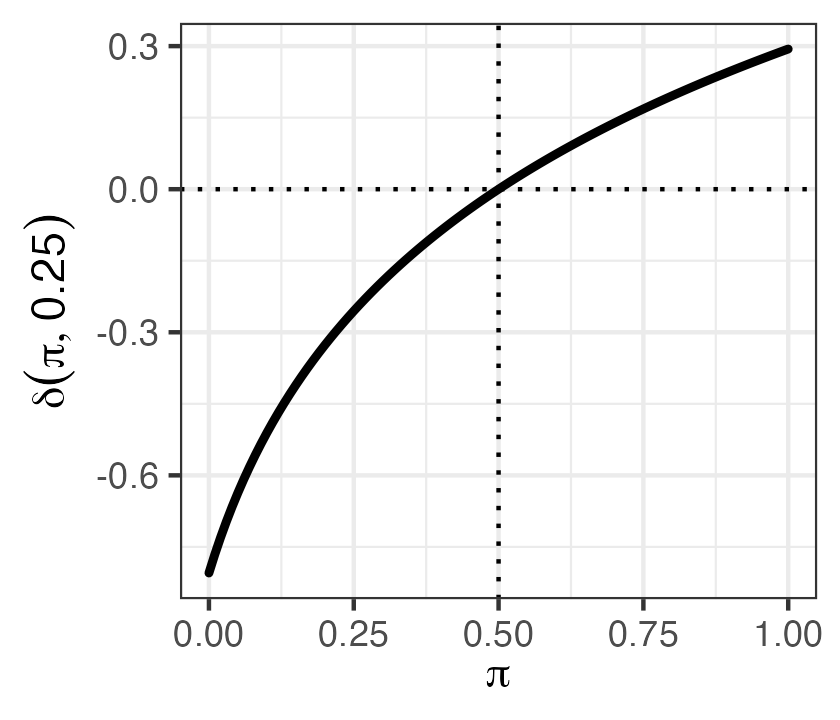}
  \includegraphics[scale=0.64]{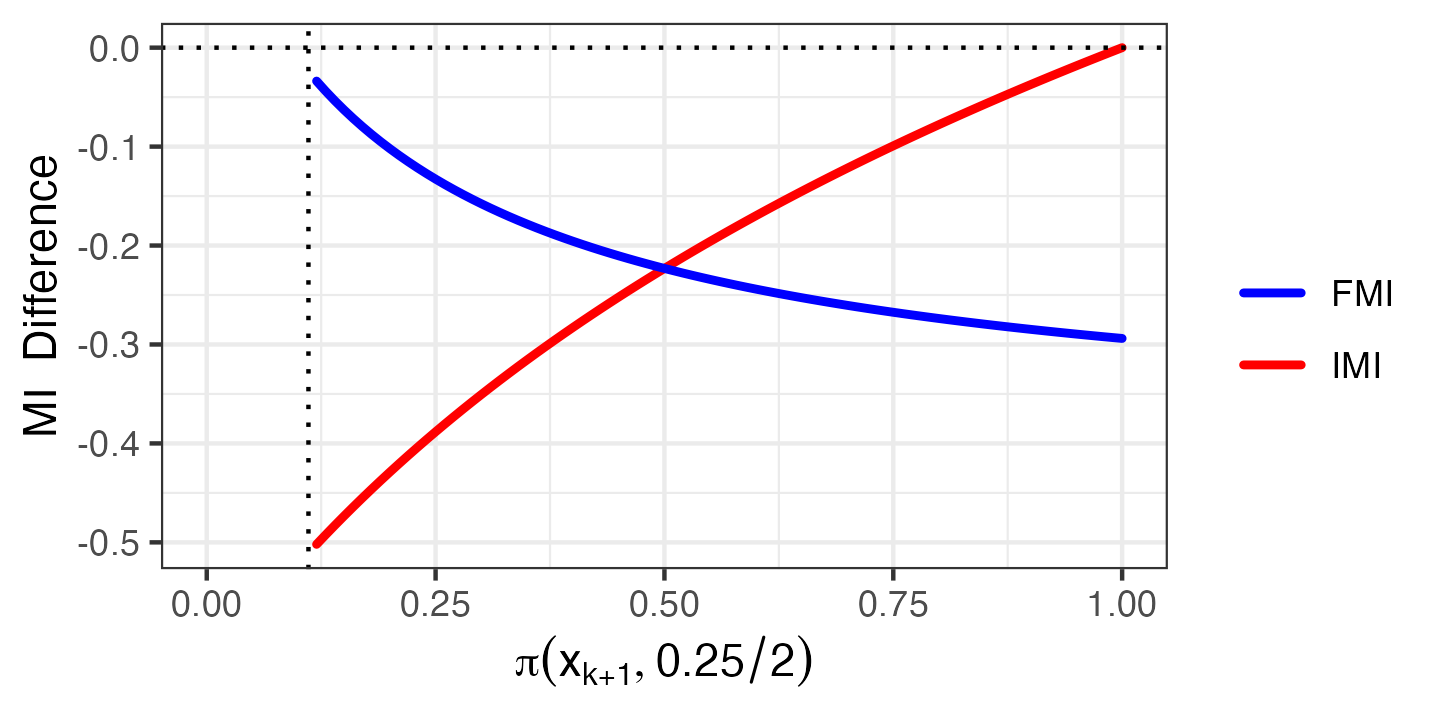}
  \caption{Left: Comparison function $\delta(\pi, \tau)$ for $\tau = 1/4$. We see a greater fall in \imi when $\delta <0$, and a greater fall in \fmi when $\delta > 0$. Right: Change in \imi (red) and \fmi (blue) from base to expanded model, plotted against $\pi(\xv_{k+1}, 0.25 / 2)$ in the special case of orthonormal $\Xm_{\b}$. When $\tau = \frac{1}{4}$, $\pi(\xv_{k+1}, \tau / 2)$ is constrained to $\left[1/9, 1\right]$.}
  \label{fig:ex1_fig}
\end{figure}

\subsubsection{Comparison with Theorem \ref{thm:tradeoff}}

In the case where $\Xm_{\b}^T\Xm_{\b} = \imat$, Theorem \ref{thm:tradeoff} holds trivially for $\pi < 1$ (since both $\imim-\imim_{\b}<0$ and $\fmim-\fmim_{\b}<0$). For $\pi = 1$, however, $\imim-\imim_{\b} = 0$ (see second panel of Figure \ref{fig:ex1_fig}), so the theorem holds nontrivially.
However, this conclusion is not always predicted by Theorem \ref{thm:tradeoff}. Suppose $n \geq k+1$ and all predictors (including $\xv_{k+1}$) are orthonormal, so that $\pi(\xv_{k+1}, \tau) = 1$ regardless of $\tau$. In this case, we have
\begin{equation*}
  \expratio = \frac{1}{k},\qquad \eta^*_{\p} = 2 - \frac{\log(1 + 2\tau^{-1})}{\log(1+\tau^{-1})}.
\end{equation*}

If the base model contains only a single predictor (i.e. $k=1$), then we will have $\expratio > \eta^*_{\p}$ for any finite prior precision $\tau>0$. On the other hand, for any fixed $\tau$, we can always make $\expratio < \eta^*_{\p}$ for $k$ large enough. In this case, the condition \eqref{eq:tradeoff_condition} of Theorem \ref{thm:tradeoff} will fail, but the conclusion \eqref{eq:id_fa_trade} will continue to hold (nontrivially since $\pi = 1$). 
This demonstrates that the condition \eqref{eq:tradeoff_condition} is sufficient but not necessary for a trade-off between the \fmi and \imi to occur. More specifically, for any fixed value of $\tau$, we can say that there is some number of predictors $k$ in the base model such that adding just one more predictor is ``too small'' an expansion for Theorem \ref{thm:tradeoff} to kick in.
   
\subsection{Example 2: Unknown Variance}\label{ex:3}

\subsubsection{Models}
For $\yv\in\RR^n$, we define the normal-normal base model:
   \begin{equation}
     \label{eq:ex4_base}
     [\yv]_i\mid\thetav \stackrel{iid}{\sim}\mathrm{normal}\left(\thetav, 1\right) \text{ for } 1\leq i\leq n,\quad \thetav \sim\mathrm{normal}\left(0, 1\right).
   \end{equation}
   We expand this model by letting the variance $\lambdav=\mathrm{Var}([\yv]_i)$ be unknown:
   \begin{equation}
     \label{eq:ex4_exp}
     \begin{gathered}
     [\yv]_i\mid\thetav,\lambdav\stackrel{iid}{\sim}\mathrm{normal}\left( \thetav,\lambdav \right) \text{ for } 1\leq i\leq n,\\
     \thetav \sim\mathrm{normal}(0 ,1), \quad\lambdav\sim\mathrm{gamma}(2\mu_{\lambdav},2).
     \end{gathered}
   \end{equation}
   Here, the prior mean and variance of $\lambdav$ is given as $\E[\lambdav] = \mu_{\lambdav}$ and $\mathrm{Var}(\lambdav) = \mu_{\lambdav}/2$.  Conditioning on $\lambdav = 1$ recovers the base model \eqref{eq:ex4_base}.

\subsubsection{Effect on Identifiability/Falsifiability}
For the base model, we have $\mathrm{Var}_{\p_{\b}}(\thetav\mid \yv = \yvc) = (1+n)^{-1}$. In the expanded model, by the law of total variance, we get
\begin{align}
  \mathrm{Var}_{\p}\Big( \thetav\;\Big\vert\;\yv=\yvc \Big) &= \E\Big[ \mathrm{Var}\Big( \thetav\;\Big\vert\; \yv,\lambdav \Big) \;\Big\vert\; \yv = \yvc \Big] + \mathrm{Var}\Big( \E\Big[ \thetav\;\Big\vert\;\yv,\lambdav \Big]\;\Big\vert\; \yv=\yvc\Big)\nonumber\\
  &= \E\Big[ \left( 1 + n\lambdav^{-1} \right)^{-1} \;\Big\vert\; \yv=\yvc\Big] + \left( \overline{\yvc}^2 \right)\mathrm{Var} \Big[ (\lambdav + n)^{-1}\;\Big\vert\; \yv=\yvc \Big].\label{eq:ex2_post_var_cond}
\end{align}
We expect $\mathrm{Var} \left[ (\lambdav + n)^{-1}\mid\yv=\yvc \right]\approx 0$ when $\mu_{\lambdav}\uparrow\infty$ (since $(\lambdav + n)^{-1} \approx 0$) and when $\mu_{\lambdav}\downarrow 0$ (since $(\lambdav+n)^{-1}\approx n^{-1}$). Combining this with the fact that $\lambda\mapsto\left( 1 + n\lambda^{-1} \right)^{-1}$ is increasing with range $(0,1)$, we approximate
\begin{equation}
  \label{eq:ex2_id_conc1}
  \begin{split}
  \mathrm{Var}\Big( \thetav\;\Big\vert\;\yv = \yvc \Big) &\approx \E\left[ \left( 1 + n\lambdav^{-1} \right)^{-1}\;\Big\vert\; \yv=\yvc \right]\\ &\approx \begin{cases}
    0 < \mathrm{Var}_{p_{\b}}(\thetav\mid \yv = \yvc), & \mu_{\lambdav} \ll 1\\
    1 > \mathrm{Var}_{p_{\b}}(\thetav\mid \yv=\yvc), & \mu_{\lambdav} \gg 1.
  \end{cases}
\end{split}
\end{equation}
Therefore, we predict $\imim<\imim_{\b}$ for $\mu_{\lambdav}$ large, and $\imim > \imim_{\b}$ for $\mu_{\lambdav}$ small. 

While, in the previous example, $\lambdav$ primarily influenced the location of the data, $\lambdav$ now primarily influences the scale of the data. Consequently, we must account for this scale explicitly when considering falsifiability. To do this, we study the ratio
\begin{equation}
  \label{eq:ex2_false_ratio}
  \frac{\mathrm{Var}(\E\left[ \yv^{\mathrm{rep}}\mid \thetav,\lambdav \right]\mid \yv = \yvc)}{\mathrm{Var}(\yv^{\mathrm{rep}}\mid \yv=\yvc)},
\end{equation}
i.e. the ratio of the spread of predicted values to the overall spread of replicated or future data given observed data $\yv$. Larger values of \eqref{eq:ex2_false_ratio} indicate a greater variety of predictions, which we associated with lower $\fmim$. Applying the law of total variance to the denominator, and recognizing $\E\left[ \yv^{\mathrm{rep}}\mid \thetav,\lambdav \right] = \thetav$ and $\mathrm{Var}\left[ \yv^{\mathrm{rep}}\mid \thetav,\lambdav \right] = \lambdav$, we reexpress \eqref{eq:ex2_false_ratio} for $\mu_{\lambdav}$ small or large as
\begin{equation}
  \label{eq:ex2_false_ratio2}
  \frac{\mathrm{Var}(\thetav\mid\yv = \yvc)}{\mathrm{Var}(\thetav\mid\yv = \yvc) + \E\left[ \lambdav \mid \yv=\yvc \right]} \approx \frac{\E\left[ (1 + n\lambdav^{-1})^{-1} \mid \yv= \yvc \right]}{\E\left[ (1 + n\lambdav^{-1})^{-1} \mid \yv = \yvc \right] + \E\left[ \lambdav \mid \yv = \yvc \right]},
\end{equation}
where the approximation follows by plugging in \eqref{eq:ex2_post_var_cond} for $\mathrm{Var}(\thetav\mid\yv = \yvc)$ with the assumption $\mathrm{Var}\left[ (\lambda + n)^{-1}\mid \yv=\yvc \right] \approx 0$ as above.
Since $\lambdavc \mapsto (1 + n\lambdavc^{-1})^{-1}$ is sublinear, and $\mu_{\lambdav} = \E\lambdav$, we expect \eqref{eq:ex2_false_ratio2} and \eqref{eq:ex2_false_ratio} to be decreasing in $\mu_{\lambdav}$. Thus, we predict $\fmim < \fmim_{\b}$ for $\mu_{\lambdav} \ll 1$ and $\fmim > \fmim_{\b}$ for $\mu_{\lambdav} \gg 1$.

While exact calculation of $\imim$ and $\fmim$ is not possible in this model, the low dimensional parameter space allows accurate numerical approximation. For $n=2$, Figure \ref{fig:ex3_fig} plots $\imim - \imim_{\b}$ and $\fmim - \fmim_{\b}$ against $\mu_{\lambdav}$. As in the previous example, and as expected from our calculations, $\imim$ and $\fmim$ trade off as $\mu_{\lambdav}$ varies. But unlike in Section \ref{ex:1}, we have $\imim > \imim_{\b}$ for $\mu_{\lambdav}$ small and $\fmim > \fmim_{\b}$ for $\mu_{\lambdav}$ large, as we predicted. 

  \begin{figure}[t]
     \centering
     \includegraphics[scale=0.66]{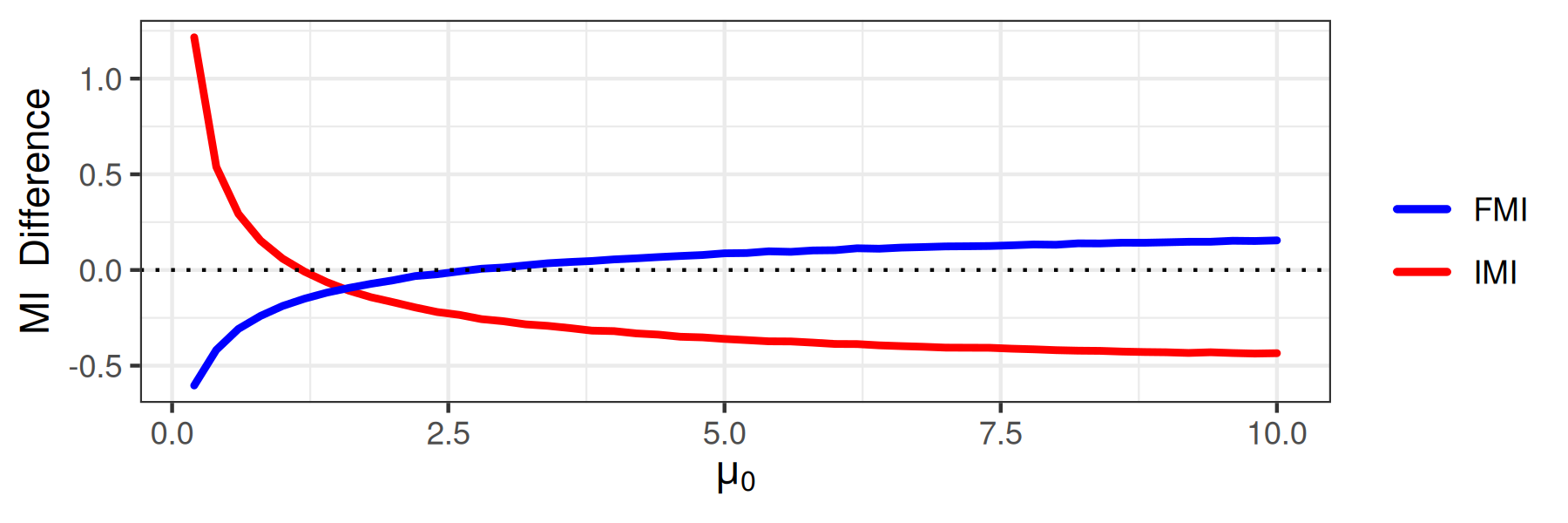}
     \caption{Change in \imi (red) and \fmi (blue) from base to expanded model against $\mu_{\lambdav}$. \imi increases relative to the base model for $\mu_{\lambdav}$ smaller and decreases for $\mu_{\lambdav}$ larger, whereas \fmi increases relative to base model for $\mu_{\lambdav}$ larger and decreases for $\mu_{\lambdav}$ smaller.}
     \label{fig:ex3_fig}
   \end{figure}

\subsubsection{Comparison with Theorem \ref{thm:tradeoff}}

   \begin{figure}[t]
  \centering
  \includegraphics[scale=0.2]{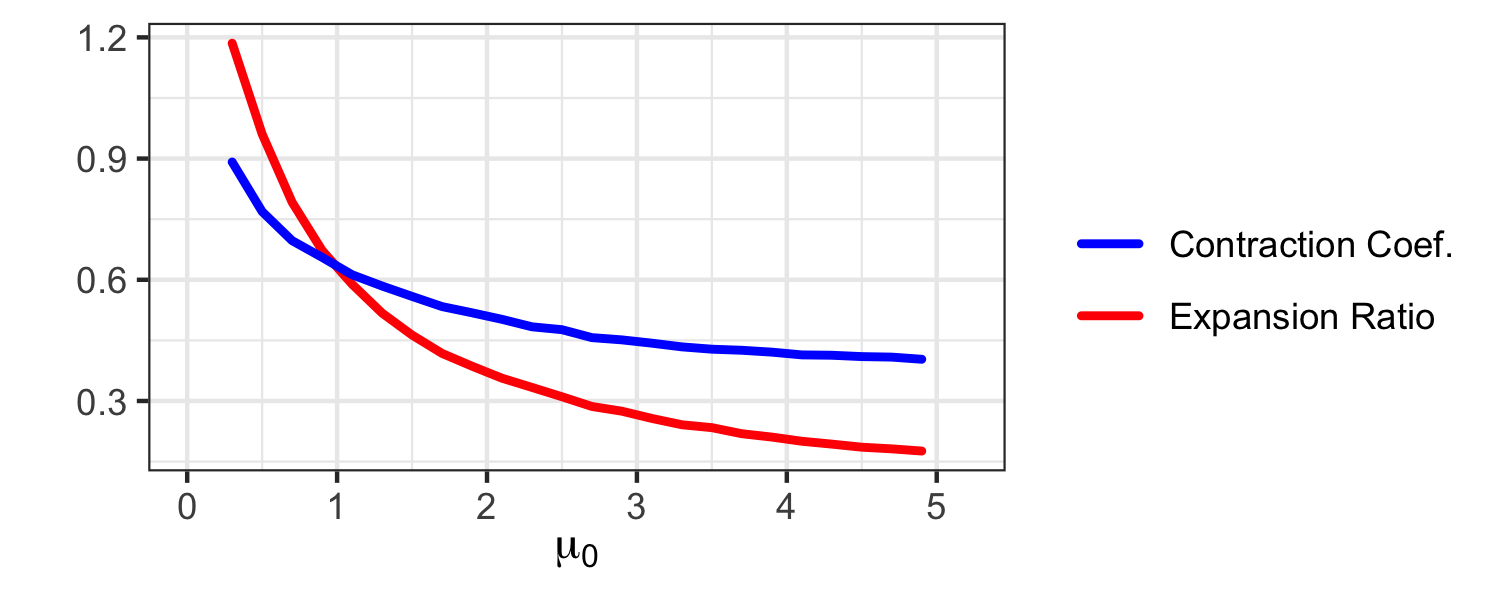}
  \caption{$\expratio$ (red) and $\eta^*_{\p}$ (blue) against $\mu_{\lambdav} = \E \mathrm{Var}([\yv]_i\mid\thetav, \lambdav)$. For $\mu_{\lambdav}$ sufficiently small, $\expratio > \eta_{\p}$, and Theorem \ref{thm:tradeoff} implies a trade-off between the \imi and \fmi. For larger $\mu_{\lambdav}$, $\expratio < \eta_{\p}$, and Theorem \ref{thm:tradeoff} does not apply.}
  \label{fig:ex3_tradeoff}
  \end{figure}

  Since at least one of $\imim < \imim_{\b}$ and $\fmim < \fmim_{\b}$ holds for all $\mu_{\lambdav}$, we again have that all expanded models satisfy the conclusions of Theorem \ref{thm:tradeoff}. For small or large values of $\mu_{\lambdav}$, only one of these inequalities holds, and the conclusions are thus satisfied nontrivially.
And as in the previous example, we find that Theorem \ref{thm:tradeoff} predicts this trade-off for some but not all $\mu_{\lambdav}$. In Figure \ref{fig:ex3_tradeoff}, we see that $\expratio > \eta^*_p$ for $\mu_{\lambdav} \leq 1$ (and so the condition \eqref{eq:tradeoff_condition} correctly predicts the trade-off). However, we have $\expratio < \eta^*_{\p}$ for $\mu_{\lambdav} > 1$ (and so the trade-off holds without the condition \eqref{eq:tradeoff_condition}).

\subsection{Example 3: Hierarchical Prior}\label{ex:4}

\subsubsection{Models}
  Consider data $\yv = [\yv^1\;\yv^2]$ partitioned into groups, $\yv^1,\yv^2\in\RR^n$. We define a base model with separate means for each group.
  \begin{equation}
     \label{eq:ex_counter_base}
     \begin{split}
       [\yv^j]_i\mid\thetav&\stackrel{iid}{\sim}\mathrm{normal}\left([\thetav]_j, 1\right) \text{ for } 1\leq i\leq n\text{ and } j=1,2,\\
       [\thetav]_1,[\thetav]_2 &\stackrel{iid}{\sim} \mathrm{normal}\left(0, \sigma^2_0\right).
     \end{split}
   \end{equation}
   We expand this model by allowing the correlation $\lambdav = \mathrm{Cor}\left( [\thetav]_1,[\thetav]_2 \right)$ to be positive:
   \begin{equation}
     \label{eq:ex_counter_exp}
     \thetav\mid\lambdav \sim\mathrm{normal}\left(0,\frac{\sigma_0^2}{\sqrt{1-\lambdav^2}}\left[
         \begin{matrix}
           1 & \lambdav\\
           \lambdav & 1
         \end{matrix}
       \right]\right), \qquad \lambdav \sim\mathrm{Beta}\left(99\mu_{\lambdav}, 99(1-\mu_{\lambdav})\right).
   \end{equation}
   The likelihood is unchanged, and the prior on $\thetav$ is parametrized so that the entropy $h_{\p(\thetav\mid\lambdav = \lambdavc)}(\thetav)$ is independent of $\lambdavc$.
   In this parametrization, $\lambdav$ has prior mean $\mu_{\lambdav}$ and standard deviation $\sqrt{\mu_{\lambdav} (1-\mu_{\lambdav})} / 10$. Conditioning on $\lambdav = 0$ recovers the base model.

\subsubsection{Effect on Identifiability/Falsifiability}

In the base model, $[\thetav]_i$ is only informed by $\yv^i$ for $i=1,2$. In the expanded model however, conditioning on any $\lambdav>0$, $\yv^1$ is informative about $[\thetav]_2$ and $\yv^2$ about $[\thetav]_1$. This data sharing between the two groups, which increases with $\lambdav$, improves identification of both means by increasing the amount of data which can be used to estimate them. We check this intuition by partially computing $\imim$, finding that
\begin{equation}
  \label{eq:counter_imim_form}
  \imim = \E\Bigg[ \underbrace{\frac{1}{2}\log\left(\frac{(1 + n\sigma_0^2\sqrt{1 - \lambdav^2})^2 - \lambdav^2}{1 - \lambdav^2}\right)}_{\phi(\lambdav)}\Bigg] + \MI\left(\lambdav, \yv\right).
\end{equation}
The function $\phi(\lambdavc)$ measures the identifiability of $\thetav$ conditional on $\lambdav$. As we expect, $\phi(\lambdavc)$ is increasing in $\lambdavc$. Furthermore, $\phi(0) = \imim_{\b}$, so \eqref{eq:counter_imim_form} implies that
\begin{equation}
  \imim\geq \E [\phi(\lambdav)] \geq \phi(0) = \imim_{\b},
\end{equation}
 where we have used that $\MI\left(\lambdav,\yv\right)\geq 0$. Numerically estimating $\MI\left(\lambdav,\yv\right)$, we plot $\imim - \imim_{\b}$ against $\mu_{\lambdav} = \E\lambdav$ in Figure \ref{fig:ex_counter_fig}, which confirms that $\imim - \imim_{\b}\geq 0$ increases with $\mu_{\lambdav}$.

Turning to $\fmim$, we note that $[\thetav]_1$ and $[\thetav]_2$ are a priori independent in the base model. In the expanded model, as $\lambdav\uparrow 1$, we will have $[\thetav]_1 \approx [\thetav]_2$ with overwhelming probability, effectively reducing the degrees of freedom in specifying the sampling distribution from 2 to nearly 1. Because falsifiability is directly connected to sampling distribution variety (as we argued in Section \ref{sec:quantity_defs}), we expect $\fmim$ to increase with $\mu_{\lambdav}$.
As above, we check this by decomposing $\fmim$:
\begin{equation}
  \label{eq:counter_fmim_form}
  \fmim = \E\Bigg[\underbrace{\frac{1}{2}\log\left(\frac{[1 + n\sigma_0^2\sqrt{1 - \lambdav^2}]^2 - \lambdav^2}{[1 + 2n\sigma_0^2\sqrt{1 - \lambdav^2}]^2 - \lambdav^2}\right)}_{\psi(\lambdav)}\Bigg] - \MI\left(\lambdav, \yrep\mid \yv\right).
\end{equation}
Now the function $\psi(\lambdavc)$ measures the posterior variety of sampling distributions conditional on $\lambdav$. For $\sigma_0 = 1$, $\psi(\lambdavc)$ again increases with $\lambdavc$, as expected. (In general, $\psi(\lambdavc)$ will be increasing for $n\sigma_0^2$ sufficiently large.) As before, we also have $\psi(0) = \fmim_{\b}$.

But unlike the analysis of $\imim$, these properties do not establish $\fmim \geq \fmim_{\b}$, as the mutual information on the right-hand side of \eqref{eq:counter_fmim_form} is now \textit{subtracted}. This reflects a small defect in the intuition sketched above. Increasing $\lambdav$ does effectively reduce the degrees of freedom in specifying $\thetav$, but adding $\lambdav$ as a new parameter also adds a degree of freedom to the model. 
However, we may still expect our original intuition to be close to correct --- since $\lambdav$ is independent of $\yv$ given $\thetav$, the added degree of freedom should be ``small''. 

To confirm this, we numerically estimate $\MI\left(\lambdav, \yrep\mid \yv\right)$ and plot $\fmim - \fmim_{\b}$ against $\mu_{\lambdav}$ in Figure \ref{fig:ex_counter_fig}. We again see $\fmim - \fmim_{\b}$ increasing with $\mu_{\lambdav}$, and exceeding $0$ for $\mu_{\lambdav}$ sufficiently large, as guessed. Because both \imi and \fmi increase in the expanded model for $\mu_{\lambdav}$ large enough, the trade-off \eqref{eq:id_fa_trade} does not occur in this example.

    \begin{figure}[t]
     \centering
     \includegraphics[scale=0.66]{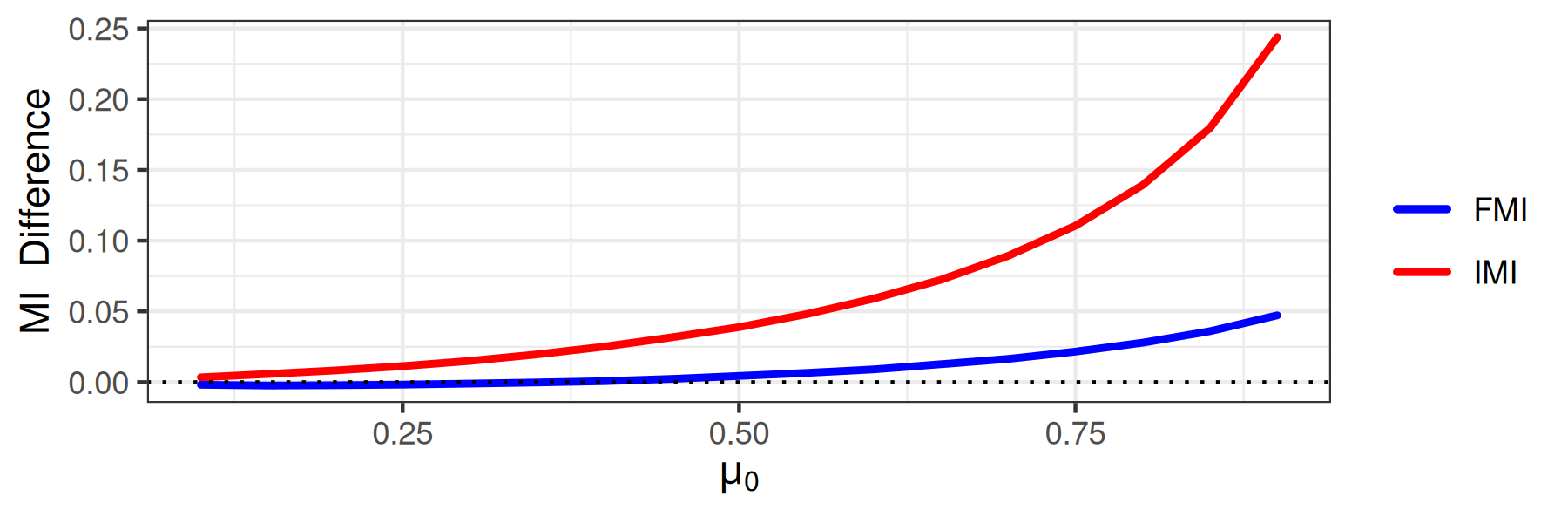}
     \caption{Change in \imi (red) and \fmi (blue) from base to expanded model against $\mu_{\lambdav}$. For $\mu_{\lambdav}$ large enough, both the \imi and \fmi improve from base to expanded model.}
     \label{fig:ex_counter_fig}
   \end{figure}

\subsubsection{Comparison with Theorem \ref{thm:tradeoff}}

Since $\lambdav$ does not enter the likelihood of the expanded model, this is a prior expansion, which means that $\yv \ind \lambdav \;\mid\; \thetav$ for almost all $\thetav$.
As we remarked after Theorem \ref{thm:tradeoff}, this conditional independence implies that the expansion ratio vanishes, i.e. $\expratio = 0$. On the other hand, we have $\eta_{\p} > 0 = \expratio$, so the condition \eqref{eq:tradeoff_condition} does not hold, as required by our observation in Figure \ref{fig:ex_counter_fig} that both $\imim > \imim_{\b}$ and $\fmim > \fmim_{\b}$.

While the previous two examples showed that the hypotheses of Theorem \ref{thm:tradeoff} can be stronger than needed for the conclusion \eqref{eq:id_fa_trade} to hold, this example shows that we cannot expect the trade-off \eqref{eq:id_fa_trade} to hold in general without a nontrivial lower bound on $\expratio$. Informally, we may conclude that when an expanded model is \textit{less} complex than its base model (as in this example), we may expect both identifiability and falsifiability to improve.

\section{Conclusions} \label{sec:conclusion}

When constructing a model, a statistician should balance various desiderata, including:
\begin{enumerate}[label=(\arabic*)]
\item predictions compatible with what is known about the world;
\item inferences sufficiently well identified to support nontrivial conclusions;
\item fitness checks powerful enough to reveal frictions between model and data.
\end{enumerate}

When fitness checks reveal deficiencies in the current model, item $(1)$ is no longer satisfied, and a better model should be sought. In practice, this is often achieved by expanding the current model. If such expansions are not accompanied by sufficiently strong prior information (e.g. in the form of prior dependence between parameters), then our results demonstrate that a tension may easily arise between items $(2)$ and $(3)$ as the model dimension and complexity grows.

This tension also underscores the importance of developing models with the particular goals of a given data analysis in mind. In particular cases, identifiability may be sacrificed (e.g. in problems where the model itself is of independent scientific interest), whereas, in other cases, falsifiability may be sacrificed (e.g. in some ``pure'' prediction problems). This work suggests that understanding these trade-offs early in the modeling process may be an essential aid in navigating the space of potential models.

Our analysis is limited in some important respects, however. The mutual information-based quantities that we study involve averages over the assumed distribution of the data $\p(\yvc)$. This both abstracts our results away from the details of a particular dataset and fails to account for the realities of misspecification (in which case our observed data will not be drawn from the assumed $\p(\yvc)$). Furthermore, our information-theoretic quantities are difficult to compute in realistically complex models, limiting our ability to monitor them in practice. Thus, a complete understanding of the challenges of model expansion will require both more realistic assumptions and more computationally tractable measures.

\bibliographystyle{plainnat} 
\bibliography{model-expansion-bayes}

\appendix

\section{Summary of Definitions and Notation}
The following table summarizes the definitions and interpretations of the most important quantities used throughout the paper.
\begin{center}
  \begin{tabular}{ |m{2cm}|m{1.7cm}|m{6.8cm}|m{1.05cm}| }\hline
    Quantity & Notation & Definition & Details\\\hline
    \small\makecell{ Identifiability\\ mutual\\  information\\ (IMI)}&
    \makecell{\gape{$\imim_{\b}$}\\ \gape{$\imim$} }&
    \makecell{\gape{$\MI_{\p_{\b}}\left( \yv, \thetav \right)$}\\\gape{$\MI_{\p}\left( \yv, \thetav \right)$}}&
    \hyperlink{imi}{(A)}
    \\\hline
    \small\makecell{Falsifiability\\ mutual\\ information\\ (FMI)}&
    \makecell{\gape{$\fmim_{\b}$} \\ \gape{$\fmim$}}&
    \makecell{\gape{$\MI_{\p_{\b}}\left(\yrep, \thetav\mid\yv \right)$} \\ \gape{$\MI_{\p}\left(\yrep, (\thetav, \lambdav)\mid\yv \right)$}}& 
    \hyperlink{fmi}{(B)}
    \\\hline
    \small\makecell{Expansion\\ ratio}&
    \makecell{\gape{$\expratio$}}&
    \makecell{\gape{$\MI_{\p}\left( \yrep, \lambdav\mid \yv,\thetav \right)/\fmim_{\b}$}} &
    \hyperlink{erat}{(C)}
    \\\hline
    \small\makecell{Contraction\\ coefficient}&
    \makecell{\gape{$\eta_{\p}$}}&
    \makecell{\gape{$\Gamma\left(\MI_{\p}\left(\yv,\thetav\right),\E_{\p}\|\thetav\|^2\right) / \MI_{\p}\left(\yv,\thetav\right)$}}&
    \hyperlink{ccoef}{(D)}
    \\\hline
    \small\makecell{$F_I$ curve}&
    \makecell{\gape{$\Gamma(t,\gamma)$}} & \makecell{\gape{$\displaystyle \begin{aligned}
    &\underset{\palt}{\sup}\left\lbrace \MI_{\paux}\left( \yv, \yv' \right) \;\Bigg\vert\; \begin{aligned}\MI_{\palt}\left(\yv',\thetav\right) &\leq t,\\ \E\|\thetav\|^2&\leq \gamma  \end{aligned}  \right\rbrace\\
    &\text{where }\paux = \p(\yvc\mid\thetavc)\palt(\yvc'\mid\thetavc)\palt(\thetavc)
    \end{aligned}$}}&
    \hyperlink{ficurve}{(E)}
    \\\hline
    \small\makecell{Contraction\\ coefficient\\ lower bound}&
    \makecell{\gape{$\eta^*_{\p}$}}&
    \makecell{\gape{$\MI_{\p}\left( \yv,\yv' \right) / \MI_{\p}\left( \yv,\thetav \right)$}}&
    \hyperlink{ccoeflb}{(F)}
    \\\hline
    \small\makecell{Posterior\\ sampling\\ divergence}&
    \makecell{\gape{$\psdm_{\p}(\yvc)$}}&
    \makecell{\gape{$\E\Big[ \kldiv{\p(\yrepc\mid\thetav)}{\p(\yrepc\mid\yvc)}\;\Big\vert\; Y = y \Big]$}}&
    \hyperlink{psd}{(G)}
    \\\hline
    \small\makecell{Excess\\ posterior\\ sampling\\ divergence}&
    \makecell{\gape{$\mathrm{EPSD}_{\p}(\yvc)$}}&
    \makecell{\gape{$\E\Big[ \kldiv{\p(\yrepc\mid \thetav,\lambdav)}{\p(\yrepc\mid \yvc,\thetav)} \;\Big\vert\; \yv = \yvc\Big]$}}&
    \hyperlink{epsd}{(H)}
    \\\hline
    \small\makecell{Posterior\\ variance\\ of $p$-values}&
    \makecell{\gape{$\pvp_T(\yvc)$}}&
    \makecell{\gape{$\mathrm{Var} \Big[ p_T(\yvc\mid\thetav) \;\Big\vert\; \yv = \yvc\Big]$}}&
    \hyperlink{pvp}{(I)}
    \\\hline
    \end{tabular}
\end{center}
\noindent
\textbf{\hypertarget{imi}{(A)} Identifiability mutual information.}
The identifiability mutual information (IMI) represents the (average) amount of information gained from observing data, where information is measured by entropy. Smaller values indicate that we expect the posterior to be very close to the prior, and for the corresponding parameters to therefore be weakly identified. The IMI can be expressed as follows as a mutual information or, equivalently, an entropy difference:
\[
  \imim_{\p}(\absparam) = \MI_{\p}\left( \yv,\absparam \right) = h_{\p(\absparam)}\left(\absparam\right) - h_{\p(\absparam, \yvc)}\left(\absparam\mid \yv\right),
\]
where $\absparam$ is any vector of model parameters. We usually take $\absparam = \thetav$, and in this case we write $\imim_{\p}$ for $\imim_{\p}(\thetav)$. When there is a particular pair of base model $\p_{\b}$ and expanded model $\p$ which is clear from context, we will also write $\imim$ for $\imim_{\p}$ and $\imim_{\b}$ for $\imim_{\p_{\b}}$.

\noindent
\textbf{\hypertarget{fmi}{(B)} Falsifiability mutual information.}
The falsifiability mutual information (FMI) represents the (negative) expected posterior uncertainty about, or variability of, the sampling distribution $\p(\yrep\mid \absparam)$. The FMI may be expressed as a conditional mutual information or, equivalently, as an expected divergence between a randomly chosen sampling distribution $\p(\yrepc\mid \absparam)$ and its posterior average $p(\yrepc\mid\yv)$.
\[
  \fmim_{\p}\left( \absparam \right) = -\MI\left( \yrep, \absparam\mid \yv \right) = -\E_{\yv\sim\p(\yvc)}\left[ \E\Big[ \kldiv{\p(\yrepc\mid\absparam)}{\p(\yrepc\mid\yv)}\;\Big\vert\; \yv \Big] \right].
\]
The inner expectation on the right-hand side is also the posterior sampling divergence (see \hyperlink{psd}{(G)} below). We define the FMI with the negative sign above so that, like the IMI, larger values are associated with better outcomes (e.g. better falsifiability). When the FMI is lower (i.e. more negative), the model becomes harder to check (see Section 3.2 for details). We typically consider the FMI with $\absparam$ taken to be a model's full parameter vector. Thus, for a base model $\p_{\b}$, we take $\absparam = \thetav$ and write $\fmim_{\b}$ for $\fmim_{\p_{\b}}(\thetav)$. For an expanded model $\p$, we take $\absparam = (\thetav,\lambdav)$ and write $\fmim = \fmim_{\p}((\thetav,\lambdav))$.

\noindent
\textbf{\hypertarget{erat}{(C)} Expansion ratio.}
For a base model $\p_{\b}$ and expanded model $\p$, the expansion ratio $\expratio$ measures the amount of uncertainty about the sampling distribution $\p(\yrep\mid \thetav,\lambdav)$ due to the expanded parameters $\lambdav$ relative to the uncertainty about the sampling distribution in the base model. This is given by the ratio
\[
\expratio = \frac{\MI\left( \yrep,\lambdav\mid \yv,\thetav \right)}{-\fmim_{\b}}.
\]
In the numerator, we condition on $\thetav$ in order to count only the posterior uncertainty about $p(\yrep\mid\thetav,\lambdav)$ which is due to or explainable by uncertainty in $\lambdav$. The numerator is also the expected value of the excess posterior sampling divergence (see \hyperlink{epsd}{(H)} below).

\noindent
\textbf{\hypertarget{ccoef}{(D)} Contraction coefficient.}
The contraction coefficient $\eta_{\p}$ is a measure of association between data $\yv$ and parameters $\thetav$, analogous to a nonlinear and multidimensional generalization of the squared correlation. The mutual information $\MI_{\p}\left( \yv,\thetav \right)$ depends on both the sampling distribution $\p(\yvc\mid\thetavc)$ and the prior $\p(\thetavc)$. The contraction coefficient measures how much of this mutual information can be attributed to the sampling distribution $\p(\yvc\mid\thetav)$. In other words, the contraction coefficient is a measure of how much information the distribution $\p(\yvc\mid\thetavc)$ preserves about its parameter $\thetavc$. The contraction coefficient is defined as 
\[
\eta_p = \frac{\Gamma\left( \MI_{\p}(\yv,\thetav), \E_{\p}\|\thetav\|^2 \right)}{\MI_{\p}(\yv,\thetav)},
\]
where $\Gamma$ is the $F_I$ (see \hyperlink{ficruve}{(E)} below). We have that $\eta_p\in [0,1]$, with $\eta_p$ larger when $\yv\sim\p(\yvc\mid\thetavc)$ is a better predictor of $\thetavc$.

\noindent
\textbf{\hypertarget{ficurve}{(E)} $F_I$ curve.}
For a sampling distribution $\p(\yvc\mid\thetavc)$, the $F_I$ curve, introduced in \cite{SDPI_Nonlinear} to study strong data processing inequalities, is given as follows.
\[
\Gamma(t;\gamma) = \sup_{\palt(\yvc',\thetavc)}\left\lbrace \MI_{\paux}\left(\yv, \yv'\right)\;\Bigg\vert\;\MI_{\palt}\left(\yv',\thetav\right) \leq t, \E_{\thetav\sim\palt(\thetavc)}\|\thetav\|^2\leq \gamma  \right\rbrace,
\]
where, for any model $\palt(\yvc',\thetavc)$, we define
\[
\paux(\yvc,\yvc',\thetavc) = \p(\yvc\mid\thetavc)\palt(\yvc'\mid\thetavc)\palt(\thetavc).
\]
Since, for any model $\paux$, $\yv'\to\thetav\to\yv$ forms a Markov chain, the data processing inequality implies that $\MI(\yv,\yv')\leq\MI(\yv',\thetav)$. The gap in this inequality depends on how much information the distribution $\p(\yvc\mid\thetavc)$ preserves about its parameter $\thetavc$. The $F_I$ curve exploits this idea to give a measure of the ``strength'' of $\p(\yvc\mid\thetavc)$ by finding the minimal gap given only the constraining parameters $t$ and $\gamma$ (and hence removing the dependence on the particular alternative model $\palt$). In particular, given any alternative model, we have the strong data processing inequality
\[
\MI_{\palt}(\yv,\yv')\leq F_I\left( \MI_{\palt}(\yv',\thetav), \E_{\palt}\|\thetav\|^2 \right).
\]

\noindent
\textbf{\hypertarget{ccoeflb}{(F)} Contraction coefficient lower bound.}
The quantity $\eta^*_{\p} = \MI(\yv,\yv')/\MI(\yv,\thetav)$ lower bounds the contraction coefficient $\eta_{\p}$ (as is easily seen from the definition of the $F_I$ curve, see \hyperlink{ficurve}{(I)} above). Unlike the contraction coefficient, $\eta_{\p}^*$ depends directly on the prior distribution $\p(\thetavc)$, and therefore cannot be interpreted as a property of the sampling distribution $\p(\yvc\mid\thetavc)$. However, in some cases where $\eta_{\p}$ has no closed form, $\eta_{\p}^*$ may be computed analytically. We use $\eta_{\p}^*$ throughout the examples in Section 4 for this reason (though we note that in most realistic cases, neither $\eta_{\p}$ nor $\eta_{\p}^*$ will be computable).

\noindent
\textbf{\hypertarget{psd}{(G)} Posterior sampling divergence.}
The posterior sampling divergence (PSD) measures, for a particular observed data set $\yvc$, the posterior variability in the sampling distribution $\p(\yrepc\mid \absparam)$. This variability is quantified using the KL divergence between a randomly chosen sampling distribution $\p(\yrepc\mid \absparam)$ and its posterior average $\p(\yrepc\mid\yvc)$. Specifically, we define
\[
  \psdm(\yvc) = \E\Big[ \kldiv{\p(\yrepc\mid\absparam)}{\p(\yrepc\mid\yvc)}\;\Big\vert\; \yv = \yvc \Big].
\]
The PSD can be viewed as a statistic-free analog of the posterior variance of $p$-values (see \hyperlink{pvp}{(I)} below). Averaging the PSD over the prior predictive distribution $\p(\yvc)$ and negating yields the falsifiability mutual information (see \hyperlink{psd}{(B)} above).

\noindent
\textbf{\hypertarget{epsd}{(H)} Excess posterior sampling divergence.}
For an expanded model $\p(\yvc,\thetavc,\lambdavc)$ and fixed observed dataset $\yvc$, the excess posterior sampling divergence (EPSD) measures the amount of posterior uncertainty about the sampling distribution $\p(\yrep\mid\thetavc,\lambdavc)$ which is explainable by excess uncertainty in $\lambdavc$ (i.e. uncertainty conditional on $\thetav$). We define this using the following divergence.
\[
  \mathrm{EPSD}(\yvc) = \E\Big[ \kldiv{\p(\yrepc\mid \thetav,\lambdav)}{\p(\yrepc\mid \yvc,\thetav)} \;\Big\vert\; \yv = \yvc\Big].
\]
We note that $\p(\yrepc\mid \yvc,\thetav) = \E_{\lambdav\sim p(\lambdavc\mid\yvc,\thetav)}\left[ \p(\yrepc\mid \thetav, \lambdav) \right]$.

\noindent
\textbf{\hypertarget{pvp}{(I)} Posterior variance of $p$-values.} 
For fixed observed data $\yvc$ and a test statistic $T$, the posterior variance of $p$-values (PVP) measures the posterior variance of the conditional $p$-values $p_T(y\mid\absparam) = \P\left( |T(\yrep)|\geq |T(\yvc)|\mid \absparam \right)$. Specifically, we have
\[
\mathrm{PVP}_T(\yvc) = \mathrm{Var}\left[ p_T(\yvc\mid\thetav)\mid \yv = \yvc \right].
\]
Larger values of $\mathrm{PVP}_T$ indicate higher posterior uncertainty about how well the model fits the data, as measured by the statistic $T$. When $\mathrm{PVP}_T$ is higher, we also expect the distribution of the (unconditional) posterior predictive $p$-value $p_T(\yv)$ to be more concentrated around $1/2$ and less uniform (when sampling $\yv\sim\p(\yvc)$).

\section{Basic Quantities and Relations from Information Theory}\label{app:it_overview}
In this section, we provide statements of the basic results from information theory that we make use of throughout this paper. Proofs of these results can be found in any introductory course on information theory. We state all results in terms of conditional entropies and mutual informations when appropriate since these contain the non-conditional statements as special cases. First we review relevant definitions. We state these results in terms of abstract random variables $\uv$, $\vv$, and $\wv$, which we substitute with (combinations of) the model quantities $\yv$, $\yrep$, $\thetav$, and $\lambdav$ in the main defintions and results of this work.
\begin{definition}[Basic Quantities of Information Theory]
  Let $\p(v,u)$ be some joint model. Then the entropy of $\uv$ is defined as
  \begin{equation}
    \label{eq:entropy_def_app}
    h_{\p(u)}(\uv) = -\E_{U\sim \p(u)}\log \p(\uv).
  \end{equation}
  The conditional entropy of $\uv$ given $\vv$ is just the average entropy of the conditional distributions:
  \begin{equation}
    \label{eq:cond_entropy_def}
    h_{\p(v,u)}\left( \uv\mid\vv \right) = \E_{V\sim \p(v)}\left[ h_{\p(u\mid \vv)}\left( \uv \right) \right] = -\E_{(V,U) \sim \p(v,u)}\log \p(\uv\mid\vv).
  \end{equation}
  The mutual information between $\uv$ and $\vv$ is the amount by which entropy is expected to decrease after conditioning $\vv$:
  \begin{equation}
    \label{eq:mi_def_app}
    \MI_{\p}\left( \vv,\uv \right) = h_{\p(u)}\left( \uv \right) - h_{\p(v,u)}\left( \uv\mid\vv \right).
  \end{equation}
  Finally, if we extend our joint model to $\p(u,v,w)$ where $\wv$ is any additional quantity, then the conditional mutual information given $\wv$ is just the difference of the corresponding conditional entropies:
  \begin{equation}
    \label{eq:cmi_def_app}
    \MI_{\p}\left( \vv,\uv\mid\wv \right) = h_{\p(w,u)}\left( \uv\mid\wv \right) - h_{\p(v,w,u)}\left( \uv\mid\vv,\wv \right).
  \end{equation}
\end{definition}

The first important result allows us to break up an entropy or mutual information expression additive over the components of vector arguments.
\begin{lemma}[Chain Rule for Entropy and Mutual Information]\label{lem:it_chain}
  Let $\p(u,v,w)$ be a joint model and suppose that $\uv$ can be partitioned into sub-vectors $\left( \uv_1,\ldots,\uv_m \right)$ for some $m\geq 1$. Then we have that
  \[
    h_{\p}\left( \uv\mid\vv \right) = \sum_{i=1}^m h_{\p}\left( \uv_i\mid \uv_{< i},\vv \right),
  \]
  where $\uv_{<i}=\left( \uv_1,\ldots,\uv_{i-1} \right)$ for $i\geq 2$, and $\uv_{<1} = \{\}$. Furthermore, we have that
  \[
    \MI_{\p}\left( \vv,\uv\mid\wv \right) = \sum_{i=1}^m \MI\left( \vv,\uv_i\mid \wv,\uv_{<i} \right).
  \]
\end{lemma}

Next, it can be useful to express the (conditional) mutual information in terms of the KL divergence, which quantifies discrepancy between two probability distributions $\p_1(v)$ and $\p_2(v)$. In particular, the KL divergence is given as
\[
\kldiv{\p_1(v)}{\p_2(v)} = \E_{\vv\sim \p_1(v)}\log\left[ \frac{\p_1(\vv)}{\p_2(\vv)} \right].
\]
The mutual information can be related to the KL divergence in two different ways.
\begin{lemma}[Mutual Information as KL Divergence]\label{lem:it_div}
  Let $\p(u,v,w)$ be a joint model. Then we have
  \begin{align*}
    \MI(\vv,\uv\mid\wv) &= \E\left[ \kldiv{\p(u,v\mid \wv)}{\p(u\mid \wv) p( v\mid \wv)} \right]\\
    &= \E\left[ \kldiv{p(u\mid\vv,\wv)}{p(u\mid\wv)} \right].
  \end{align*}
\end{lemma}

It is of fundamental importance that the KL divergence is always nonnegative, which follows by an application of Jensen's inequality.
\begin{lemma}[Nonnegativity of the KL Divergence]
  For any densities $\p_1(v)$ and $\p_2(v)$, we have
  \[
    \kldiv{\p_1(v)}{\p_2(v)}\geq 0
  \]
  with equality if and only if $\p_1(v)=\p_2(v)$ $\p_1$-almost surely.
\end{lemma}
This immediately implies nonnegativity of the mutual information, and in turn the fact that
\[
  h_{\p(u,v)}\left(\vv\mid\uv  \right)\leq h_{\p(v)}\left( \vv \right)
\]
for any joint distribution $\p(u,v)$.

It is often useful to know how these metrics operate under certain transformations of the random quantities in terms of which they are defined. This is characterized by the following result.
\begin{lemma}[Entropy and Mutual Information Under Transformation]
  Let $\Am\in\RR^{d\times d}$ be any invertible matrix and let $\vv^\prime=\Am\vv$. Then we have
  \[
    h_{\p(v^\prime)}\left( \vv^\prime \right) = h_{\p(v)}\left( \vv \right) + \log \left\lvert \det\Am \right\rvert.
  \]
  Furthermore, if $\vv^\prime = \vv + c$ for any $c\in\RR$, then $h\left( \vv \right)=h\left( \vv^\prime \right)$. Thus, the entropy is invariant under translations and orthogonal transformations. The mutual information satisfies the stronger property of invariance under arbitrary smooth reparametrizations of the individual arguments. Specifically, let $\phi,\psi$ be smooth, invertible maps, and define $\vv^\prime=\phi(\vv)$ and $\uv^\prime = \psi(\uv)$. Then we have that
  \[
    \MI(\vv^\prime,\uv^\prime\mid\wv) = \MI(\vv,\uv\mid\wv).
  \]
\end{lemma}

The general behavior of the mutual information under potentially noninvertible transformations is characterized by the data processing inequality.
\begin{lemma}[Data Processing Inquality]
  Let $\p(u,v,w)$ be any distribution, and suppose that $\uv$ and $\wv$ are conditionally independent given $\vv$. Then we have that
  \[
    \MI(\uv,\vv) \geq \MI(\uv,\wv).
  \]
  In particular, the above inequality holds if $\wv = \psi(\vv)$ for any function $\psi$.
\end{lemma}

Finally, certain distributions maximize the entropy under certain conditions. For our purposes, it suffices to note that normal distributions on $\RR^d$ maximize the entropy among all distributions with fixed covariance matrix and support equal to $\RR^d$.
\begin{lemma}[Maximum Entropy of Normal]
  Let $\p(v)$ be any probability distribution supported on $\RR^d$, and let $\p_Z(\vv)$ be a normal distribution with any mean and covariance matrix equal to the covariance $\covm$ of $\p_Z(v)$. Then we have
  \[
    h_{\p(v)}\left( \vv \right)\leq h_{\p_Z(v)}\left( \vv \right) = \frac{1}{2}\log\left( \det\left( 2\pi e \covm \right) \right).
  \]
\end{lemma}

\section{Proof of Lemma 1}
\begin{proof}
  First, we note that the Bretagnolle-Huber inequality tells us that
  \begin{equation}
    \label{eq:BH_ineq}
    d^2_{\mathrm{TV}}\left(\p\left(\yrepc\mid\thetavc\right), \p\left(\yrepc\mid\yvc\right)\right) \leq 1 - \exp\left( - \kldiv{\p\left(\yrepc\mid\thetavc\right)}{\p\left(\yrepc\mid\yvc\right)}\right),
  \end{equation}
  where $d_{\mathrm{TV}}$ is the total variation distance between probability distributions.  Noting that $x \mapsto 1 - \exp(-x)$ is concave, Jensen's inequality and the definition of the posterior sampling divergence along with \eqref{eq:BH_ineq} gives us that
  \begin{equation}
    \label{eq:PSD_ub}
    \E \left[ d^2_{\mathrm{TV}}\left(\p\left(\yrepc\mid\thetav\right), \p\left(\yrepc\mid\yvc\right)\right) \mid \yv=\yvc \right] \leq 1 - \exp\left( - \mathrm{PSD}(\yvc)\right).
  \end{equation}
  Expanding the definition of the total variation distance, we can lower bound the left-hand side of \eqref{eq:PSD_ub} as
  \begin{align}
    \E&\left[ d^2_{\mathrm{TV}}\left(\p\left(\yrepc\mid\thetav\right), \p\left(\yrepc\mid\yvc\right)\right) \mid\yv = \yvc \right]\nonumber\\
                          & = \E\left[ \sup_{E}\left[\P_{\p(\yrepc\mid\thetav)}(E) - \P_{\p(\yrepc\mid\yvc)}(E)\right]^2\;\Big\vert\; \yv=\yvc \right]\nonumber\\
                          &\geq \E\Big[ \left[\P_{p(\yrepc\mid\thetav)}(|T(\yrep)| \geq |T(\yvc)|) - \P_{p(\yrepc\mid\yvc)}(|T(\yrep)| \geq |T(\yvc)|)\right]^2\;\Big\vert\;\yv=\yvc \Big]\label{eq:tv_lb1}\\
                          &= \E\left[ \left[p_T(\yvc\mid\thetav) - p_T(\yvc)\right]^2\;\Big\vert\;\yv=\yvc \right]\label{eq:tv_lb2}\\
                          &= \E\left[ \left[p_T(\yvc\mid\thetav) - \E\left[ p_T(\yvc\mid\thetav)\mid\yv=\yvc \right]\right]^2\;\Big\vert\;\yv=\yvc \right]\label{eq:tv_lb3}\\
    &= \mathrm{Var}\left[p_T(\yvc\mid\thetav)\;\Big\vert\;\yv=\yvc\right],\nonumber
  \end{align}
where \eqref{eq:tv_lb1} follows by plugging in the particular event $E = \{|T(\yrep)| \geq |T(\yv)|\}$, \eqref{eq:tv_lb2} follows by the definitions of the posterior predictive and conditional $p$-values $p_T(\yvc)$ and $p_T(\yvc\mid\thetavc)$ respectively, and \eqref{eq:tv_lb3} follows from the identity $p_T(\yvc) = \E\left[ p_T(\yvc\mid\thetav)\mid \yv=\yvc \right]$. Combining the above with \eqref{eq:PSD_ub} gives the desired conclusion.
\end{proof}

\section{Proof of Theorem 1}

In order to prove our main result on the tradeoff of identifiability and falsifiability under model expansion, we first establish the following lemma.

\begin{lemma}\label{lem:main_it_lem}
  Assume that $\E_{\p}\thetav^2<\infty$, and let $\eta_{\p}$ be the contraction coefficient for the sampling distribution $\p\left(\yvc\mid\thetavc\right)$. Then, we have that
  \begin{equation}
    \label{eq:cmi_mi_lb}
    \MI\left(\yv,\thetav\right)\geq \MI\left(\yrep,\thetav\mid \yv\right) \geq \left(1 - \eta_{\p}\right)\MI\left(\yv, \thetav\right).
  \end{equation}
\end{lemma}

\begin{proof}
  The first inequality follows simply as
  \begin{align}
    \MI\left(\yv,\thetav\right) &= h\left(\yv\right) - h\left(\yv\mid\thetav\right)\nonumber\\
                                &= h\left(\yrep\right) - h\left(\yrep\mid\thetav\right)\label{eq:lmi3}\\
                                &\geq h\left(\yrep\mid\yv\right) - h\left(\yrep\mid\thetav\right)\label{eq:lmi4}\\
                                &= \MI\left(\yrep,\thetav\mid \yv\right),\label{eq:lmi5}
  \end{align}
  where \eqref{eq:lmi3} follows from the fact that $(\yv,\thetav) \stackrel{d}{=} (\yrep,\thetav)$, \eqref{eq:lmi4} follows from the fact that conditioning reduces entropy, and \eqref{eq:lmi5} follows from the fact that $\yv$ and $\yrep$ are conditionally independent given $\thetav$.

  Now the second inequality follows as
  \begin{align}
    \MI\left(\yrep,\thetav\mid\yv\right) &= \MI\left(\yv,\thetav\right) - \MI\left(\yv,\yrep\right)\label{eq:lmi1}\\
                                         &=\MI\left(\yv,\thetav\right)\left[1 - \frac{\MI\left(\yv, \yrep\right)}{\MI\left(\yv,\thetav\right)}\right]\nonumber\\
    &\geq \MI\left(\yv,\thetav\right)\left[1 - \eta_{\p}\right].\label{eq:lmi2}
  \end{align}
  Here, \eqref{eq:lmi1} follows from the chain rule for mutual information and that fact that $\MI\left(\yrep, (\yv,\thetav)\right) = \MI\left(\yrep, \thetav\right) = \MI\left(\yv,\thetav\right)$ by the conditional independence of $\yrep$ and $\yv$ given $\thetav$. Then \eqref{eq:lmi2} follows by the fact that $p(\yrep,\thetav)$ lies within the set of distributions
  over which the $F_I$ curve is defined.
\end{proof}

With this lemma established, we can now prove the main tradeoff result.

\begin{proof}
  Recall our main assumption:
  \begin{equation}
    \label{eq:tmi1}
    \frac{\MI\left(\yrep, \lambdav\mid \yv,\thetav\right)}{\MI_{\b}\left(\yrep, \thetav\mid \yv\right)} \geq \eta_{\p}.
  \end{equation}

  First suppose that falsifiability is nondecreasing in the sense that
  \begin{equation}
    \label{eq:tmi2}
    \fmim = -\MI\left(\yrep, (\thetav,\lambdav)\mid \yv\right) \geq -\MI_{\b}\left(\yrep,\thetav\mid \yv\right) = \fmim_{\b}.
  \end{equation}

Then we have that
\begin{align}
  \imim = \MI\left(\yv,\thetav\right) &\leq (1-\eta_{\p})^{-1}\MI\left(\yrep, \thetav\mid \yv\right)\label{eq:tmi3}\\
                              &= (1-\eta_{\p})^{-1}\left[\MI\left(\yrep,(\thetav,\lambdav)\mid \yv\right) - \MI\left(\yrep,\lambdav\mid \yv,\thetav\right) \right] \label{eq:tmi4}\\
                              &\leq (1-\eta_{\p})^{-1}\left[\MI_{\b}\left(\yrep,\thetav\mid \yv\right) - \eta_{\p}\MI_{\b}\left(\yrep,\thetav\mid\yv\right)\right]\label{eq:tmi5}\\
                              &= \MI_{\b}\left(\yrep, \thetav\mid\yv\right)\nonumber\\
  &\leq \MI_{\b}\left(\yv,\thetav\right) = \imim_{\b}.\label{eq:tmi6}
\end{align}

In the above, \eqref{eq:tmi3} follows directly from Lemma \ref{lem:main_it_lem}, \eqref{eq:tmi4} follows from the chain rule for (conditional) mutual information, \eqref{eq:tmi5} follows from the nondecreasing falsification assumption \eqref{eq:tmi2} (for the first term) as well as the core assumption \eqref{eq:tmi1} (for the second term), and finally \eqref{eq:tmi6} follows again by Lemma \ref{lem:main_it_lem}.

Next, suppose that identification is nondecreasing in the sense that
\begin{equation}
  \label{eq:tmi7}
  \imim = \MI\left(\yv,\thetav\right) \geq \MI_{\b}\left(\yv,\thetav\right) = \imim_{\b}.
\end{equation}

With this assumption, we have
\begin{align}
 - \fmim = \MI\left(\yrep, (\thetav,\lambdav)\mid \yv\right) &= \MI\left(\yrep,\thetav\mid\yv\right) + \MI\left(\yrep,\lambdav\mid\yv,\thetav\right)\label{eq:tmi8}\\
                                                    &\geq (1-\eta_{\p}) \MI\left(\yv,\thetav\right) + \eta_{\p}\MI_{\b}\left(\yrep,\thetav\mid\yv\right)\label{eq:tmi9}\\
                                                    &\geq (1-\eta_{\p}) \MI_{\b}\left(\yv,\thetav\right) + \eta_{\p}\MI_{\b}\left(\yrep,\thetav\mid\yv\right)\label{eq:tmi10}\\
                                                    &\geq (1-\eta_{\p}) \MI_{\b}\left(\yrep,\thetav\mid\yv\right) + \eta_{\p}\MI_{\b}\left(\yrep,\thetav\mid\yv\right)\label{eq:tmi11}\\
  &=  \MI_{\b}\left(\yrep,\thetav\mid\yv\right) = -\fmim_{\b}\nonumber.
\end{align}

In the above, \eqref{eq:tmi8} follows from the chain rule for conditional mutual information, \eqref{eq:tmi9} follows from Lemma \ref{lem:main_it_lem} (for the first term) and from the core assumption \eqref{eq:tmi1} (for the second term), \eqref{eq:tmi10} follows from the nondecreasing identification assumption \eqref{eq:tmi7}, and finally \eqref{eq:tmi11} follows again from Lemma \ref{lem:main_it_lem}, completing the proof of the theorem.

\end{proof}

\section{Decompositions of \texorpdfstring{$\imim$}{IMI} and \texorpdfstring{$\fmim$}{FMI}}

We present a pair of simple decompositions for the $\imim$ and $\fmim$. These decompositions establish a certain bias towards falling $\imim$ and $\fmim$ under model expansion. They will also be used in the next section to provide another perspective on the analysis of the unknown variance example presented in Section 4.2.

\begin{lemma}[Decomposition of $\imim$]\label{lem:mi_decomp}
For base model $\p_{\b}$ and expanded model $\p$, we have
\begin{equation}
  \label{eq:mi_decomp}
  \imim = \imim_{\b} + \ddelid + \daddid,
\end{equation}
where we define
\begin{align*}
  \ddelid &= \E_{\lambdav \sim \p(\lambdavc)}\left[\MI_{\p(\yvc,\thetavc\mid\lambdav)}\left( \yv,\thetav \right) - \MI_{\p(\yvc,\thetavc\mid\lambdav = \lambdavc_0)}\left( \yv,\thetav\right)\right]\\
  \daddid &= \MI_{\p(\yvc,\thetavc,\lambdavc)}\left(\thetav,\lambdav\right) - \MI_{\p(\yvc,\thetavc,\lambdavc)}(\thetav,\lambdav\mid\yv).
\end{align*}
\end{lemma}
\begin{proof}
Using the chain rule for mutual information (Lemma \ref{lem:it_chain}) twice, we have that
\begin{equation}
  \label{eq:imi_chain}
  \MI\left(\yv, \lambdav\mid\thetav\right) + \MI\left(\yv, \thetav\right) = \MI\left(\yv,(\thetav,\lambdav)\right) = \MI\left(\yv, \thetav\mid\lambdav\right) + \MI\left(\yv,\lambdav\right).
\end{equation}
Rearraging, this is equivalent to
\begin{equation}
  \label{eq:imi_rearr}
  \imim = \MI\left(\yv, \thetav\right) = \MI\left(\yv, \thetav\mid\lambdav\right) + \left[\MI\left(\yv,\lambdav\right) - \MI\left(\yv, \lambdav\mid\thetav\right)\right].
\end{equation}
Again, a double application of Lemma \ref{lem:it_chain} gives
\begin{equation}
  \label{eq:imi_chain2}
  \MI\left(\yv, \lambdav\right) + \MI\left(\thetav,\lambdav\mid\yv\right) = \MI\left((\yv,\thetav),\lambdav\right) = \MI\left(\thetav,\lambdav\right) + \MI\left(\yv,\lambdav\mid\thetav\right).
\end{equation}
Rearranging, we see that
\begin{equation}
  \label{eq:imi_rearr2}
  \MI\left(\yv,\lambdav\right) - \MI\left(\yv, \lambdav\mid\thetav\right) = \MI\left(\thetav,\lambdav\right) - \MI\left(\thetav,\lambdav\mid\yv\right) = \daddid.
\end{equation}
Combining \eqref{eq:imi_rearr} and \eqref{eq:imi_rearr2}, we have that
\begin{equation}
  \imim = \MI\left(\yv, \thetav\mid\lambdav\right) + \daddid.
\end{equation}
Now, it is easy to see by Lemma \ref{lem:it_div} that $\MI\left(\yv,\thetav\mid\lambdav\right) =\E_{\lambdav \sim \p(\lambdavc)}\MI_{\p(\yvc,\thetavc\mid\lambdav)}\left(\yv, \thetav\right)$. Adding and subtracting $\MI_{\p(\yvc,\thetavc\mid\lambdav = \lambdav_0)}\left(\yv, \thetav\right) = \imim_{\b}$ thus gives
\begin{equation}
  \imim = \imim_{\b} + \E_{\lambdav\sim \p(\lambdavc)}\left[\MI_{\p(\yvc,\thetavc\mid\lambdav)}\left(\yv, \thetav\right) - \MI_{\p(\yvc,\thetavc\mid\lambdav = \lambdavc_0)}\left(\yv, \thetav\right)\right] + \daddid.
\end{equation}
Recognizing the second term on the right-hand side as the definition of $\ddelid$ gives the desired result.
\end{proof}

The $\ddelid$ term is the difference in the amount of information about $\thetav$ contained in $\yv$ given $\lambdav$ (averaging over $p(\lambdav)$) and given $\lambdav = \lambdavc_0$. Whether $\ddelid$ is positive or negative depends on the model. The $\daddid$ term is the difference in the amount of information $\lambdav$ provides about $\thetav$ before and after observing $\yv$.
The effect of $\daddid$ depends on the prior. If $\thetav$ and $\lambdav$ are a priori independent, i.e. if $\p(\thetavc,\lambdavc) = \p(\thetavc)\p(\lambdavc)$, then $\MI_{\p(\yv,\thetav,\lambdav)}\left(\lambdav,\thetav\right) = 0$. In this case,
\begin{equation*}
  \daddid = - \MI_{\p(\yvc,\thetavc,\lambdavc)}(\thetav,\lambdav\mid\yv) \leq 0,
\end{equation*}
with equality if and only if $\p(\thetavc,\lambdavc\mid\yvc) = \p(\thetavc\mid\yvc)p(\lambdavc\mid \yvc)$ for almost all $\yvc$. If we do have this posterior independence, then in fact we also have $\ddelid=0$, and \eqref{eq:mi_decomp} implies that $\imim$ will be unchanged from base to expanded model. Such posterior independence only occurs when the likelihood also factorizes over $\thetavc$ and $\lambdavc$ -- a rare occurrence in natural model expansions. Thus, when we have prior independence between $\thetav$ and $\lambdav$, we regard $\daddid$ as creating a downward bias on identifiability.
When $\thetav$ and $\lambdav$ are \textit{not} independent under the prior, the effect of $\daddid$ is more subtle. Some such expansions are able to escape the conclusions of Theorem 1 (e.g. the example presented in Section 4.3).

The corresponding decomposition for the $\fmim$ has a similar form.

\begin{lemma}[Decomposition of $\fmim$]\label{lem:fmi_decomp}
  For base model $\p_{\b}$ and expansion $\p$, we have
\begin{equation}
  \label{eq:fmi_decomp}
  \fmim = \fmim_{\b} +\ddelfa + \daddfa,
\end{equation}
where we define
\begin{align*}
  \ddelfa &= -\E_{\lambdav \sim \p(\lambdavc)}\left[ \MI_{\p(\yvc,\yrepc,\thetavc\mid\lambdav)}(\yrep,\thetav\mid\yv) - \MI_{\p(\yvc,\yrepc,\thetavc\mid\lambdav = \lambdavc_0)}(\yrep,\thetav\mid\yv) \right],\\
  \daddfa &= -\MI_{\p}\left(\yrep,\lambdav\mid\yv\right).
\end{align*}
\end{lemma}
\begin{proof}
Using the chain rule for mutual information (Lemma \ref{lem:it_chain}), we have that
\begin{equation}\label{eq:fmi_chain}
  \fmim = -\MI\left(\yrep, (\thetav,\lambdav)\mid \yv\right) = - \MI\left(\yrep, \thetav\mid \yv,\lambdav\right) - \MI\left(\yrep,\lambdav\mid\yv\right).
\end{equation}
Note that the last term is exactly $\daddfa$.
Using Lemma \ref{lem:it_div}, it is easy to see that
\begin{equation}
  \label{eq:it_cond_rep}
  \MI\left(\yrep, \thetav\mid \yv,\lambdav\right) = \E_{\lambdav\sim\p(\lambdavc)}\left[\MI_{\p(\yvc,\yrepc,\thetavc\mid\lambdav)}\left(\yrep,\thetav\mid\yv\right)\right].
\end{equation}
Now we add and subtract $-\MI_{\p(\yvc,\yrepc,\thetavc\mid\lambdav = \lambdavc_0)}\left(\yrep,\thetav\mid\yv\right)$ and note that this term is exactly $-\MI_{\b}\left(\yrep,\thetav\mid\yv\right) =  \fmim_{\b}$ by the definition of model expansion. The right-hand side of \eqref{eq:fmi_chain} becomes
\begin{equation}
  \fmim_{\b} -\E_{\lambdav\sim\p(\lambdavc)}\left[\MI_{\p(\yvc,\yrepc,\thetavc\mid\lambdav)}\left(\yrep,\thetav\mid\yv\right) - \MI_{\p(\yvc,\yrepc,\thetavc\mid\lambdav = \lambdavc_0)}\left(\yrep,\thetav\mid\yv\right)\right] + \daddfa.
\end{equation}
Recognizing the second term on the right-hand side as $\ddelfa$ gives the result.
\end{proof}

As with $\ddelid$ in \eqref{eq:mi_decomp}, $\ddelfa$ can be positive or negative depending on the model. Unlike $\daddid$ however, we always have $\daddfa \leq 0$, so $\fmim$ is always biased downward.

\section{Computations and Additional Details for Examples in Section 4}

We present computations and additional details for the information-theoretic quantities given in our three worked examples.

\subsection{Example 1: Linear Regression}

In this example, the multivariate normal form of the posterior permits explicit calculations of our information-theoretic quantities.

\subsubsection{Expression for \texorpdfstring{$\imim$}{IMI}}

First we derive the expression 
\[
  \imim = \imim_{\b}- \frac{1}{2}\log\frac{1+\tau}{\pi(\tau) + \tau}.
\]
For the expanded regression model, the marginal posterior of $\thetav$ is normal with covariance matrix $\left(\left[\Xm^T\Xm + \tau\imat_{k+1}\right]^{-1}\right)_{-(k+1)}$, where, for a matrix $\mathbb{M}\in\RR^{(k+1)\times (k+1)}$, $\mathbb{M}_{-(k+1)}$ denotes the $k\times k$ submatrix obtained from $\mathbb{M}$ by removing the $(k+1)^{\mathrm{th}}$ row and column, and where $\imat_{k+1}$ is the $(k+1)\times (k+1)$ identity matrix. Writing $\Pi_{k+1} = \Xm^T\Xm + \tau\imat_{k+1}$, the IMI for this expanded model is given as
\begin{equation}
  \label{eq:imi_reg_exp}
  \imim = h(\thetav) - h(\thetav\mid\yv) = -\frac{k}{2}\log\tau - \frac{1}{2}\log \det  \left(\left[\Pi_{k+1}^{-1}\right]_{-(k+1)}\right).
\end{equation}

In order to simplify the determinant on the right-hand side, we first express $\Pi_{k+1}$ as a block matrix:
\begin{equation}
  \label{eq:lambda_block}
  \Pi_{k+1} = \left[\begin{matrix}
  \Pi_k & v \\ v^T & 1 + \tau
  \end{matrix}\right],
\end{equation}
where $\Pi_k = \left[\Pi_{k+1}\right]_{-(k+1)} = \Xm_{\b}^T\Xm_{\b} + \tau\imat_k$, and $v = \Xm_{\b}^T\xv_{k+1}$. We note that $\Pi_{k}$ is the posterior precision of $\thetav$ in the base regression model. Now, inverting the block matrix \eqref{eq:lambda_block} yields
\begin{equation}
  \label{eq:marg_cov_expression}
  \left[\Pi_{k+1}^{-1}\right]_{-(k+1)} = \left[\Pi_k - \frac{vv^T}{1+\tau}\right]^{-1}. 
\end{equation}  
Now, usual expressions for the determinant of a block matrix yield
\begin{equation}
  \label{eq:block_det}
  \begin{split}
    \det\left(\Pi_{k+1}\right) &= \det\left(\Pi_k\right)\left(\pi(\tau) + \tau\right),\\
    \det\left(\Pi_{k+1}\right) &= \det\left(\Pi_k - \frac{vv^T}{1 + \tau}\right)\left(1 + \tau\right),
  \end{split}
\end{equation}
where $\pi(\tau)$ is given by
\begin{equation}
  \label{eq:pi_tau}
  \pi(\tau) = \xv_{k+1}^T\left[\imat_k - \Xm_{\b}\left(\Xm_{\b}^T\Xm_{\b} + \tau\imat_k\right)^{-1}\Xm_{\b}^T\right]\xv_{k+1}.
\end{equation}
Combining \eqref{eq:block_det} with \eqref{eq:marg_cov_expression} and using the fact that $\det(\mathbb{M}^{-1}) = 1/\det(\mathbb{M})$ for all invertible matrices $\mathbb{M}$, we arrive at
\begin{equation}
  \label{eq:det_expression}
  \det\left(\left[\Pi_{k+1}^{-1}\right]_{-(k+1)}\right) = \det\left(\Pi_k^{-1}\right)\left[\frac{1 + \tau}{\pi(\tau) + \tau}\right].
\end{equation}
 Now combining \eqref{eq:det_expression} with \eqref{eq:imi_reg_exp}, we get
 \begin{align}
  \imim &= -\frac{k}{2}\log\tau -\frac{1}{2}\log\det\left(\Pi_k^{-1}\right) - \frac{1}{2}\log\frac{1+\tau}{\pi(\tau) + \tau}\nonumber\\
  &= \imim_{\b}- \frac{1}{2}\log\frac{1+\tau}{\pi(\tau) + \tau}\label{eq:base_imim},
 \end{align}
 where \eqref{eq:base_imim} follows from the fact that the base model posterior is multivariate normal with covariance matrix $\Pi_k^{-1}$.

\subsubsection{Expression for \texorpdfstring{$\fmim$}{FMI}}

Next, we derive the following expression for $\fmim$:
\[
 \fmim = \fmim_{\b} - \frac{1}{2}\log\frac{2\pi(\tau) + \tau}{\pi(\tau) + \tau}.
\]
First, we decompose the $\fmim$ of the expanded model as follows:
\begin{equation}
  \label{eq:decomp_fmim}
  \fmim = -h\left((\thetav,\lambdav)\mid \yv\right) + h\left((\thetav,\lambdav)\mid \yv,\yrep\right) = -\frac{1}{2}\log\det\left(\Pi_{k+1}^{-1}\right) + \frac{1}{2}\log\det\left(\widetilde{\Pi}_{k+1}^{-1}\right),
\end{equation}
where $\Pi_{k+1}$ is defined as in the last section, and $\widetilde{\Pi}_{k+1} = 2\Xm^T\Xm + \tau\imat_{k+1}$. Now, again using block determinant expressions, we have
\begin{equation}
  \label{eq:fmim_block_det}
  \begin{split}
    \det\left(\Pi_{k+1}\right) &= \det\left(\Pi_k\right)\left(\pi(\tau) + \tau\right),\\
    \det\left(\widetilde{\Pi}_{k+1}\right) &= \det\left(\widetilde{\Pi}_k\right)\left(2\pi(\tau/2) + \tau\right),
  \end{split}
\end{equation}
where $\widetilde{\Pi}_k = 2\Xm_{\b}^T\Xm_{\b} + \tau\imat_k$. Combining \eqref{eq:fmim_block_det} with \eqref{eq:decomp_fmim}, we get
\begin{align}
  \fmim &= -\frac{1}{2}\log\det\left(\Pi_k^{-1}\right) + \frac{1}{2}\log\det\left(\widetilde{\Pi}^{-1}_{k}\right) - \frac{1}{2}\log\frac{2\pi(\tau/2) + \tau}{\pi(\tau)+\tau}\\
  &=\fmim_{\b} - \frac{1}{2}\log\frac{2\pi(\tau/2) + \tau}{\pi(\tau)+\tau}.\label{eq:fmim_conc},
\end{align}
where \eqref{eq:fmim_conc} follows from the fact that $\Pi_k^{-1}$ and $\widetilde{\Pi}_k^{-1}$ are the covariance matrices of $\p_{\b}(\thetavc\mid\yvc)$ and $\p_{\b}(\thetavc\mid\yvc,\yrepc)$ respectively.

\subsubsection{Expression for \texorpdfstring{$\expratio$}{expansion ratio}}

We now derive the expression $\expratio = \frac{1}{k}$ in the special case where $\Xm^T\Xm = \imat_{k+1}$. First, recall that the expansion ratio is given by the formula
\begin{equation}
  \label{eq:er_recall}
  \expratio = \frac{\MI\left(\lambdav, \yrep\mid \thetav, \yv\right)}{\MI_{\b}\left(\thetav,\yrep\mid\yv\right)}.
\end{equation}
We can express the numerator of \eqref{eq:er_recall} as 
\begin{equation}
  \label{eq:expr_num}
  \MI\left(\lambdav, \yrep\mid \thetav, \yv\right) = h\left(\lambdav\mid \yv,\thetav\right) - h\left(\lambdav\mid\yv,\yrep,\thetav\right) = \frac{1}{2}\E\log\frac{\mathrm{Var}(\lambdav\mid\yv,\thetav)}{\mathrm{Var}(\lambdav\mid\yv,\yrep,\thetav)}.
\end{equation}
Recall that if $\Pi$ is a precision matrix for a random vector $\xv\in\RR^{k+1}$, then
\begin{equation}
  \label{eq:prec_cond_var_rel}
  1/\Pi_{k+1,k+1} = \mathrm{Var}([\xv]_{k+1}\mid [\xv]_1,\ldots,[\xv]_k).
\end{equation} 
Applying this identity to $\xv = (\thetav,\lambdav)$ with $\Pi$ equal to $\Pi_{k+1}$ and $\widetilde{\Pi}_{k+1}$ (defined in the previous section), we obtain
\begin{align}
  \MI\left(\lambdav, \yrep\mid \thetav, \yv\right) &= \frac{1}{2}\E\log\frac{\left[\widetilde{\Pi}_{k+1}\right]_{k+1,k+1}}{\left[\Pi_{k+1}\right]_{k+1,k+1}}\nonumber\\
  &= \frac{1}{2}\log\frac{2\xv_{k+1}^T\xv_{k+1} + \tau}{\xv_{k+1}^T\xv_{k+1} + \tau}\label{eq:rnum_1}\\
  &= \frac{1}{2}\log\frac{2 + \tau}{1 + \tau}.\label{eq:rnum_2},
\end{align}
where \eqref{eq:rnum_1} follows by the definitions of $\Pi_{k+1}$ and $\widetilde{\Pi}_{k+1}$, and \eqref{eq:rnum_2} follows from the fact that $\Xm^T\Xm = \imat_{k+1}$.

Next we express the denominator of \eqref{eq:er_recall} as 
\begin{equation}
  \label{eq:expr_den}
  \MI_{\b}\left(\thetav,\yrep\mid\yv\right) = h_{\b}\left(\thetav\mid \yv\right) - h_{\b}\left(\thetav\mid\yv,\yrep\right) = \frac{1}{2}\log\frac{\det\widetilde{\Pi}_k}{\det\Pi_k},
\end{equation}
where we have used the fact that $\Pi_k$ and $\widetilde{\Pi}_k$ are the posterior precision matrices for the base model given $\yv$ and $(\yv,\yrep)$ respectively, as well as the fact that $\det\left(\mathbb{M}^{-1}\right) = 1 / \det\left(\mathbb{M}\right)$ for any invertible $\mathbb{M}$.

Now we observe that
\begin{equation}
  \label{eq:rexp_den_simpl}
  \Pi_k = \Xm_{\b}^T\Xm_{\b} + \tau\imat_k = (1+\tau)\imat_k,\qquad \widetilde{\Pi}_k = 2\Xm_{\b}^T\Xm_{\b} + \tau\imat_k = (2 + \tau)\imat_k
\end{equation}
Combining \eqref{eq:rexp_den_simpl} with \eqref{eq:expr_den} yields
\begin{equation}
  \label{eq:rexp_den_final}
  \MI_{\b}\left(\thetav,\yrep\mid\yv\right) = \frac{1}{2}\log\frac{(2+\tau)^k}{(1+\tau)^k} = \frac{k}{2}\log\frac{2+\tau}{1 + \tau}.
\end{equation}

Finally, combining \eqref{eq:rexp_den_final}, \eqref{eq:rnum_2}, and \eqref{eq:er_recall} yields
\begin{equation}
  \label{eq:rexp_final}
  \expratio = \frac{1}{k},
\end{equation} 
as claimed.

\subsubsection{Expression for \texorpdfstring{$\eta^*_{\p}$}{contraction coefficient lower bound}}

Finally, we show that $\eta^*_{\p} = 2 - \frac{\log(1 + 2\tau^{-1})}{\log(1 + \tau^{-1})}$. First, recall that
\begin{equation}
  \label{eq:eta_recall}
  \eta^*_{\p} = \frac{\MI(\yv,\yrep)}{\MI(\thetav,\yv)} = 1 - \frac{\MI\left(\thetav,\yrep\mid\yv\right)}{\MI\left(\thetav,\yv\right)},
\end{equation}
where the second inequality follows from the chain rule for mutual information and the fact that $\MI\left(\yrep, (\yv,\thetav)\right) = \MI\left(\yrep, \thetav\right) = \MI\left(\yv,\thetav\right)$ by the conditional independence of $\yrep$ and $\yv$ given $\thetav$.

For the term $\MI\left(\thetav,\yrep\mid\yv\right)$ on the right-hand side of \eqref{eq:eta_recall}, we observe that the chain rule for mutual information gives us that
\begin{equation}
  \label{eq:eta_num_decomp}
  \MI\left(\thetav,\yrep\mid\yv\right) = \MI\left((\thetav,\lambdav), \yrep\mid \yv\right) - \MI\left(\lambdav, \yrep\mid \yv,\thetav\right).
\end{equation}
Using the fact that $\Pi_{k+1}$ and $\widetilde{\Pi}_{k+1}$ are the posterior precision matrices of $(\thetav,\lambdav)$ given $\yv$ and $(\yv,\yrep)$ respectively, as well as the relation between the precision and conditional variances \eqref{eq:prec_cond_var_rel}, we express \eqref{eq:eta_num_decomp} as
\begin{align}
  \MI\left(\thetav,\yrep\mid\yv\right) &= \frac{1}{2}\log\frac{\det\widetilde{\Pi}_{k+1}}{\det\Pi_{k+1}} -\frac{1}{2}\log\frac{2 + \tau}{1+\tau}\nonumber\\
  &= \frac{k}{2}\log\frac{2 + \tau}{1 + \tau}\label{eq:eta_num_simpl}\\
  &= \frac{k}{2}\log\frac{1 + 2\tau^{-1}}{1 + \tau^{-1}},\label{eq:eta_num_final}
\end{align}
where \eqref{eq:eta_num_simpl} follows from the fact that $\Pi_{k+1} = (1 + \tau)\imat_{k+1}$, and $\widetilde{\Pi}_{k+1} = (2 + \tau)\imat_{k+1}$.

Next for the term $\MI\left(\thetav,\yv\right)$ on the right-hand side of \eqref{eq:eta_recall}, we proceed similarly, writing
\begin{align}
  \MI\left(\thetav,\yv\right) &= \MI\left((\thetav,\lambdav), \yv\right) - \MI\left(\lambdav, \yv\mid\thetav\right)\nonumber\\
  &= \frac{1}{2}\log\frac{\det\Pi_{k+1}}{\det (\tau\imat_{k+1})} - \frac{1}{2}\log\frac{1 + \tau}{\tau}\label{eq:eta_den_decomp}\\
  &= \frac{k+1}{2}\log\frac{1+\tau}{\tau} - \frac{1}{2}\log\frac{1+\tau}{\tau}\nonumber\\
  &= \frac{k}{2}\log\frac{1+\tau}{\tau}\nonumber\\
  &= \frac{k}{2}\log (1 + \tau^{-1}),\label{eq:eta_denom_final}
\end{align}
where \eqref{eq:eta_den_decomp} follows from the fact that $\tau\imat_{k+1}$ and $\Pi_{k+1}$ are the prior and posterior precision matrices of $(\thetav,\lambdav)$ respectively, and $\tau$ and $1+\tau$ are prior and posterior conditional variances of $\lambdav$ given $\thetav$.

Finally, combining \eqref{eq:eta_denom_final}, \eqref{eq:eta_num_final}, and \eqref{eq:eta_recall}, we obtain
\begin{equation}
  \eta^*_{\p} = 1 - \frac{\log\left(1 + 2\tau^{-1}\right) - \log\left(1 + \tau^{-1}\right)}{\log\left(1 + \tau^{-1}\right)} = 2 - \frac{\log\left(1 + 2\tau^{-1}\right)}{\log\left(1 + \tau^{-1}\right)},
\end{equation}
as claimed.

\subsection{Example 2: Unknown Variance}

For this example, we illustrate the effect of expansion graphically using the decompositions presented in the previous section.

\begin{figure}[H]
  \centering
  \begin{tabular}{m{1.8in} m{0.4in} m{1.8in}}
    \includegraphics[scale = 0.42]{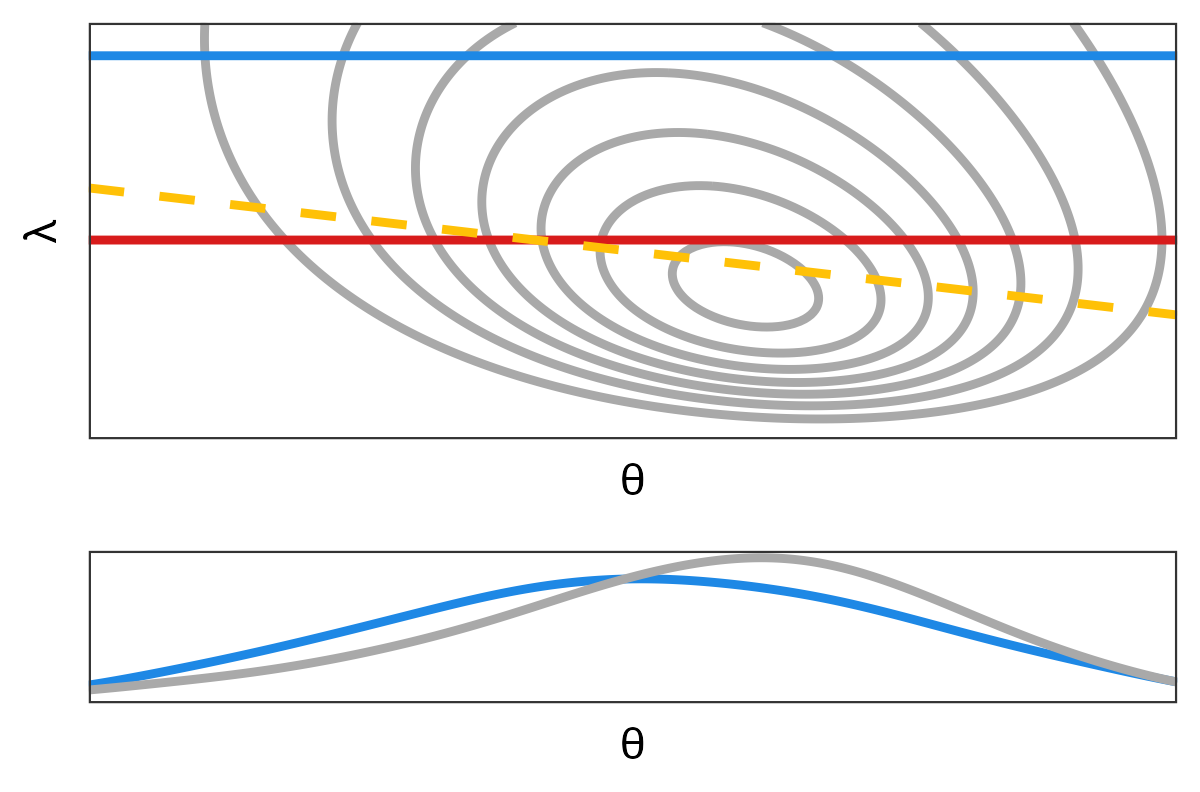} & $\textcolor{sigcolor}{\boldsymbol{\xrightarrow{-\daddid}}}$& \includegraphics[scale = 0.42]{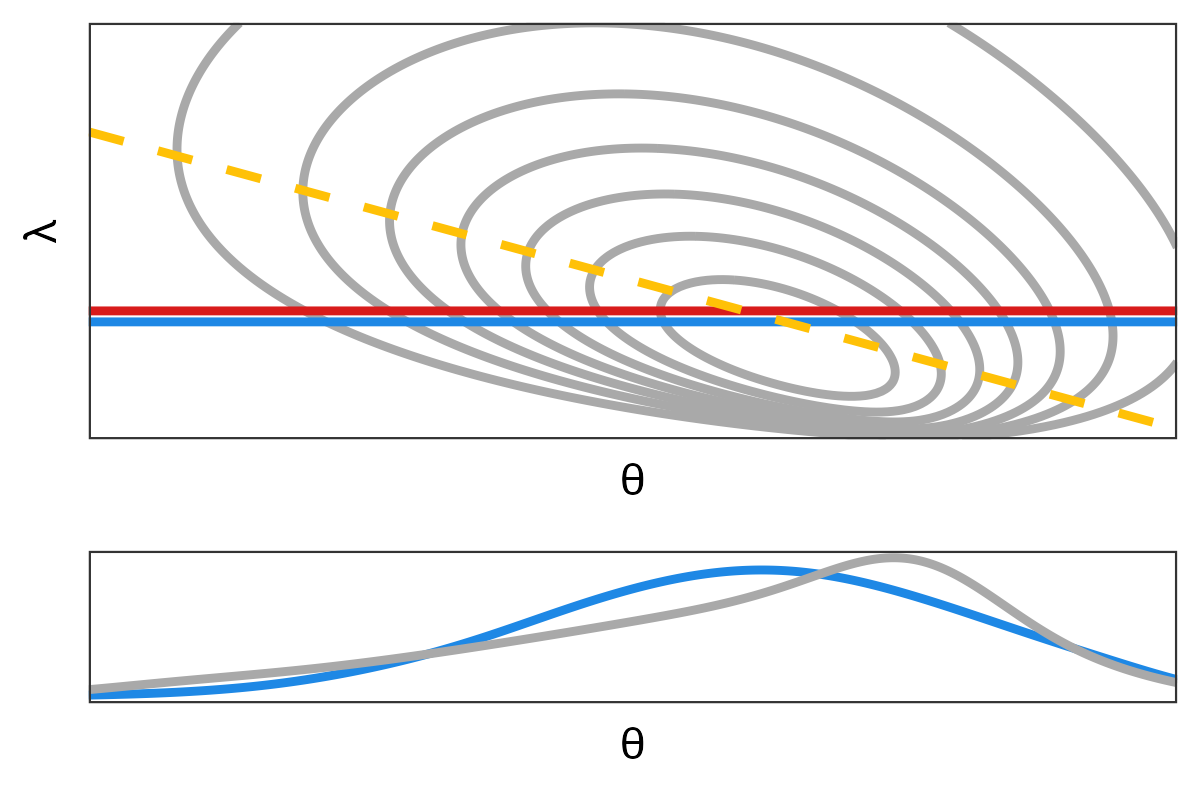}\\
    \vspace{0.1in}\hspace{0.9in}$\textcolor{delcolor}{\boldsymbol{\downarrow -\ddelid}}$\vspace{0.1in} && \vspace{0.1in}\hspace{0.9in}$\textcolor{delcolor}{\boldsymbol{\downarrow -\ddelid}}$\vspace{0.1in}\\
    \includegraphics[scale = 0.42]{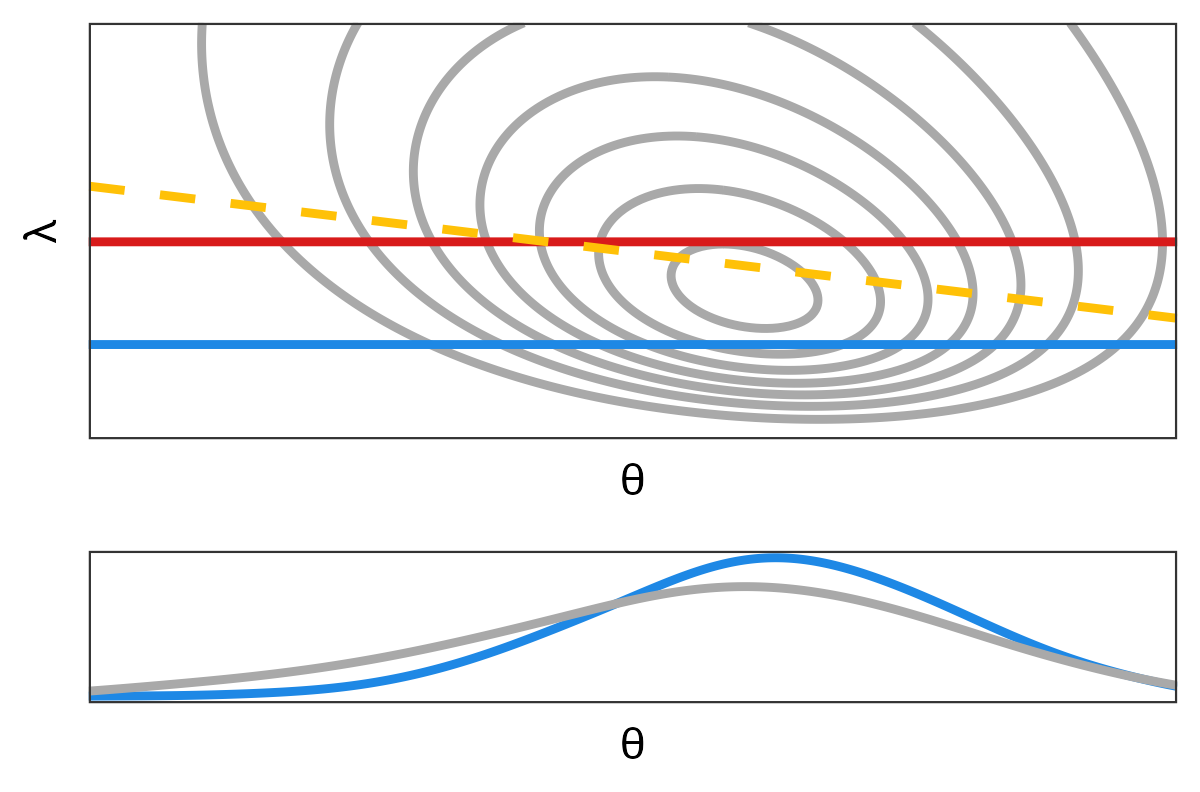} & $\textcolor{sigcolor}{\boldsymbol{\xrightarrow{-\daddid}}}$& \includegraphics[scale = 0.42]{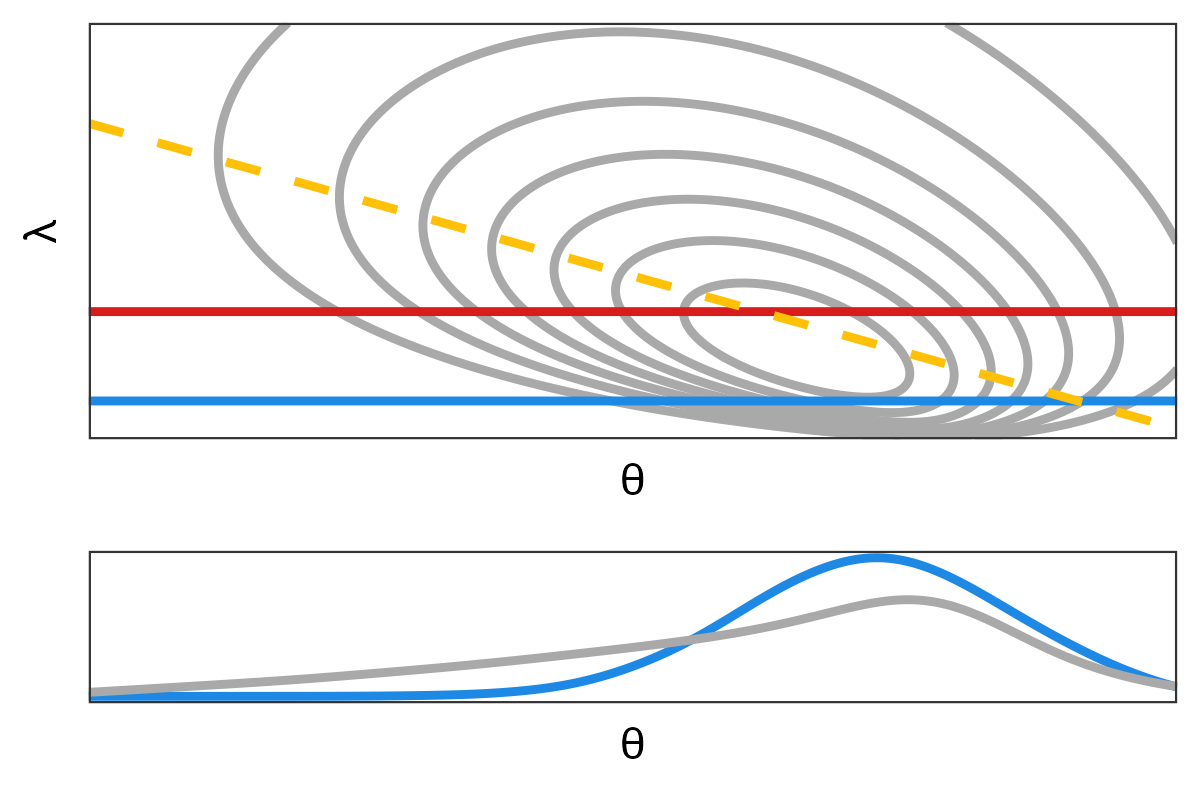}\\
  \end{tabular}
  \caption{Visualization of the effect of the terms $\ddelid$ and $\daddid$ defined in Lemma \ref{lem:mi_decomp} using the base and expanded models defined in the unknown variance example presented in Section 4.2.}
  \label{fig:imi_decomp}
\end{figure}

Figure \ref{fig:imi_decomp} shows the effect of varying the base and expanded models on the identifiability mutual information using its decomposition from Lemma \ref{lem:mi_decomp}. Each of the four panels corresponds to a different choice of base and expanded model (i.e. to a different choice of hyperparameters in (56) and (57)). In each panel, the top plot displays the posterior $\p(\thetavc,\lambdavc\mid \yvc)$ for a typical value of $\yvc$ from the prior predictive distribution $\p(\yvc)$. The bottom plot in each panel displays the marginal posterior distribution of $\thetav$ for both the base (blue) and expanded model (grey).

The three lines in top plots roughly illustrate the quantities $\ddelid$ and $\daddid$ defined in Lemma \ref{lem:mi_decomp}. The quantity $\ddelid$ measures how much better or worse $\thetav$ is identified in the expanded model conditional on $\lambdav$. We can thus get a rough sense of $\ddelid$ by comparing $\p(\thetavc\mid \yvc, \lambdavc_{0}) = \p_{\b}(\thetavc\mid \yvc)$ (indicated by the blue horizontal line) to $\p(\thetavc\mid \yvc, \lambdavc_{(0.5)})$ where $\lambdavc_{(0.5)}$ is the median of $\p(\lambdavc\mid \yvc)$ (indicated by the red horizontal line). In all four panels, the identification of $\thetav$ conditional on $\lambdav = \lambdavc$ decreases as $\lambdavc$ increases, as seen by the fact that the contour plots grow wider as $\lambdavc$ increases. Thus, we expect $\ddelid$ to be more negative (positive) the farther the red line is above (below) the blue line. As we pass from the top two panels to the bottom two panels, $\ddelid$ decreases (and $\daddid$ is unchanged), and in both cases we find the red line further above the blue line, as expected. The negative effect on overall identification is also confirmed by comparing the marginal posteriors of $\thetav$ in the bottom plots passing again from the top row to the bottom row.

The quantity $\daddid$, on the other hand, measures the posterior dependence between $\thetav$ and $\lambdav$. As $\daddid$ becomes more negative, $\thetav$ and $\lambdav$ become more dependent, decreasing the marginal identification of $\thetav$.
This dependence is indicated roughly by a regression of $\lambdav$ on $\thetav$, which is indicated by the yellow dashed line. Passing from the left column to the right column, $\daddid$ increases (while $\ddelid$ is unchanged), which is confirmed by the steeper slopes of the regression lines. And again the overall negative effect on the marginal identification of $\thetav$ is seen in the bottom plots, with the expanded posterior distributions becoming more spread out relative to the posterior distribution in the base model.

\begin{figure}[H]
  \centering
  \begin{tabular}{m{1.8in} m{0.5in} m{1.8in}}
    \includegraphics[scale = 0.42]{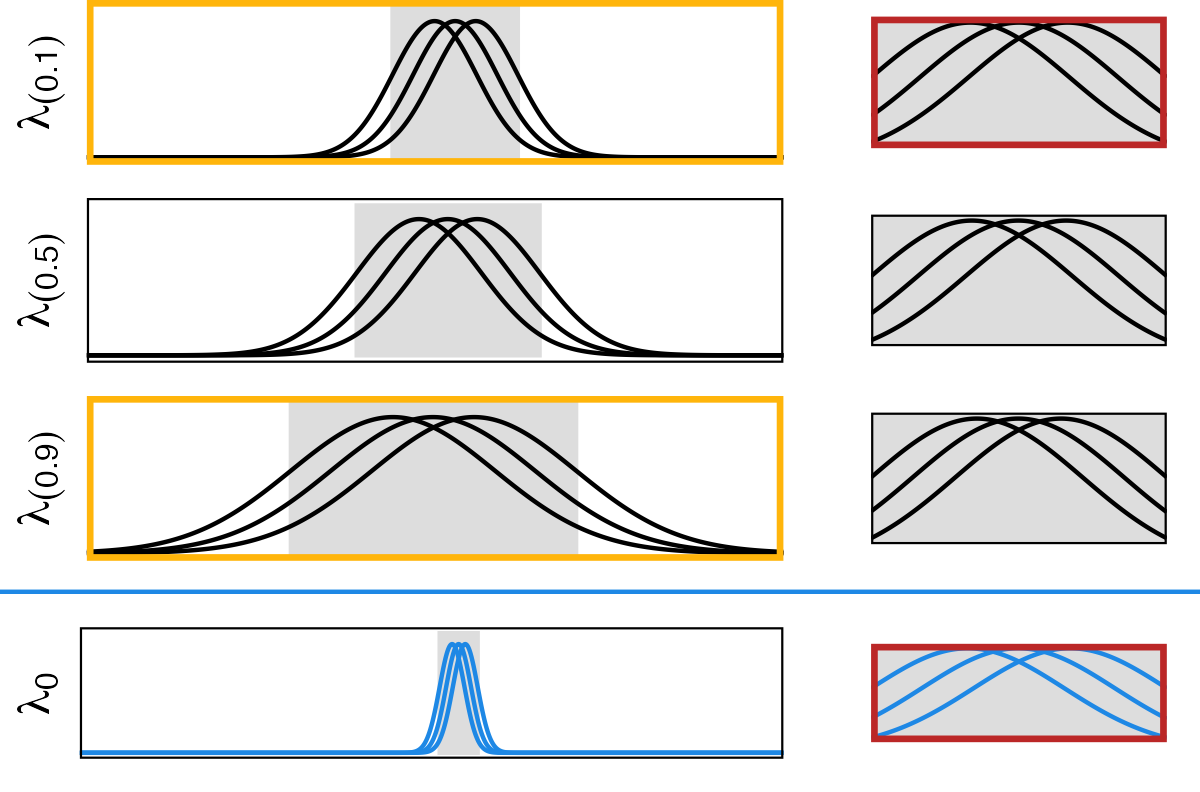} & $\textcolor{sigcolor}{\boldsymbol{\xrightarrow{-\daddfa}}}$& \includegraphics[scale = 0.42]{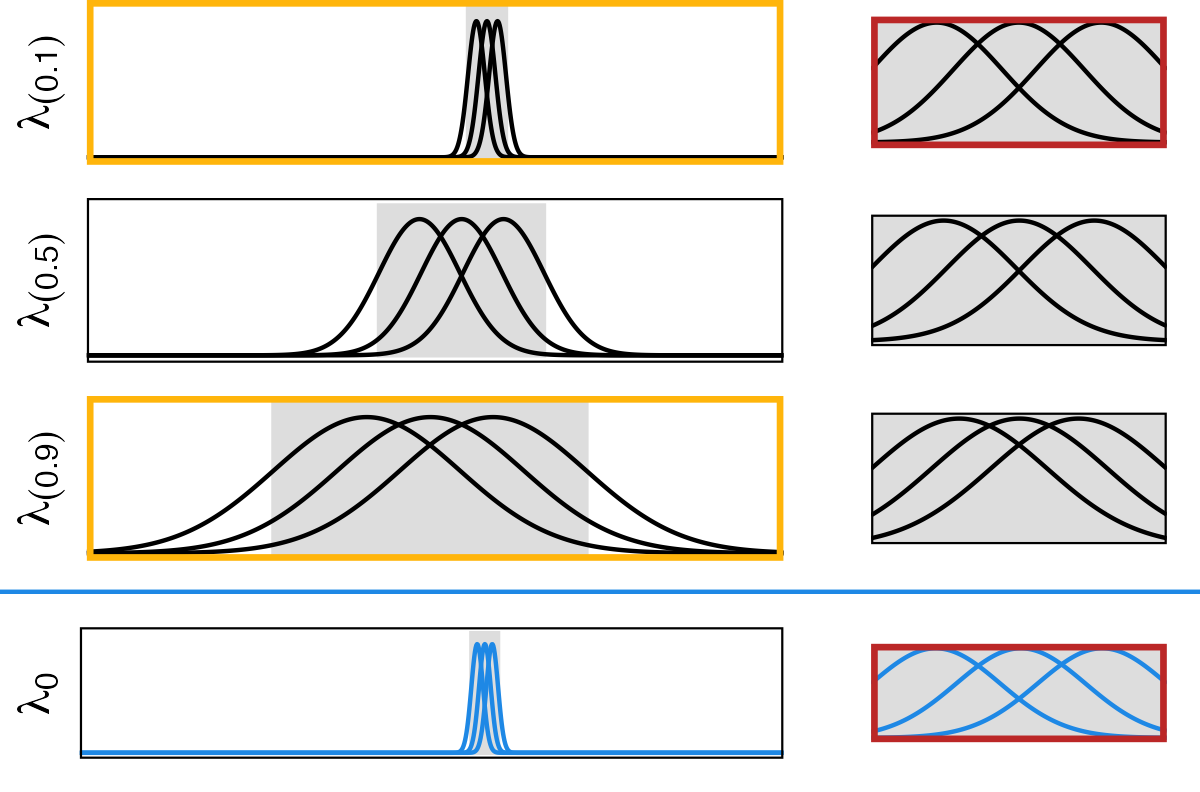}\\
    \vspace{0.1in}\hspace{0.9in}$\textcolor{delcolor}{\boldsymbol{\downarrow -\ddelfa}}$\vspace{0.1in} && \vspace{0.1in}\hspace{0.9in}$\textcolor{delcolor}{\boldsymbol{\downarrow -\ddelfa}}$\vspace{0.1in}\\
    \includegraphics[scale = 0.42]{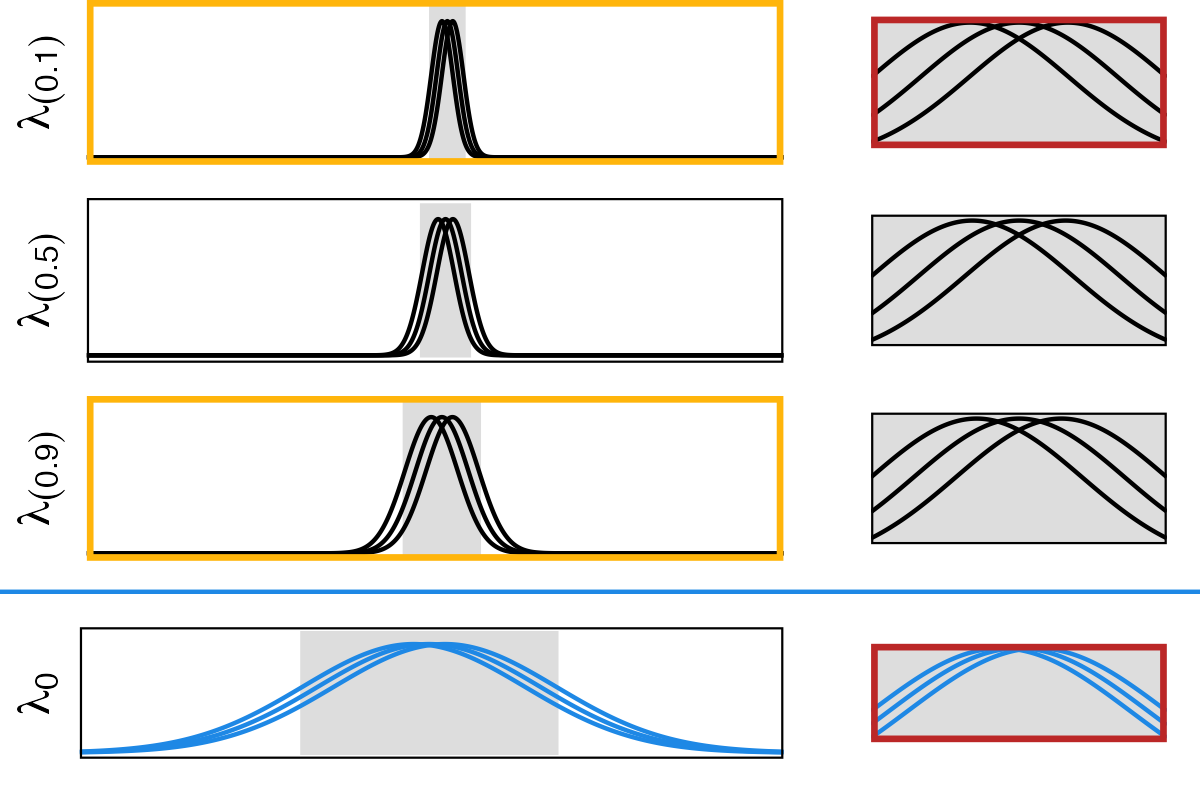} & $\textcolor{sigcolor}{\boldsymbol{\xrightarrow{-\daddfa}}}$& \includegraphics[scale = 0.42]{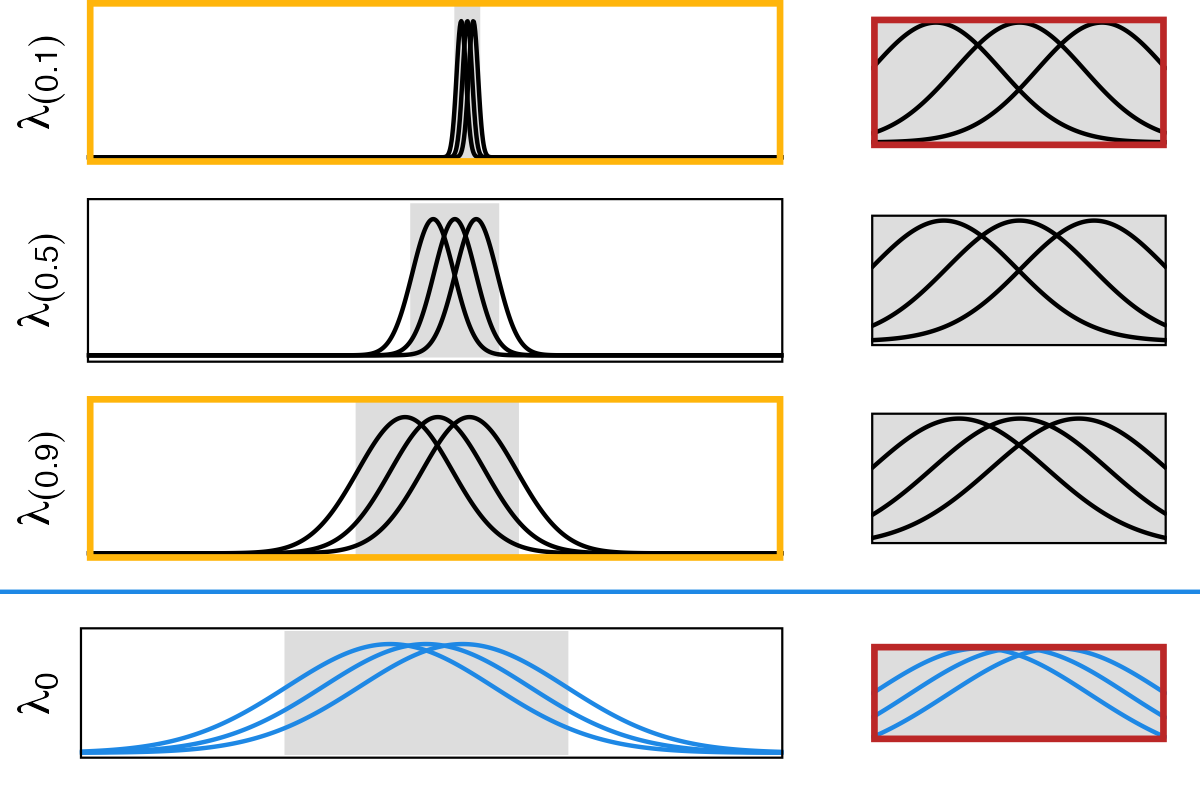}\\
  \end{tabular}
  \caption{Visualization of the effect of terms $\ddelfa$ and $\daddfa$ defined in Lemma \ref{lem:fmi_decomp} using the base and expanded models from the unknown variance example in Section 4.2.
}
  \label{fig:fmi_decomp}
\end{figure}

Figure \ref{fig:fmi_decomp} shows the effect of varying the base and expanded models on the identifiability mutual information using its decomposition from Lemma \ref{lem:fmi_decomp}.
Each of the four panels represents a different pair of base and expanded models $\p_{\b}$ and $\p$ (i.e. a different choice of hyperparameters in (56) and (57)). In each panel, each plot corresponds to a different choice of $\lambda$ (indicated to the left of each plot). For the top three plots (above the blue horizontal line), we select three posterior quantiles of $\lambdav$ under the expanded model. Specifically, we take $\lambda = \lambda_{(q)}$ for $q = 0.1, 0.5, 0.9$ (top to bottom), where $\lambda_{(q)}$ is the $q$ quantile of $p(\lambda \mid y)$, and where $y$ is taken to be a typical value under the prior predictive $\p(\yvc)$. The bottom plot (in blue) corresponds to $\lambda = \lambda_0$ (i.e. to the base model).

For each such $\lambda$, we plot (in the left, wider subplot) the sampling distributions $\p(y\mid \theta, \lambda)$ for three values of $\thetavc$ which represent the conditional posterior $\p(\theta\mid y,\lambda)$. Specifically, for $p=0.1, 0.5, 0.9$, we take $\theta_{(p)\mid\lambdavc}$, where denotes the $p$ quantile of $\p(\theta\mid y,\lambda)$. The regions highlighted in grey correspond to one standard deviation below and above the smallest and largest values of $\thetavc$ (i.e. to the intervals $(\thetavc_{(0.1)\mid\lambdavc} - \sqrt{\lambdavc}, \thetavc_{(0.9)\mid\lambdavc} + \sqrt{\lambdavc})$). These grey regions are zoomed in on in the right subplots to emphasize the relative (dis)similarity of the plotted distributions.

The quantity $\ddelfa$ from Lemma \ref{lem:fmi_decomp} represents the change from base to expanded model in the variability of the sampling distributions $\p(\yvc\mid\thetavc,\lambdavc)$ due to variation in $\thetavc$ (conditional on $\lambdavc$). This is represented in the figure by the dissimilarity between the three sampling distributions in each subplot. In particular, passing from the top row to the bottom row, $\ddelfa$ decreases (while $\daddfa$ is unchanged). This is seen in a comparison of the zoomed-in plots between base and expanded model. In the top row, conditional on the different $\lambdavc$, the sampling distributions show a similar level of dissimilarity between base and expanded models. But in the bottom row, the zoomed-in plots show that the conditional sampling distributions are much more varied in the expanded model than in the base model.

The quantity $\daddfa$ represents how much the sampling distributions $\p(\yvc\mid\lambdavc)$ vary in the expanded model as $\lambdavc$ varies (with $\thetavc$ now marginalized out). This is visible in the figure as the dissimilarity between the sampling distributions across different values of $\lambdavc$ (e.g. by comparing the first and third plots, highlighted in yellow, corresponding to $\lambdavc_{(0.1)}$ and $\lambdavc_{(0.9)}$).
Passing from the left column to the right column, $\daddfa$ decreases, but $\ddelfa$ is unchanged. This is seen in the greater variability of the sampling distribution scale between $\lambdavc_{(0.1)}$ and $\lambdavc_{(0.9)}$ in the right column relative to the left column. On the whole, passing from the top-left panel to the bottom-right panel, the overall effect is an increase in the diversity of sampling distributions under the posterior distribution.

\subsection{Example 3: Hierarchical Prior}
For this example, some important expressions can be computed analytically. This section presents the details of those computations.

\subsubsection{Partial Expression for \texorpdfstring{$\imim$}{IMI}}

We first derive an expression for $\imim_{\b}$. Using the conjugate form of the prior and likelihood, we know that the posterior distribution of $\thetav$ is normal with precision matrix $\left(\sigma_0^{-2} + n\right)\imat_2$. Consequently, we can write
\begin{equation}
  \label{eq:counter_base_imim}
  \imim_{\b} = h(\thetav) - h\left(\thetav\mid\yv\right) = \log\left(\frac{\sigma_0^{-2} + n}{\sigma_0^{-2}}\right) = \log\left(1 + n\sigma_0^2\right).
\end{equation}
Next, for the expanded model, we decompose the $\imim$ as follows:
\begin{equation}
  \label{eq:counter_imim_decomp}
  \imim = \MI\left(\thetav,\yv\right) = \MI\left((\thetav,\lambdav),\yv\right) = \MI\left(\thetav,\yv\mid\lambdav\right) + \MI\left(\lambdav,\yv\right),
\end{equation}
where the second equality follows from the fact that $\yv$ and $\lambdav$ are indpendent given $\thetav$. 

The second term on the right-hand side $\MI(\lambdav,\yv)$ cannot be evaluated analytically. Instead, we express this as $\MI(\lambdav,\yv) = h(\lambdav) - h(\lambdav\mid\yv)$. We use the closed-form formula for the entropy of a Beta distribution to evaluate $h(\lambdav)$ and estimate $h(\lambdav\mid\yv)$ using posterior samples of $\lambdav$ (averaged over many $\yv\sim p(\yvc)$). We note that estimating the entropy is tractable in this case since $\lambdav$ is one-dimensional and compactly supported.

The first term on the right-hand side of \eqref{eq:counter_imim_decomp} can be expressed as
\begin{align}
  \MI\left(\thetav,\yv\mid\lambdav\right) &= h\left(\thetav\mid\lambdav\right) - h\left(\thetav\mid\yv,\lambdav\right)\nonumber\\
  &= \E\left[h_{p(\thetavc\mid\lambdav)}\left(\thetav\right) - h_{p(\thetavc\mid\yv,\lambdav)}\left(\thetav\right)\right].\label{eq:counter_imim_decomp2}
\end{align}

Recalling that the prior covariance of $\thetav$ given $\lambdav=\lambdavc$ is 
\begin{equation}
  \label{eq:prior_cov_recall}
  \Sigma_0(\lambdavc, \sigma_0) = \frac{\sigma_0^2}{\sqrt{1-\lambdavc^2}}\left[
         \begin{matrix}
           1 & \lambdavc\\
           \lambdavc & 1
         \end{matrix}
       \right],
\end{equation}
it is easy to see that $\det\left(\Sigma_0(\lambdavc,\sigma_0)\right) = \sigma_0^4$ for all $\lambdavc\in [0,1)$, and thus
\begin{equation}
  \label{eq:counter_pr_ent}
  \E\left[ h_{\p(\thetavc\mid\lambdav)}\left(\thetav\right) \right] = \frac{1}{2}\log(2\pi e) + \frac{1}{2}\log\left(\sigma_0^4\right).
\end{equation}

Next we note that, conditional on $\lambdav$, the expanded model for $\thetav$ has a normal-normal conjugate form. Thus, the posterior precision matrix is given as 
\begin{equation}
  \label{eq:counter_post_prec}
  \Pi = \Sigma_0^{-1}\left(\lambdavc,\sigma_0\right) + n\imat_2.
\end{equation}
Using the standard inversion formula for $2\times 2$ matrices, we get
\begin{equation}
  \label{eq:counter_prior_prec}
  \Sigma^{-1}_0(\lambdavc, \sigma_0) = \frac{1}{\sigma_0^2\sqrt{1-\lambdavc^2}}\left[
         \begin{matrix}
           1 & -\lambdavc\\
           -\lambdavc & 1
         \end{matrix}
       \right].
\end{equation}
Combining \eqref{eq:counter_pr_ent} and \eqref{eq:counter_prior_prec}, we see that
\begin{align}
  \det\left(\Pi\right) &= \frac{1}{\sigma_0^4(1 - \lambdavc^2)}\det\left(\left[
         \begin{matrix}
           1 + n\sigma_0^2\sqrt{1 - \lambdavc^2} & -\lambdavc\\
           -\lambdavc & 1 + n\sigma_0^2\sqrt{1 - \lambdavc^2}
         \end{matrix}
       \right]\right)\nonumber\\
       &= \frac{(1 + n\sigma_0^2\sqrt{1 - \lambdavc^2})^2 - \lambdavc^2}{\sigma_0^4(1 - \lambdavc^2)}. \label{eq:counter_prec_det}
\end{align}
Now, combining \eqref{eq:counter_prec_det}, \eqref{eq:counter_pr_ent}, and \eqref{eq:counter_imim_decomp2}, we get that
\begin{equation}
  \label{eq:counter_imim_2_final}
  \MI\left(\thetav,\yv\mid\lambdav\right) = \frac{1}{2}\int\log\left(\frac{(1 + n\sigma_0^2\sqrt{1 - \lambdavc^2})^2 - \lambdavc^2}{1 - \lambdavc^2}\right)\p(\lambdavc)d\lambdavc.
\end{equation}
For any $n$ and $\sigma_0$, the integrand in \eqref{eq:counter_imim_2_final} is increasing in $\lambdavc$, and at $\lambdavc = 0$ it reduces to
\begin{equation}
  \label{eq:counter_imim_base_case}
  \log\left(1 + n\sigma_0^2\right) = \imim_{\b}.
\end{equation}
In light of \eqref{eq:counter_imim_decomp}, this proves that $\imim \geq \imim_{\b}$ for all priors $\p(\lambdavc)$.

\subsubsection{Partial expression for \texorpdfstring{$\fmim$}{FMI}}

First we derive an expression for $\fmim_{\b}$. We note that the posterior precision matrix for $\thetav$ in the base model given $(\yv,\yrep)$ is just $(\sigma_0^{-2} + 2n)\imat_2$. Thus, we obtain
\begin{equation}
  \label{eq:counter_fmim_base}
  -\fmim_{\b} = h\left(\thetav\mid \yv\right) - h\left(\thetav\mid\yv,\yrep\right) = \log\left(\frac{\sigma_0^{-2} + 2n}{\sigma_0^{-2} + n}\right) = \log\frac{1 + 2n\sigma_0^2}{1 + n\sigma_0^2}.
\end{equation}

Next, we decompose $\fmim$ in the expanded model just as we did in the previous section, writing
\begin{equation}
  \label{eq:counter_fmim_decomp}
  -\fmim = \MI\left((\thetav,\lambdav), \yrep\mid \yv\right) = \MI\left(\thetav, \yrep\mid \yv,\lambdav\right) + \MI\left(\lambdav, \yrep\mid\yv\right).
\end{equation}
As before, the second term on the right-hand side of \eqref{eq:counter_fmim_decomp} must be estimated (by taking the difference of estimates of the posterior entropy of $\lambdav$ given $\yv$ and given $(\yv,\yrep)$).

The first term on the right-hand side of \eqref{eq:counter_fmim_decomp} can be written as 
\begin{equation}
  \label{eq:counter_fmim_decomp_2}
  \MI\left(\thetav, \yrep\mid \yv,\lambdav\right) = \E\left[h_{\p(\thetavc\mid\yv,\lambdav)}(\thetav) - h_{\p(\thetavc\mid\yv,\yrep,\lambdav)}(\thetav)\right].
\end{equation}
The first entropy was determined in the previous section, and the second entropy is derived in exactly the same way except that all occurrences of $n$ are replaced by $2n$. This leads directly the following expression:
\begin{equation}
  \label{eq:counter_fmim_2_final}
  \MI\left(\thetav, \yrep\mid \yv,\lambdav\right) = \frac{1}{2}\int \log\left(\frac{\left(1 + 2n\sigma_0^2\sqrt{1 - \lambdavc^2}\right)^2 - \lambdavc^2}{\left(1 + n\sigma_0^2\sqrt{1 - \lambdavc^2}\right)^2 - \lambdavc^2}\right)\p(\lambdavc) d\lambdavc.
\end{equation}
In this case, for $n\sigma_0^2$ sufficiently large ($n\sigma_0^2> 1$ is large enough), the integrand is increasing and maximized at $\lambdavc = 0$, where it reduces to
\[
  \log\left(\frac{1 + 2n\sigma_0^2}{1 + n\sigma_0^2}\right) = -\fmim_{\b}.
\]

\end{document}